\documentclass[9pt]{amsart}
\usepackage{amsmath}
\usepackage{amsxtra}
\usepackage{amscd}
\usepackage{amsthm}
\usepackage{amsfonts}
\usepackage{amssymb}
\usepackage{eucal}
\usepackage[all]{xy}
\usepackage{graphicx}

\newtheorem{cor}[subsubsection]{Corollary}
\newtheorem{lem}[subsubsection]{Lemma}
\newtheorem{prop}[subsubsection]{Proposition}

\newtheorem{thm}[subsubsection]{Theorem}

\theoremstyle{definition}

\theoremstyle{remark}
\newtheorem{rem}[subsubsection]{Remark}

\newcommand{\thmref}[1]{Theorem~\ref{#1}}

\newcommand{\secref}[1]{Sect.~\ref{#1}}
\newcommand{\lemref}[1]{Lemma~\ref{#1}}
\newcommand{\propref}[1]{Proposition~\ref{#1}}
\newcommand{\corref}[1]{Corollary~\ref{#1}}

\numberwithin{equation}{section}

\newcommand{\nc}{\newcommand}
\nc{\renc}{\renewcommand}
\nc{\ssec}{\subsection}
\nc{\sssec}{\subsubsection}
\nc{\on}{\operatorname}

\nc\ol{\overline}
\nc\wt{\widetilde}
\nc\tboxtimes{\wt{\boxtimes}}
\nc\tstar{\wt{\star}}
\nc{\alp}{a}

\nc{\ZZ}{{\mathbb Z}}
\nc{\NN}{{\mathbb N}}
\nc{\OO}{{\mathbb O}}
\renc{\SS}{{\mathbb S}}
\nc{\DD}{{\mathbb D}}
\nc{\GG}{{\mathbb G}}

\nc{\Fq}{{\mathbb F}_q}
\nc{\Fqb}{\ol{{\mathbb F}_q}}
\nc{\Ql}{\ol{{\mathbb Q}_\ell}}
\nc{\id}{\text{id}}
\nc\X{\mathcal X}

\nc{\Hom}{\on{Hom}}
\nc{\Lie}{\on{Lie}}
\nc{\Loc}{\on{Loc}}
\nc{\Pic}{\on{Pic}}
\nc{\Bun}{\on{Bun}}
\nc{\IC}{\on{IC}}
\nc{\Fls}{\on{Fl}^{\frac{\infty}{2}}}
\nc{\ICs}{\on{IC}^{\frac{\infty}{2}}}
\nc{\ICsl}{\on{IC}^{\lambda+\frac{\infty}{2}}}
\nc{\ICslm}{\on{IC}^{\lambda+\frac{\infty}{2},-}}
\nc{\ICsm}{\on{IC}^{\frac{\infty}{2},-}}
\nc{\Aut}{\on{Aut}}
\nc{\rk}{\on{rk}}
\nc{\Sh}{\on{Sh}}
\nc{\Perv}{\on{Perv}}
\nc{\pos}{{\on{pos}}}
\nc{\Conv}{\on{Conv}}
\nc{\Sph}{\on{Sph}}
\nc{\Sat}{\on{Sat}}
\nc{\Sym}{\on{Sym}}
\nc{\BunBb}{\overline{\Bun}_B}
\nc{\BunNb}{\overline{\Bun}_N}
\nc{\BunTb}{\overline{\Bun}_T}
\nc{\BunBbm}{\overline{\Bun}_{B^-}}
\nc{\BunBbel}{\overline{\Bun}_{B,el}}
\nc{\BunBbmel}{\overline{\Bun}_{B^-,el}}
\nc{\Buno}{\overset{o}{\Bun}}
\nc{\BunPb}{{\overline{\Bun}_P}}
\nc{\BunBM}{\Bun_{B(M)}}
\nc{\BunBMb}{\overline{\Bun}_{B(M)}}
\nc{\BunPbw}{{\widetilde{\Bun}_P}}
\nc{\BunBP}{\widetilde{\Bun}_{B,P}}
\nc{\GUb}{\overline{G/U}}
\nc{\GUPb}{\overline{G/U(P)}}

\nc{\Hhom}{\underline{\on{Hom}}}
\nc\syminfty{\on{Sym}^{\infty}}
\nc\lal{\ol{\kappa_x}}
\nc\xl{\ol{x}}
\nc\thl{\ol{\theta}}
\nc\nul{\ol{\nu}}
\nc\mul{\ol{\mu}}
\nc\Sum\Sigma
\nc{\oX}{\overset{o}{X}{}}
\nc{\hl}{\overset{\leftarrow}h{}}
\nc{\hr}{\overset{\rightarrow}h{}}
\nc{\M}{{\mathcal M}}
\nc{\N}{{\mathcal N}}
\nc{\F}{{\mathcal F}}
\nc{\D}{{\mathcal D}}
\nc{\Q}{{\mathcal Q}}
\nc{\Y}{{\mathcal Y}}
\nc{\G}{{\mathcal G}}
\nc{\E}{{\mathcal E}}
\nc{\CalC}{{\mathcal C}}
\nc\Dh{\widehat{\D}}

\nc{\C}{{\mathcal C}}
\nc{\K}{{\mathcal K}}
\renewcommand{\H}{{\mathcal H}}

\nc{\T}{{\mathcal T}}
\nc{\V}{{\mathcal V}}
\renc{\P}{{\mathcal P}}
\nc{\A}{{\mathcal A}}
\nc{\B}{{\mathcal B}}
\nc{\U}{{\mathcal U}}

\nc{\Gr}{{\on{Gr}}}

\nc{\frn}{{\check{\mathfrak u}(P)}}

\nc{\fC}{\mathfrak C}
\nc{\p}{\mathfrak p}
\nc{\q}{\mathfrak q}
\nc\f{{\mathfrak f}}

\nc{\qo}{{\mathfrak q}}
\nc{\po}{{\mathfrak p}}
\nc{\s}{{\mathfrak s}}
\nc\w{\text{w}}

\renewcommand{\mod}{{\on{-mod}}}

\nc\mathi\iota
\nc\Spec{\on{Spec}}
\nc\Mod{\on{Mod}}
\nc{\tw}{\widetilde{\mathfrak t}}
\nc{\pw}{\widetilde{\mathfrak p}}
\nc{\qw}{\widetilde{\mathfrak q}}
\nc{\jw}{\widetilde j}

\nc{\grb}{\overline{\Gr}}
\nc{\I}{\mathcal I}

\nc{\kappach}{{\check\kappa_x}}
\nc{\Lambdach}{{\check\Lambda}{}}
\nc{\much}{{\check\mu}}
\nc{\omegach}{{\check\omega}}
\nc{\nuch}{{\check\nu}}
\nc{\etach}{{\check\eta}}
\nc{\alphach}{{\checka}}
\nc{\oblvtach}{{\check\oblvta}}
\nc{\pich}{{\check\pi}}
\nc{\ch}{{\check h}}

\nc{\Hb}{\overline{\H}}

\nc{\BA}{{\mathbb{A}}}
\nc{\BC}{{\mathbb{C}}}
\nc{\BE}{{\mathbb{E}}}
\nc{\BF}{{\mathbb{F}}}
\nc{\BH}{{\mathbb{H}}}
\nc{\BG}{{\mathbb{G}}}
\nc{\BM}{{\mathbb{M}}}
\nc{\BO}{{\mathbb{O}}}
\nc{\BD}{{\mathbb{D}}}
\nc{\BN}{{\mathbb{N}}}
\nc{\BP}{{\mathbb{P}}}
\nc{\BQ}{{\mathbb{Q}}}
\nc{\BR}{{\mathbb{R}}}
\nc{\BZ}{{\mathbb{Z}}}
\nc{\BS}{{\mathbb{S}}}

\nc{\CA}{{\mathcal{A}}}
\nc{\CB}{{\mathcal{B}}}

\nc{\CE}{{\mathcal{E}}}
\nc{\CF}{{\mathcal{F}}}
\nc{\CG}{{\mathcal{G}}}
\nc{\CH}{{\mathcal{H}}}

\nc{\CL}{{\mathcal{L}}}
\nc{\CC}{{\mathcal{C}}}
\nc{\CM}{{\mathcal{M}}}
\nc{\CN}{{\mathcal{N}}}
\nc{\cCN}{\check{{\mathcal{N}}}}
\nc{\CK}{{\mathcal{K}}}
\nc{\CO}{{\mathcal{O}}}
\nc{\CP}{{\mathcal{P}}}
\nc{\CQ}{{\mathcal{Q}}}
\nc{\CR}{{\mathcal{R}}}
\nc{\CS}{{\mathcal{S}}}
\nc{\CT}{{\mathcal{T}}}
\nc{\CU}{{\mathcal{U}}}
\nc{\CV}{{\mathcal{V}}}
\nc{\CW}{{\mathcal{W}}}
\nc{\CX}{{\mathcal{X}}}
\nc{\CY}{{\mathcal{Y}}}
\nc{\CZ}{{\mathcal{Z}}}
\nc{\CI}{{\mathcal{I}}}
\nc{\CJ}{{\mathcal{J}}}

\nc{\csM}{{\check{\mathcal A}}{}}
\nc{\oM}{{\overset{\circ}{\mathcal M}}{}}
\nc{\obM}{{\overset{\circ}{\mathbf M}}{}}
\nc{\oCA}{{\overset{\circ}{\mathcal A}}{}}
\nc{\obA}{{\overset{\circ}{\mathbf A}}{}}
\nc{\ooM}{{\overset{\circ}{M}}{}}
\nc{\osM}{{\overset{\circ}{\mathsf M}}{}}
\nc{\vM}{{\overset{\bullet}{\mathcal M}}{}}
\nc{\nM}{{\underset{\bullet}{\mathcal M}}{}}
\nc{\oD}{{\overset{\circ}{\mathcal D}}{}}
\nc{\obD}{{\overset{\circ}{\mathbf D}}{}}
\nc{\oA}{{\overset{\circ}{\mathbb A}}{}}
\nc{\op}{{\overset{\bullet}{\mathbf p}}{}}
\nc{\cp}{{\overset{\circ}{\mathbf p}}{}}
\nc{\oU}{{\overset{\bullet}{\mathcal U}}{}}
\nc{\oZ}{{\overset{\circ}{\mathcal Z}}{}}
\nc{\ofZ}{{\overset{\circ}{\mathfrak Z}}{}}
\nc{\oF}{{\overset{\circ}{\fF}}}

\nc{\fa}{{\mathfrak{a}}}
\nc{\fb}{{\mathfrak{b}}}
\nc{\fd}{{\mathfrak{d}}}
\nc{\ff}{{\mathfrak{f}}}
\nc{\fg}{{\mathfrak{g}}}
\nc{\fgl}{{\mathfrak{gl}}}
\nc{\fh}{{\mathfrak{h}}}
\nc{\fj}{{\mathfrak{j}}}
\nc{\fl}{{\mathfrak{l}}}
\nc{\fm}{{\mathfrak{m}}}
\nc{\fn}{{\mathfrak{n}}}
\nc{\fu}{{\mathfrak{u}}}
\nc{\fp}{{\mathfrak{p}}}
\nc{\fr}{{\mathfrak{r}}}
\nc{\fs}{{\mathfrak{s}}}
\nc{\ft}{{\mathfrak{t}}}
\nc{\fz}{{\mathfrak{z}}}
\nc{\fsl}{{\mathfrak{sl}}}
\nc{\hsl}{{\widehat{\mathfrak{sl}}}}
\nc{\hgl}{{\widehat{\mathfrak{gl}}}}
\nc{\hg}{{\widehat{\mathfrak{g}}}}
\nc{\chg}{{\widehat{\mathfrak{g}}}{}^\vee}
\nc{\hn}{{\widehat{\mathfrak{n}}}}
\nc{\chn}{{\widehat{\mathfrak{n}}}{}^\vee}

\nc{\fA}{{\mathfrak{A}}}
\nc{\fB}{{\mathfrak{B}}}
\nc{\fD}{{\mathfrak{D}}}
\nc{\fE}{{\mathfrak{E}}}
\nc{\fF}{{\mathfrak{F}}}
\nc{\fG}{{\mathfrak{G}}}
\nc{\fK}{{\mathfrak{K}}}
\nc{\fL}{{\mathfrak{L}}}
\nc{\fM}{{\mathfrak{M}}}
\nc{\fN}{{\mathfrak{N}}}
\nc{\fP}{{\mathfrak{P}}}
\nc{\fU}{{\mathfrak{U}}}
\nc{\fV}{{\mathfrak{V}}}
\nc{\fZ}{{\mathfrak{Z}}}

\nc{\ba}{{\mathbf{a}}}
\nc{\bb}{{\mathbf{b}}}
\nc{\bc}{{\mathbf{c}}}
\nc{\bd}{{\mathbf{d}}}
\nc{\bbf}{{\mathbf{f}}}
\nc{\be}{{\mathbf{e}}}
\nc{\bi}{{\mathbf{i}}}
\nc{\bj}{{\mathbf{j}}}
\nc{\bn}{{\mathbf{n}}}
\nc{\bo}{{\mathbf{o}}}
\nc{\bp}{{\mathbf{p}}}
\nc{\bq}{{\mathbf{q}}}
\nc{\bu}{{\mathbf{u}}}
\nc{\bv}{{\mathbf{v}}}
\nc{\bx}{{\mathbf{x}}}
\nc{\bs}{{\mathbf{s}}}
\nc{\by}{{\mathbf{y}}}
\nc{\bw}{{\mathbf{w}}}
\nc{\bA}{{\mathbf{A}}}
\nc{\bK}{{\mathbf{K}}}
\nc{\bB}{{\mathbf{B}}}
\nc{\bF}{{\mathbf{F}}}
\nc{\bC}{{\mathbf{C}}}
\nc{\bG}{{\mathbf{G}}}
\nc{\bD}{{\mathbf{D}}}
\nc{\bE}{{\mathbf{E}}}
\nc{\bH}{{\mathbf{H}}}
\nc{\bI}{{\mathbf{I}}}
\nc{\bM}{{\mathbf{M}}}
\nc{\bN}{{\mathbf{N}}}
\nc{\bO}{{\mathbf{O}}}
\nc{\bV}{{\mathbf{V}}}
\nc{\bW}{{\mathbf{W}}}
\nc{\bX}{{\mathbf{X}}}
\nc{\bZ}{{\mathbf{Z}}}
\nc{\bS}{{\mathbf{S}}}

\nc{\sA}{{\mathsf{A}}}
\nc{\sB}{{\mathsf{B}}}
\nc{\sC}{{\mathsf{C}}}
\nc{\sD}{{\mathsf{D}}}
\nc{\sF}{{\mathsf{F}}}
\nc{\sK}{{\mathsf{K}}}
\nc{\sM}{{\mathsf{M}}}
\nc{\sO}{{\mathsf{O}}}
\nc{\sW}{{\mathsf{W}}}
\nc{\sQ}{{\mathsf{Q}}}
\nc{\sP}{{\mathsf{P}}}
\nc{\sZ}{{\mathsf{Z}}}
\nc{\sr}{{\mathsf{r}}}
\nc{\bk}{{\mathsf{k}}}
\nc{\sg}{{\mathsf{g}}}
\nc{\sff}{{\mathsf{f}}}
\nc{\sfe}{{\mathsf{e}}}
\nc{\sfj}{{\mathsf{j}}}
\nc{\sfb}{{\mathsf{b}}}
\nc{\sfc}{{\mathsf{c}}}
\nc{\sd}{{\mathsf{d}}}
\nc{\sv}{{\mathsf{v}}}

\nc{\BK}{{\bar{K}}}

\nc{\tA}{{\widetilde{\mathbf{A}}}}
\nc{\tB}{{\widetilde{\mathcal{B}}}}
\nc{\tg}{{\widetilde{\mathfrak{g}}}}
\nc{\tG}{{\widetilde{G}}}
\nc{\TM}{{\widetilde{\mathbb{M}}}{}}
\nc{\tO}{{\widetilde{\mathsf{O}}}{}}
\nc{\tU}{{\widetilde{\mathfrak{U}}}{}}
\nc{\TZ}{{\tilde{Z}}}
\nc{\tx}{{\tilde{x}}}
\nc{\tbv}{{\tilde{\bv}}}
\nc{\tfP}{{\widetilde{\mathfrak{P}}}{}}
\nc{\tz}{{\tilde{\zeta}}}
\nc{\tmu}{{\tilde{\mu}}}

\nc{\urho}{\underline{\pi}}
\nc{\uB}{\underline{B}}
\nc{\uC}{{\underline{\mathbb{C}}}}
\nc{\ui}{\underline{i}}
\nc{\uj}{\underline{j}}
\nc{\ofP}{{\overline{\mathfrak{P}}}}
\nc{\oB}{{\overline{\mathcal{B}}}}
\nc{\og}{{\overline{\mathfrak{g}}}}
\nc{\oI}{{\overline{I}}}

\nc{\eps}{\varepsilon}
\nc{\hrho}{{\hat{\pi}}}

\nc{\one}{{\mathbf{1}}}
\nc{\two}{{\mathbf{t}}}

\nc{\Rep}{{\mathop{\operatorname{\rm Rep}}}}
\nc{\Tot}{{\mathop{\operatorname{\rm Tot}}}}
\nc{\Ker}{{\mathop{\operatorname{\rm Ker}}}}
\nc{\Hilb}{{\mathop{\operatorname{\rm Hilb}}}}
\nc{\End}{{\mathop{\operatorname{\rm End}}}}
\nc{\Ext}{{\mathop{\operatorname{\rm Ext}}}}
\nc{\CHom}{{\mathop{\operatorname{{\mathcal{H}}\it om}}}}
\nc{\CEnd}{{\mathop{\operatorname{{\mathcal{E}}\it nd}}}}
\nc{\GL}{{\mathop{\operatorname{\rm GL}}}}
\nc{\gr}{{\mathop{\operatorname{\rm gr}}}}
\nc{\Id}{{\mathop{\operatorname{\rm Id}}}}
\nc{\de}{{\mathop{\operatorname{\rm def}}}}
\nc{\length}{{\mathop{\operatorname{\rm length}}}}
\nc{\supp}{{\mathop{\operatorname{\rm supp}}}}

\nc{\Cliff}{{\mathsf{Cliff}}}
\nc{\Fl}{\on{Fl}}
\nc{\Fib}{{\mathsf{Fib}}}
\nc{\Coh}{{\mathsf{Coh}}}
\nc{\QCoh}{{\on{QCoh}}}
\nc{\IndCoh}{{\on{IndCoh}}}
\nc{\FCoh}{{\mathsf{FCoh}}}

\nc{\reg}{{\text{\rm reg}}}

\nc{\cplus}{{\mathbf{C}_+}}
\nc{\cminus}{{\mathbf{C}_-}}
\nc{\cthree}{{\mathbf{C}_*}}
\nc{\Qbar}{{\bar{Q}}}
\nc\Eis{\on{Eis}}
\nc\Eisb{\ol\Eis{}}
\nc\Eisr{\on{Eis}^{rat}{}}
\nc\wh{\widehat}
\nc{\Def}{\on{Def_{\check{\fb}}(E)}}
\nc{\barZ}{\overline{Z}{}}
\nc{\barbarZ}{\overline{\barZ}{}}
\nc{\barpi}{\overline\iota}
\nc{\barbarpi}{\overline\barpi}
\nc{\barpip}{\overline\iota{}^+}
\nc{\barpim}{\overline\iota{}^-}

\nc{\fq}{\mathfrak q}

\nc{\fqb}{\ol{\fq}{}}
\nc{\fpb}{\ol{\fp}{}}
\nc{\fpr}{{\fp^{rat}}{}}
\nc{\fqr}{{\fq^{rat}}{}}

\nc{\hattimes}{\wh\otimes}

\nc{\bh}{{\bar{h}}}
\nc{\bOmega}{{\overline{\Omega(\check \fn)}}}

\nc{\seq}[1]{\stackrel{#1}{\sim}}

\nc{\cT}{{\check{T}}}
\nc{\cG}{{\check{G}}}
\nc{\cM}{{\check{M}}}
\nc{\cB}{{\check{B}}}
\nc{\cN}{{\check{N}}}

\nc{\ct}{{\check{\mathfrak t}}}
\nc{\cg}{{\check{\fg}}}
\nc{\cb}{{\check{\fb}}}
\nc{\cn}{{\check{\fn}}}

\nc{\cLambda}{{\check\Lambda}}

\nc{\cla}{{\check\kappa_x}}
\nc{\cmu}{{\check\mu}}
\nc{\clambda}{{\check\lambda}}
\nc{\cnu}{{\check\nu}}
\nc{\ceta}{{\check\eta}}

\nc{\DefbE}{{\on{Def}_{\cB}(E_\cT)}}

\nc{\imathb}{{\ol{\imath}}}
\nc{\rlr}{\overset{\longrightarrow}{\underset{\longrightarrow}\longleftarrow}}

\nc{\KG}{K\backslash G}
\nc{\comult}{{co\text{-}mult}}
\nc{\counit}{{co\text{-}unit}}
\nc{\uHom}{{\underline{\Maps}}}
\nc{\dgSch}{\on{Sch}}
\nc{\Sch}{\on{Sch}}
\nc{\affdgSch}{\on{Sch}^{\on{aff}}}
\nc{\affSch}{\on{Sch}^{\on{aff}}}
\nc{\Groupoids}{\on{Grpd}}
\nc{\inftygroup}{\on{Spc}}
\nc{\inftyCat}{\infty\on{-Cat}}
\nc{\StinftyCat}{\inftyCat^{\on{St}}}
\nc{\MoninftyCat}{\infty\on{-Cat}^{\on{Mon}}}
\nc{\SymMoninftyCat}{\infty\on{-Cat}^{\on{SymMon}}}
\nc{\SymMonStinftyCat}{\on{DGCat}^{\on{SymMon}}}
\nc{\MonStinftyCat}{\on{DGCat}^{\on{Mon}}}
\nc{\inftystack}{\on{Stk}}
\nc{\inftystackalg}{Stk^{1\text{-}alg}}
\nc{\inftyprestack}{\on{PreStk}}
\nc{\inftydgnearstack}{\on{NearStk}}
\nc{\inftydgstack}{\on{Stk}}
\nc{\inftydgstackalg}{DGStk^{1\text{-}alg}}
\nc{\inftydgprestack}{\on{PreStk}}
\nc{\dgindSch}{\on{indSch}}
\nc{\indSch}{{}^{\on{cl}}\!\on{indSch}}
\nc{\infSch}{\on{infSch}}
\nc{\dr}{{\on{dR}}}

\nc{\mmod}{{\on{-}\!{\mathbf{mod}}}}

\nc{\starr}{\text{\dh}}
\nc{\Spectra}{\on{Spectra}}
\nc{\Crys}{\on{Crys}}
\nc{\oblv}{{\mathbf{oblv}}}
\nc{\ind}{{\mathbf{ind}}}
\nc{\coind}{{\mathbf{coind}}}
\nc{\inv}{{\mathbf{inv}}}
\nc{\triv}{{\mathbf{triv}}}
\nc{\CMaps}{{\mathcal Maps}}
\nc{\Maps}{\on{Maps}}
\nc{\bMaps}{\mathbf{Maps}}
\nc{\BMaps}{\ul{\on{Maps}}}
\nc{\Grid}{\on{Grid}}
\nc{\hGrid}{\on{Grid}^{\geq\,\on{dgnl}}}
\nc{\Diag}{\on{Diag}}
\nc{\bDelta}{\mathbf{\Delta}}
\nc{\tCateg}{(\infty\on{-2)-Cat}}
\nc{\ul}{\underline}
\nc{\Seg}{\on{Seq}}
\nc{\biSeg}{\on{bi-Seq}}
\nc{\triSeg}{\on{tri-Seq}}
\nc{\quadSeg}{\on{quad-Seq}}
\nc{\nSeg}{\on{n-Seq}}
\nc{\Segm}{\on{Seg}^{\on{mkd}}}
\nc{\fLm}{\fL^{\on{mkd}}}
\nc{\inftyCatm}{\inftyCat^{\on{mkd}}}
\nc{\Blocks}{\mathbf{Blocks}}
\nc{\Snakes}{\mathbf{Snakes}}
\nc{\bifL}{\on{bi-}\!\fL}
\nc{\Sets}{\on{Sets}}
\nc{\Ran}{{\on{Ran}}}
\nc{\Vect}{\on{Vect}}
\nc{\Shv}{\on{Shv}}
\nc{\unn}{\mathbf{union}}
\nc{\Spc}{\on{Spc}}
\nc{\ppart}{(\!(t)\!)}
\nc{\qqart}{[\![t]\!]}
\nc{\Dmod}{\on{D-mod}}
\nc{\cD}{\mathcal D}
\nc{\ocD}{\overset{\circ}{\cD}}
\nc{\sfp}{\mathsf{p}}
\nc{\sfq}{\mathsf{q}}
\nc{\DGCat}{\on{DGCat}}
\renc{\det}{\on{det}}
\nc{\Conf}{\on{Conf}}
\nc{\Whit}{\on{Whit}}
\nc{\Reg}{\on{Reg}}
\nc{\Res}{\on{Res}}
\nc{\BunNbx}{(\BunNb)_{\infty\cdot x}}
\nc{\bHecke}{\overset{\bullet}{\on{Hecke}}}
\nc{\Hecke}{\on{Hecke}}
\nc{\bCZ}{\ol\CZ}
\nc{\oCZ}{\overset{\circ}\CZ} 
\nc{\boCZ}{\ol{\oCZ}}
\nc{\sotimes}{\overset{!}\otimes}
\nc{\semiinf}{\on{SI}}
\nc{\coInd}{\on{coInd}}
\nc{\bCM}{\overset{\bullet}\CM{}}
\nc{\bCF}{\overset{\bullet}\CF{}}
\nc{\SI}{\on{SI}}
\nc{\KL}{\on{KL}}
\nc{\Av}{\on{Av}}
\nc{\sV}{\mathsf{V}}

\begin{document}

\author{Dennis Gaitsgory}

\title[The local and global versions of the Whittaker category]{The local and global versions of the Whittaker category}

\dedicatory{To Professor Kyoji Saito, with admiration} 

\date{\today}

\maketitle

\tableofcontents

\section*{Introduction}

\ssec{Whittaker categories}

\sssec{}

Passage to the Whittaker model is an important tool in representation theory of reductive groups over local fields,
and in the theory of automorphic functions.

\medskip

In the local situation, given a representation $V$ of $G(\bK)$ (let us say, for $\bK$ non-Archimedian),
the Whittaker model of $V$ is defined as the space 
of coinvariants 
$$\Whit(V):=V_{N(\bK),\chi},$$
where $N\subset G$ is the maximal unipotent, and $\chi:N(\bK)\to \BC^*$ is a non-degenerate character.

\medskip

In the global situation (let us say over a function field $K$), the global Whittaker space is 
$$\Whit_{\on{glob}}:=\on{Funct}((N(\BA),\chi)\backslash G(\BA)),$$
where $\chi:N(\BA)\to \BC^*$ is chosen so that it is trivial on $N(K)$. 

\medskip

A key feature of the global Whittaker space, which makes it particularly useful for \emph{local-to-global} constructions is that, unlike 
the space $\on{Funct}(G(K)\backslash G(\BA))$ of automorphic functions, the space $\Whit_{\on{glob}}$ is local in nature in that it splits as the
(restricted) tensor product
$$\underset{x}\otimes\, \on{Funct}((N(\bK_x),\chi)\backslash G(\bK_x))$$
(here $x$ runs through the set of places of $K$, and for a given place we denote by $\bK_x$ the corresponding local field). 

\sssec{}

In this paper we work in the geometric context, which by its nature forces us to go one level up in the hierarchy
$$\text{Elements }\to \text{ Sets }\to \text{ Categories } \to \text{ 2-Categories }\to \text{ etc.}$$

Locally, our object of study is \emph{categories equipped with an action of the loop group $\fL(G)$} (\secref{ss:action on categ}
for what this means). In practice (and for the most part in this paper), we will consider the action of $\fL(G)$ on the category
of sheaves on the quotient $\fL(G)/K_n$, where $K_n\subset \fL(G)$ is a congruence subgroup. 

\medskip

Given a category $\bC$ with an action of $\fL(G)$, we wish to attach to it its Whittaker model. However, geometry allows more flexibility 
than the classical theory, and as a result there are two ways in which one can proceed\footnote{We emphasize that the dichotomy explained
below does not seem to have an immediate analog in the classical theory.}: 

\medskip

One can consider the category $\Whit(\bC):=\bC^{\fL(N),\chi}$ by imposing $\fL(N)$-equivariance against $\chi$. Or one
can consider the corresponding category of coinvariants $\Whit(\bC)_{\on{co}}:=\bC_{\fL(N),\chi}$.

\medskip

Now, if instead of $\fL(N)$ we had a finite-dimensional algebraic group (or a pro-finite dimensional algebraic group),
we would know that the two definitions agree (see \thmref{t:coinvariants}). However, $\fL(N)$ is a group ind-scheme,
and there is a priori no reason for such an invariants/coinvariants equivalence to hold\footnote{Such an equivalence does,
however, hold for $\fL(G)$ with $G$ reductive, see \thmref{t:coinvariants and limits, loop}.}. Yet, one can define a functor
(by a non-tautological procedure)
\begin{equation} \label{e:Ps-Id intr}
\on{Ps-Id}:\Whit(\bC)_{\on{co}}\to \Whit(\bC).
\end{equation} 

One of the key results of the paper \cite{Ras} is that the functor \eqref{e:Ps-Id intr} is an equivalence. In the present 
paper, we give an alternative proof of this result (by methods that use global geometry). 

\medskip

Let us emphasize that the fact that \eqref{e:Ps-Id intr} is an equivalence is a specialty of the Whittaker situation. For example, 
an analogously defined functor would \emph{not} be an equivalence if instead of the non-degenerate character
$\chi$ we considered the trivial character (i.e., just invariants/coinvariants for $\fL(N)$, instead of the $\chi$-twisted
version). 

\sssec{}

The fact that the equivalence \eqref{e:Ps-Id intr} holds is really good news in that it says that the operation of passage
to the Whittaker model in the local geometric situation is a well-behaved operation. For example, it implies that the
assignment 
\begin{equation} \label{e:loc Whit intro}
\bC\mapsto \Whit(\bC)
\end{equation}
commutes with limits and colimits, and with the operation of passage to the dual category. 

\sssec{Why do we care about the local Whittaker model?}

The assignment \eqref{e:loc Whit intro}, viewed as a (2)-functor from the (2)-category of DG categories
equipped with an action of $\fL(G)$ to that of plain DG categories, plays a key role in the 
\emph{local geometric Langlands correspondence}. 

\medskip

To explain this, we need to place ourselves in the context of D-modules. In this case for every choice of 
\emph{level} (which is a $W$-invariant symmetric bilinear form $\Lambda\otimes \Lambda \to k$, where
$\Lambda$ is the coweight lattice of $G$) there corresponds the notion of \emph{category acted on by 
$\fL(G)$ at level $\kappa$}.  Denote the (2)-category of such by $\fL(G)\mmod_\kappa$.

\medskip

Assume that $\kappa$ is non-degenerate, i.e., it defines an \emph{isomorphism}
$$\ft\simeq k\underset{\BZ}\otimes \Lambda\to k\underset{\BZ}\otimes \cLambda\simeq \ct.$$
Transferring $\kappa$ to $\ct$, we obtain a form $\check\kappa$ on $\cLambda$. 

\medskip

The local geometric Langlands conjecture says that there exists a canonical (2)-equivalence of (2)-categories
\begin{equation} \label{e:loc geom Langlands}
\fL(G)\mmod_\kappa\simeq \fL(\cG)\mmod_{-\check\kappa}.
\end{equation}

\medskip

Now, the Whittaker model plays a crucial role in characterizing the (2)-equivalence \eqref{e:loc geom Langlands}. 
Namely, if 
$$\bC\in \fL(G)\mmod_\kappa \text{ and } \check\bC\in  \fL(\cG)\mmod_{-\check\kappa}$$
are two objects that \emph{correspond to each other} under the (2)-equivalence \eqref{e:loc geom Langlands},
then the Whittaker model of $\bC$, i.e., $\Whit(\bC)$, and the \emph{Kac-Moody model} of $\check\bC$ 
are equivalent \emph{as DG categories}. And vice versa, i.e., when the roles of $G$ and $\cG$ are swapped.  

\medskip

Here the Kac-Moody model of a category $\bC$ acted on by the loop group $\fL(G)$ at level $\kappa$,
denoted $\on{KM}(\bC)$, is the DG category of \emph{weak invariants} on $\bC$ with respect to
the loop group. Equivalently,
$$\on{KM}(\bC)=\on{Funct}_{\fL(G)\mmod_\kappa}(\hg\mod_\kappa,\bC),$$
where $\hg\mod_\kappa$ is the category of modules for the Kac-Moody algebra at level $\kappa$, viewed as an object of 
$\fL(G)\mmod_\kappa$.

\ssec{The global Whittaker category}

\sssec{}  \label{sss:why global}

We now come to the main point of focus of this paper. Let us take $\bC$ to be the category
of sheaves on $\fL(G)/K_n$, so that
$$\Whit(\bC)=\Shv(\fL(G)/K_n)^{\fL(N),\chi}.$$

\medskip

There are two issues one needs to address for practical applications:

\medskip

\noindent(a) The category $\Shv(\fL(G)/K_n)^{\fL(N),\chi}$ is inherently infinite-dimensional in 
nature in that all of its objects have infinite-dimensional support\footnote{That said, Raskin's results in 
\cite{Ras} show that $\Shv(\fL(G)/K_n)^{\fL(N),\chi}$ is a union of full subcategories, such that
compact objects in each of them can be expressed through sheaves with finite-dimensional support.}. So it would be desirable to
find another description of $\Shv(\fL(G)/K_n)^{\fL(N),\chi}$ that would involve sheaves on 
finite-dimensional algebro-geometric objects (such as algebraic stacks).

\medskip

\noindent(b) In the global Geometric Langlands theory, one studies the functor that 
relates the category $\Shv(\fL(G)/K_n)^{\fL(N),\chi}$
to the category of sheaves on $\Shv(\Bun_G^{\on{level}_{n\cdot x}})$, where 
$\Bun_G^{\on{level}_{n\cdot x}}$ is the moduli stack of $G$-bundles (on a given curve $X$)
with structure of level $n$ at $x$. The functor in question is 
$$\Shv(\fL(G)/K_n)^{\fL(N),\chi}\to \Shv(\fL(G)/K_n)\to \Shv(\Bun_G^{\on{level}_{n\cdot x}}),$$
where the first arrow is the forgetful functor, and the second arrow is the functor of !-direct image along
the uniformization map $\fL(G)/K_n\to \Bun_G^{\on{level}_{n\cdot x}}$. However, in order to control 
various properties of this functor (e.g., behavior with respect to Verdier duality), it would again be desirable
a finite-dimensional model for $\Shv(\fL(G)/K_n)^{\fL(N),\chi}$, as well as the above functor itself.

\medskip

The goal of this paper is to describe such a finite-dimensional model for $\Shv(\fL(G)/K_n)^{\fL(N),\chi}$,
addressing  points (a) and (b) above. 

\sssec{}

In order to explain what this finite-dimensional model is, let us return to the classical situation.
Consider the space of Whittaker functions that are non-ramified away from a particular place $x$, i.e.,
$$\on{Funct}((N(\BA),\chi)\backslash G(\BA)/K_n\times G(\BO^x)),$$
where
$$\BO^x=\underset{x'\neq x}\Pi\, \bO_{x'},$$
and $K_n$ is the $n$-th congruence subgroup at $x$. We normalize $\chi$ so that its conductor is $N(\BO)\subset N(\BA)$.

\medskip

Let
$$\on{Funct}((N(\BA),\chi)\backslash G(\BA)/K_n\times G(\BO^x))_0\subset \on{Funct}((N(\BA),\chi)\backslash G(\BA)/K_n\times G(\BO^x))$$
be the subspace of functions that are supported on 
$$G(\bK_x)\times N(\BA^x)\cdot G(\BO^x) \subset 
G(\bK_x)\times G(\BA^x)=G(\BA),$$
where 
$$\BA^x=\underset{x'\neq x}\Pi\, \bK_{x'},$$

Then we have an isomorphism 
\begin{equation} \label{e:funct loc to glob}
\on{Funct}((N(\BA),\chi)\backslash G(\BA)/K_n\times G(\BO^x))_0\simeq \on{Funct}((N(\bK_x),\chi)\backslash G(\bK_x)/K_n).
\end{equation}

\medskip

Thus, 
\begin{equation} \label{e:funct Whit glob}
\on{Funct}((N(\BA),\chi)\backslash G(\BA)/K_n\times G(\BO^x))_0
\end{equation}
is isomorphic to $\Whit(\on{Funct}(G(\bK_x)/K_n))$.

\sssec{}  \label{sss:glob Whit recipe}

The next observation is that the space \eqref{e:funct Whit glob} can be realized as the subspace of
functions on the quotient
\begin{equation} \label{e:N glob quo}
N(K)\backslash G(\bK_x) \times N(\BA^x)/K_n\times N(\BO^x).
\end{equation}

Moreover, this subspace can be characterized by a certain equivariance property, as follows. Choose a point $x'\neq x$, and consider the space
\begin{equation} \label{e:N glob quo prime}
N(K)\backslash G(\bK_x) \times N(\BA^x)/K_n\times N(\BO^{x\cup x'}),
\end{equation}
where $\BO^{x\cup x'}$ is defined in the same was as $\BO^x$ above with two places instead of one.

\medskip

This space is acted on by $N(\bK_{x'})$ on the right, so that
$$N(K)\backslash G(\bK_x) \times N(\BA^x)/K_n\times N(\BO^x)\simeq 
\left(N(K)\backslash G(\bK_x) \times N(\BA^x)/K_n\times N(\BO^{x\cup x'})\right)/N(\bO_{x'}).$$

Now
\begin{equation} \label{e:glob Whit intro funct}
\on{Funct}((N(\BA),\chi)\backslash G(\BA)/K_n\times G(\BO^x))_0 \subset N(K)\backslash G(\bK_x) \times N(\BA^x)/K_n\times N(\BO^x)
\end{equation}
consists of those elements that after pullback along
$$N(K)\backslash G(\bK_x) \times N(\BA^x)/K_n\times N(\BO^{x\cup x'})\to N(K)\backslash G(\bK_x) \times N(\BA^x)/K_n\times N(\BO^x)$$
transform along the character $\chi|_{\bN(K_{x'})}$ with respect to the above action of $N(\bK_{x'})$. 

\sssec{}

The point is that the space \eqref{e:N glob quo} (as well as \eqref{e:N glob quo prime}) has a direct analog in geometry,
and the resulting geometric object is an (ind)-algebraic stack, with one caveat. 

\medskip

The (ind)-algebraic stack whose $k$-points (almost) match \eqref{e:N glob quo} is a version of Drinfeld's compactification, 
denoted $\BunNbx^{G\on{-level}_{n\cdot x}}$; it is introduced in \secref{ss:Drinf comp}. (The (ind)-algebraic stack corresponding
to \eqref{e:N glob quo prime} is introduced in \secref{sss:extend to loop group}).  

\medskip

We introduce the global version of the Whittaker category to be a full subcategory
\begin{equation} \label{e:glob Whit intro}
\Whit(\BunNbx^{G\on{-level}_{n\cdot x}})\subset \Shv(\BunNbx^{G\on{-level}_{n\cdot x}}),
\end{equation}
by mimicking the recipe in \secref{sss:glob Whit recipe}. 

\begin{rem}

The caveat, alluded to above, is that $k$-points of $\BunNbx^{G\on{-level}_{n\cdot x}}$ do not really match 
\eqref{e:N glob quo}. In fact, the former is a proper subset of the latter. Geometrically, $\BunNbx^{G\on{-level}_{n\cdot x}}$
has a stratification, and \eqref{e:N glob quo} corresponds to the union of $k$-points of some of the strata (let us call these
strata {\it relevant}, and the other strata {\it irrelevant}).

\medskip

That said, a feature of the category $\Whit(\BunNbx^{G\on{-level}_{n\cdot x}})$ is that for all of its objects, their
restrictions (both $!$- and $*$-versions) to the irrelevant strata vanish. So, the subcategory \eqref{e:glob Whit intro}
is a geometric analog of the subspace \eqref{e:glob Whit intro funct}. 

\end{rem}

\sssec{}

We have a tautologically defined map
$$\pi:\fL(G)/K_n\to \BunNbx^{G\on{-level}_{n\cdot x}},$$
and one shows that the pullback functor
$$\pi^!:(\BunNbx^{G\on{-level}_{n\cdot x}})\to \Shv(\fL(G)/K_n)$$
sends
$$\Whit(\BunNbx^{G\on{-level}_{n\cdot x}})\mapsto \Whit(\fL(G)/K_n):=\Shv(\fL(G)/K_n)^{\fL(N),\chi}.$$

The main theorem of the present paper, \thmref{t:main}, says that the resulting functor
\begin{equation} \label{e:main intro}
\Whit(\BunNbx^{G\on{-level}_{n\cdot x}})\to \Whit(\fL(G)/K_n)
\end{equation}
is an equivalence.

\medskip

\thmref{t:main} is a geometric analog of the (more or less tautological) function-theoretic isomorphism
\eqref{e:funct loc to glob}. However, \thmref{t:main} is not tautological. It is easy to show that the functor
\eqref{e:main intro} induces a \emph{strata-wise} equivalence (and at the level of functions this is all one
needs to show). But the fact that the subquotients on both sides corresponding to different strata glue in
the same way requires a non-trivial argument. 

\sssec{}
 
The left-hand side of the equivalence \eqref{e:main intro} provides the sought-for finite-dimensional model for
$\Shv(\fL(G)/K_n)^{\fL(N),\chi}$, see \secref{sss:why global}(a). It also provides the answer to 
\secref{sss:why global}(b).  Namely, the corresponding functor
$$\Whit(\BunNbx^{G\on{-level}_{n\cdot x}})\to \Shv(\Bun_G^{\on{level}_{n\cdot x}})$$
is the composite 
$$\Whit(\BunNbx^{G\on{-level}_{n\cdot x}})\to \Shv(\BunNbx^{G\on{-level}_{n\cdot x}})\to \Shv(\Bun_G^{\on{level}_{n\cdot x}}),$$
where the first arrow is the tautological inclusion, and the second arrow is !-direct image with respect to the natural
morphism of algebraic stacks:
$$\BunNbx^{G\on{-level}_{n\cdot x}}\to \Bun_G^{\on{level}_{n\cdot x}}.$$

\begin{rem}
Historically, one has been using $\Whit(\BunNbx^{G\on{-level}_{n\cdot x}})$ as a surrogate for the local Whittaker category 
long before the appearance of the direct local definition of $\Whit(\fL(G)/K_n)$ as $\Shv(\fL(G)/K_n)^{\fL(N),\chi}$. So this paper 
provides a justification of why this surrogate is valid. 

\medskip

The model for $\Whit(\fL(G)/K_n)$ as $\Whit(\BunNbx^{G\on{-level}_{n\cdot x}})$
had been used for both local considerations (see, e.g., \cite{FGV}, where it is used to prove the geometric 
Casselman-Shalika formula, or \cite{Ga3}), and for global ones (see, e.g., \cite{FGKV,Ga4,Ga5}). 

\medskip

The reason for this was that in order to define $\Whit(\fL(G)/K_n)$ as $\Shv(\fL(G)/K_n)^{\fL(N),\chi}$, one needed to
overcome several (mostly psychological) obstructions: 

\medskip

For one thing, when definining $\Shv(\fL(G)/K_n)^{\fL(N),\chi}$ one needs to work with the \emph{large} category
$\Shv(\fL(G)/K_n)$ (i.e., the ind-completion of the more conventional category of sheaves with finite-dimensional
support). 

\medskip

Second, as we shall see in \propref{p:invis}, the objects of $\Shv(\fL(G)/K_n)^{\fL(N),\chi}$ are \emph{invisible}
from the point of view of the t-structure on $\Shv(\fL(G)/K_n)$ (technically, all these objects are \emph{infinitely connective}).
Thus, one had to really leave the world of abelian categories to define $\Shv(\fL(G)/K_n)^{\fL(N),\chi}$. (That said, we should
mention that the category $\Shv(\fL(G)/K_n)^{\fL(N),\chi}$ carries its own t-structure with a non-trivial heart.)

\medskip

And third, prior to Raskin's paper or our \thmref{t:main}, even if one defined $\Whit(\fL(G)/K_n)$ as $\Shv(\fL(G)/K_n)^{\fL(N),\chi}$,
it would be totally unclear how to compute anything in it: we need a finite-dimensional model to perform actual computations.

\end{rem}

\ssec{What is actually done in this paper?}

Here is a brief synopsis of the mathematical contents of this paper. 

\sssec{Definition of the local Whittaker category} \hfill

\medskip

\noindent--We define the Whittaker category $\Whit(\CY)$ for $\CY=\fL(G)/K_n$ as $\Shv(\CY)^{\fL(N),\chi}$. 

\smallskip

\noindent--We state (and subsequently prove) the non-obvious fact that $\Whit(\CY)$ is compactly generared. 

\smallskip

\noindent--We define a stratification on $\Whit(\CY)$ that corresponds to the stratification of $\Gr_G=\fL(G)/\fL^+(G)$
by $\fL(N)$-orbits. We show that the category on each stratum can be expressed in terms of finite-dimensional 
algebro-geometric objects. 

\sssec{Dual definition} \hfill

\medskip

\noindent--We define the dual version of the Whittaker category, denoted $\Whit(\CY)_{\on{co}}$.  

\smallskip

\noindent--We study the strata-wise behavior of $\Whit(\CY)_{\on{co}}$, and we show that it reproduces that of 
$\Whit(\CY)$.

\smallskip

\noindent--We define the functor $\on{Ps-Id}:\Whit(\CY)_{\on{co}}\to \Whit(\CY)$ and state (and subsequently prove)
the theorem that says that it is an equivalence. 

\sssec{Global definition} \hfill

\medskip

\noindent--We define the global Whittaker category $\Whit(\BunNbx^{G\on{-level}_{n\cdot x}})$. 

\medskip

\noindent--We construct a functor $\Whit(\BunNbx^{G\on{-level}_{n\cdot x}})\to \Whit(\CY)$ and state (and subsequently prove)
our main result, \thmref{t:main}, which says that the above functor is an equivalence. We explain that the non-trivial
part is the fully-faithfulness assertion (something that does not have a counterpart in the classical theory). 

\medskip

\noindent--We introduce a \emph{Ran version} of $\Whit(\fL(G)/K_n)$, denoted $\Whit((\fL(G)/K_n)_{\Ran})$. We prove that the pullback
functor $\Whit(\BunNbx^{G\on{-level}_{n\cdot x}})\to \Whit((\fL(G)/K_n)_{\Ran})$ is fully faithful. 
 
\medskip

\noindent--We prove the equivalence $\Whit((\fL(G)/K_n)\times \Ran(X)_x)\simeq \Whit((\fL(G)/K_n)_{\Ran})$. This implies the 
fully-faithfulness of the functor in \thmref{t:main} by an easy retraction argument. 

\sssec{Generalizations} \hfill

\medskip

\noindent--We define the Whittaker models $\Whit(\bC)$ and $\Whit(\bC)_{\on{co}}$ for an abstract DG category $\bC$ with
an action of $\fL(G)$, and use \thmref{t:main} to deduce an equivalence $\Whit(\bC)\simeq \Whit(\bC)_{\on{co}}$. 

\medskip

\noindent--We consider the ``factorizable" situation, when instead of a fixed formal disc (that the loop group $\fL(G)$ 
is attached to), we consider the \emph{multi-disc} parameterized by points of $X^n$ for some integer $n$. The results
of this paper hold in this more general situation with no major modifications. 

\sssec{Groups acting on categories} \hfill

\medskip

\noindent--We review the theory of actions of a (finite-dimensional) algebraic group on a DG category. We prove
that in this case, the categories of invariants and coinvariants are canonically equivalent.

\medskip

\noindent--We review the theory of \emph{placid (ind)-schemes} and sheaves on these objects. 

\medskip

\noindent--We review the notion of action of a loop group $\fL(G)$ on a DG category.  We show that for $G$
reductive, the resulting categories of invariants and coinvariants are canonically equivalent.

\ssec{Structure of the paper}

We will now describe the contents of the paper, section-by-section.

\sssec{}

In \secref{s:prelim} we collect some preliminaries: loop and arc spaces, categories of sheaves, group actions on categories and
equivariance.

\medskip

The reader familiar with this material is advised to skip this section and return to it when necessary. 

\sssec{}

In \secref{s:loc Whit} we introduce our basic object of study: the local Whittaker category. 

\medskip

We define the local Whittaker category $\Whit(\CY)$ as $\Shv(\CY)^{\fL(N),\chi}$, where $\CY=\fL(G)/K_n$.
We show how the stratification of the affine Grassmannian $\fL(G)/\fL^+(G)$ by $\fL(N)$-orbits defines a stratification of 
$\Whit(\CY)$ with explicit subquotients.

\medskip

We discuss the question of compact generation of $\Whit(\CY)$. This is not obvious, as the definition of $\Whit(\CY)$
involves an infinite limit. We introduce the notion of \emph{adapted object} of $\Shv(\CY)$; these are objects for
which the functor of !-averaging along $\fL(N)$ against the character $\chi$ is defined and well-behaved. We show
that if $\Shv(\CY)$ has ``enough" adapted objects, then $\Whit(\CY)$ is compactly generated. We then exhibit
a supply of adapted objects, using a recipe from \cite{Ras}. 

\sssec{}

In \secref{s:dual} we define the other version of the local Whittaker category, denoted $\Whit(\CY)_{\on{co}}$,
as $(\fL(N),\chi)$-coinvariants of $\Shv(\CY)$. We show that the stratification of $\Whit(\CY)_{\on{co}}$ 
that arises from the stratification of the affine Grassmannian has subquotients isomorphic to those of 
$\Whit(\CY)$. 

\medskip

We show that the supply of adapted objects in $\Shv(\CY)$
makes $\Whit(\CY)_{\on{co}}$ compactly generated as well, and that it is the category dual of $\Whit(\CY)$
(up to replacing $\chi$ by its inverse). 

\medskip

We introduce the ``non-standard" averaging functor 
\begin{equation} \label{e:PsId intro}
\on{Ps-Id}:\Whit(\CY)_{\on{co}}\to \Whit(\CY),
\end{equation}
which would automatically be an equivalence if instead of $\fL(N)$ we had a (pro)finite-dimensional algebraic
group. We state the theorem that $\on{Ps-Id}$ is an equivalence in our case as well. We emphasize that
for the validity of this assertion it is crucial that we are working with a non-degenerate character $\chi$. 

\sssec{}

In \secref{s:global} we introduce the global Whittaker category, by mimicking the procedure in \secref{sss:glob Whit recipe}.
The underlying geometric object is a version of Drinfeld's compactification, denoted $\BunNbx^{G\on{-level}_{n\cdot x}}$. 
To spell out the definition, we first choose a collection of auxiliary points $\ul{y}$ on our curve, and then we show that the
definition is independent of this choice. 

\medskip

We show that a natural stratification (by the order of degeneration) on $\BunNbx^{G\on{-level}_{n\cdot x}}$ defines a 
stratification on the Whittaker category $\Whit(\BunNbx^{G\on{-level}_{n\cdot x}})$, with only the ``relevant" strata
carrying non-zero objects. 

\sssec{}

In \secref{s:global-to-local} we show that pullback along the map
$$\pi:\CY=\fL(G)/K_n\to \BunNbx^{G\on{-level}_{n\cdot x}}$$
defines a functor
\begin{equation} \label{e:main functor intro}
\Whit(\BunNbx^{G\on{-level}_{n\cdot x}})\to \Whit(\CY).
\end{equation}

We first show that this functor is a strata-wise equivalence. We then proceed to state our main result,
\thmref{t:main}, which says that the functor \eqref{e:main functor intro} is an equivalence. Given the strata-wise
equivalence, we see that \thmref{t:main} is equivalent to the assertion that the functor \eqref{e:main functor intro} 
is fully faithful.

\medskip

We show that
the equivalence \eqref{e:PsId intro} follows formally from \thmref{t:main}.

\sssec{}

In \secref{s:Ran} we prove \thmref{t:main}. The idea of the proof is to consider two more versions 
of the category $\Whit(\CY)$ that involve the Ran space, denoted $\Whit(\CY_{\Ran_x} )$ and $\Whit(\CY\times \Ran(X)_x)$. 
We will have a commutative diagram of functors
$$
\CD
\Whit(\BunNbx^{G\on{-level}_{n\cdot x}})   @>>> \Whit(\CY_{\Ran_x} ) \\
@V{\pi^!}VV   @VVV \\
\Whit(\CY)  @>>> \Whit(\CY\times \Ran(X)_x).
\endCD
$$

We will see that the two horizontal functors are fully faithful: this is a general contractibility-type
assertion. Finally, we will show that the right vertical arrow is an equivalence. This would involve
showing that there are ``enough" adapted objects, so we will essentially use Raskin's recipe again. 
As a result, we will see that $\pi^!$ is fully faithful, as required.

\sssec{}

In \secref{s:gen} we discuss several generalizations of \thmref{t:main}. 

\medskip

We show that instead of considering a fixed punctured disc, we can consider the multi-disc parameterized by points of $X^n$.
In this way we obtain a \emph{factorizable version} of the results obtained in the preceding sections. 

\medskip

We consider the abstract setting of a DG category $\bC$ acted on by the loop group $\fL(G)$, and study its associated
Whittaker categories $\Whit(\bC)$ and $\Whit(\bC)_{\on{co}}$. We show that the fact that the functor \eqref{e:PsId intro}
is an equivalence implies that the corresponding functor $\Whit(\bC)_{\on{co}}\to \Whit(\bC)$ is also an equivalence.

\sssec{}

In \secref{s:proof of Rask} we review the input we need from Raskin's work \cite{Ras} for the present paper;
we provide a detailed proof of the relevant geometric results.

\sssec{}

In \secref{s:coinv} we show that for a category $\bC$ acted on by a finite-dimensional group $H$, the functor
$\bC_H\to \bC^H$, given by *-averaging with respect to $H$, is an equivalence.

\sssec{}

In \secref{s:placid} we review the notion of {\it placid} scheme (resp., ind-scheme). These are algebro-geometric
objects of \emph{inifinite type}, but ones for which one can easily bootstrap the theory of sheaves from the
finite type situation\footnote{In the present paper we are not (yet) trying to attack sheaf theory in infinite type 
directly, e.g., \`a la \cite{BS}.}. 

\medskip

We should emphasize, however, that for a placid (ind)-scheme $\CY$, the resulting category $\Shv(\CY)$ does \emph{not} 
come equipped with a t-structure. Choosing a t-structure on $\Shv(\CY)$ involves a trivialization of a certain $\BZ$-gerbe
(the dimension gerbe). We will not pursue this in the present paper. 

\medskip

We show that the loop group $\fL(G)$ is a placid ind-scheme, so the category $\Shv(\fL(G))$
is something manageable. 

\sssec{}

Finally, in \secref{s:loop coinv} we show that if $G$ is reductive, for a category $\bC$ acted on by $\fL(G)$, there
exists a canonical equivalence $\bC^{\fL(G)}\simeq \bC_{\fL(G)}$. This extends the result from \secref{s:coinv} from
the case of a finite-dimensional group to the case of a loop group of a reductive group $G$. 

\ssec{Conventions}

\sssec{}  \label{sss:algebraic geometry}

We will be working over an algebraically closed ground field, denoted $k$.   In this paper we will \emph{not} need
derived algebraic geometry (this is because we will work with sheaf theories of topological nature, see \secref{sss:shvs}). 

\medskip

We let $\affSch$ denote the category of affine schemes over $k$, and by
$\affSch_{\on{ft}}\subset \affSch$ 
its full subcategory consisting of affine schemes of finite type over $k$.

\medskip

All other algebro-geometric objects that we will encounter are classical \emph{prestacks}, i.e., (accessible) functors
\begin{equation} \label{e:prestk}
(\affSch)^{\on{op}}\to \on{Grpds},
\end{equation}
where $\on{Grpds}$ is the category of classical groupoids. 

\medskip

We let $\on{PrStk}_{\on{lft}}\subset \on{PrStk}$ be the full subcategory of prestacks locally of finite type. 
By definition, an object $\CY\in \on{PrStk}$ belongs to $\on{PrStk}_{\on{lft}}$, if when viewed as a functor
\eqref{e:prestk}, it takes filtered limits in $\affSch$ to colimits in $\on{Grpds}$. Equivalently, $\CY\in \on{PrStk}_{\on{lft}}$
if it is isomorphic to the left Kan extension of its own restriction to 
$$(\affSch_{\on{ft}})^{\on{op}}\subset (\affSch)^{\on{op}}.$$

Thus, we can identify $\on{PrStk}_{\on{lft}}$ with the category of functors
$$(\affSch_{\on{ft}})^{\on{op}}\to \on{Grpds}.$$

\sssec{}  \label{sss:DG categories}

We let $\sfe$ denote the field of coefficients. We will be working with DG categories over $\sfe$; we refer the 
reader to \cite[Chapter 1, Sect. 10]{GR1} for a detailed exposition of the theory of DG categories. 

\medskip

Unless specified otherwise, we will assume our DG categories to be cocomplete (i.e., contain infinite direct sums,
equivalently colimits). 

\medskip

We let $\on{DGCat}_{\on{cont}}$ denote the $\infty$-category, whose objects are \emph{cocomplete} DG
categories and whose 1-morphisms are \emph{continuous} (i.e., colimit preserving) functors. 

\medskip

The category $\on{DGCat}_{\on{cont}}$ carries a symmetric monoidal structure, given by the Lurie tensor product.
Thus, for $\bC\in \on{DGCat}_{\on{cont}}$ it makes sense to talk about its dualizability as an object of $\on{DGCat}_{\on{cont}}$.

\sssec{}

Given a DG category $\bC$, we let $\bC_c$ denote its full (but not cocomplete) subcategory consisting
of compact objects. 

\medskip

We remind the reader that if $\bC$ is a compactly generated category, then it is dualizable, and we have a canonical
equivalence 
$$(\bC^\vee)_c\simeq (\bC_c)^{\on{op}}.$$

\sssec{}

We will fix a sheaf theory, see \secref{sss:shvs}, whose field of coefficients is $\sfe$. So for every
$\CY\in \on{PrStk}_{\on{lft}}$, we have a well-defined object $\Shv(\CY)\in \on{DGCat}_{\on{cont}}$.

\medskip

This category is compactly generated for $\CY=S\in \affSch_{\on{ft}}$, more or less by definition. From
here one can deduce that it is compactly generated also for \emph{schemes} (resp., \emph{ind-schemes})
that are of finite type (resp., of ind-finite type). In general, the question of compact generation of $\Shv(\CY)$ 
for a given prestack is a non-trivial one. 

\sssec{}

We let $G$ denote a reductive group over $k$. We fix a Bore subgroup $B\subset G$. Let $N\subset B$
denote its unipotent radical, and let $T$ denote the Cartan quotient of $B$.

\medskip

We let $\Lambda$ denote the coweight lattice of $T$, and $\cLambda$ its dual, i.e., the weight lattice. 
Let $\Lambda^+\subset \Lambda$ denote the sub-monoid of dominant coweights, and similarly for $\cLambda$.

\medskip

We let
$\Lambda^{\on{pos}}\subset \Lambda$ be sub-monoid equal to the non-negative integral span of simple coroots.

\ssec{Acknowledgements}

I would like to thank Sam Raskin for numerous discussions related to subjects treated in this paper.
In particular, I am grateful to him for explaining to me the contents of \cite{Ras}, which this paper crucially relies on.  

\medskip

The definition of the local Whittaker category as $\bC^{\fL(N),\chi}$
was the result of my conversations with Jacob Lurie back in 2008. I am much indebted to him
for teaching me not only higher algebra, but also how to think systematically about big vs 
small categories, compact generation, etc. 

\medskip

I am grateful to E.~Frenel, D. ~Kazhdan and K.~Vilonen for the collaboration in \cite{FGKV} and \cite{FGV}--it is through that work
that the idea to use the Drinfeld compactification for defining the global version of the Whittaker category was created. 

\medskip

The author's research is supported by NSF grant DMS-1063470. He has also received support from ERC grant 669655.

\section{Preliminaries}  \label{s:prelim}

In this section we collect some miscellanea: loop spaces, sheaf theory, ind--schemes and group actions 
on categories.

\medskip

The reader in encouraged to skip this section and return to it when necessary. 

\ssec{The geometric objects}

\sssec{}

Let $Z$ be a scheme of finite type. We define the scheme of arcs $\fL^+(Z)$ to represent the functor on $k$-algebras
$$R\mapsto \Hom(\Spec(R\qqart),Z).$$

We have:
$$\fL^+(Z)\simeq \underset{n\in \BZ^{\geq 0}}{\on{lim}}\, \fL^+(Z)_n,$$
where each $\fL^+(Z)_n$ is a scheme (of finite type) given by the functor
$$R\mapsto \Hom(\Spec(R[t]/t^n),Z).$$

\medskip

Note that if $Z$ is smooth, then the transition maps
$$\fL^+(Z)_{n''}\to \fL^+(Z)_{n'},\quad n''\geq n'$$
are smooth.

\medskip

We define the prestack of loops $\fL(Z)$ to represent the functor on $k$-algebras
$$R\mapsto \Hom(\Spec(R\ppart),Z).$$

One easily shows (by reducing to the case of the affine space) that when $Z$ is affine, the prestack 
$\fL(Z)$ is an ind-scheme, which contains $\fL^+(Z)$ as a closed subscheme. 

\sssec{}

Let $G$ be a an affine algebraic group. Then by functoriality $\fL^+(G)$ and all $\fL^+(G)_n$
are group-schemes, and $\fL(G)$ is a group ind-scheme. In what follows we will denote by $K_n\subset \fL^+(G)$
the $k$-th congruence subgroup, i.e., the kernel of the projection $\fL^+(G)\to \fL^+(G)_n$. 

\medskip

For each $n$ we can consider the stack-quotient $\CY:=\fL(G)/K_n$ (i.e., we take the prestack 
quotient, and sheafify it in the \'etale (or, which would produce the same result, fpqc) topology). 

\medskip

The prestack $\CY$ is known to be an ind-scheme of ind-finite type, and it represents the functor that sends 
a $k$-algebra $R$ to the set of triples $(\CP_G,\gamma,\epsilon)$, where $\CP_G$ is a $G$-bundle 
on $\Spec(R\qqart)$, $\gamma$ is a trivialization of the restriction to $\CP_G$ to $\Spec(R\ppart)$,
and $\epsilon$ is a trivialization of the restriction to $\CP_G$ to $\Spec(R[t]/t^n)$. 

\medskip

We can write $\CY$ as a (filtered) union of its closed $\fL^+(G)$-stable subschemes $Y_i$. The action of 
$\fL^+(G)$ on each $Y_i$ factors through the quotient $\fL^+(G)\to \fL^+(G)_n$ for all $n$ sufficiently large
(depending on $Y_i$). 

\sssec{}

Let $T$ be a torus and let $\lambda:\BG_m\to T$ be a co-character. We will denote by $\lambda(t)$ the point of
$\fL(T)$ given by the map
$$\Spec(k\ppart)\to \Spec(k[t,t^{-1}])=\BG_m\overset{\lambda}\to T.$$

\sssec{}

Let $N$ be the unipotent radical of a Borel inside a reductive group $G$. Consider the corresponding group ind-scheme
$\fL(N)$. We observe that in addition to being a group ind-scheme (i.e., a group-object in the category of ind-schemes),
it is naturally an ind group-scheme (i.e., an ind-object in the category of group-schemes). In other words, $\fL(N)$
can be written as a filtered union of its closed group sub-schemes: 
\begin{equation} \label{e:loop N as colim}
\fL(N)\simeq \underset{\alpha\in A}{\on{colim}}\, N^\alpha,
\end{equation} 
where $A$ is a filtered category.

\medskip

Indeed, we can take the index category to be the set of dominant coweights in the adjoint quotient of $G$; for each
such coweight (denoted $\lambda$), we let the corresponding subgroup be
$$\on{Ad}_{-\lambda(t)}(\fL^+(N)).$$

 \medskip

Each group-scheme $N^\alpha$, can be written as
\begin{equation} \label{e:arc N as lim}
\underset{\beta\in B_{\alpha,i}}{\on{lim}}\, N^\alpha_\beta,
\end{equation} 
where:

\medskip

\noindent(i) $B_{\alpha,i}$ is a filtered category;

\smallskip

\noindent(ii) Each $N^\alpha_\beta$ is a unipotent algebraic group (of finite type);

\smallskip

\noindent(iii) For every $(\beta\to \beta')\in B_{\alpha,i}$ the corresponding map
$N^\alpha_{\beta'}\to N^\alpha_\beta$
is surjective. 

\medskip

Moreover, if $N^\alpha$ acts on a scheme $Y$ of finite type, this action 
comes from a compatible family of actions of $N^\alpha_\beta$'s on $Y$.

\sssec{}

The choice of the uniformizer $t\in k\qqart$ gives rise to a homomorphism
$$\fL(\BG_a)\to \BG_a, \quad \underset{n}\Sigma, a_n\cdot t^n\mapsto a_{-1}.$$

From here we obtain a homomorphism $\chi:\fL(N)\to \BG_a$
equal to
$$\fL(N) \to \fL(N/[N,N])\to \fL(\BG_a)\to \BG_a,$$
where the second arrows comes from the map
$$N/[N,N]\simeq \underset{i\in \CI}\Pi\, \BG_a \overset{\on{sum}}\longrightarrow \BG_a,$$
where $\CI$ is the set of vertices of the Dynkin diagram of $G$.

\ssec{Categories of sheaves}  

\sssec{}  \label{sss:shvs}

We adopt the conventions regarding sheaf theory from \cite{Ga6}.  We will denote by
\begin{equation} \label{e:Shv}
\Shv:(\affSch_{\on{ft}})^{\on{op}}\to \on{DGCat}_{\on{cont}}
\end{equation} 
the functor that attaches to an affine scheme of finite type $S$ the (DG) category $\Shv(S)$ and to a morphism
$S_1\overset{f}\to S_2$ the pullback functor $f^!:\Shv(S_2)\to \Shv(S_1)$.

\medskip

Examples of such theories are:

\smallskip

\noindent(i) If the ground field $k$ has characteristic $0$, we can take $\Shv(S)=\Dmod(S)$ (see \cite[Chapter 4, Sect. 1.2]{GR2}
with the caveat that {\it loc.cit.} the notation is $\on{Crys}(S)$ rather than $\Dmod(S)$). In this case the field of coefficients $\sfe$ 
equals $k$. 

\smallskip

\noindent(ii) For any ground field, one can take $\Shv(S)$ to be the ind-completion of the DG category of constructible
$\BQ_\ell$-adic sheaves (for $\ell$ invertible in $k$), as defined in \cite[Sect. 2.3.2]{GL}. In this case, the field $\sfe$ is 
coefficients is $\BQ_\ell$ or a finite extension thereof. 

\smallskip

\noindent(iii) If the ground field is the field of complex numbers $\BC$, one can take $\Shv(S)$ to be the ind-completion 
of the DG category of constructible $\sfe$-sheaves, for any field characteristic zero field $\sfe$. 

\sssec{}

We apply to \eqref{e:Shv} the procedure of right Kan extension and obtain a functor
\begin{equation}  \label{e:shvs on prestacks}
\Shv:(\on{PreStk}_{\on{lft}})^{\on{op}}\to \on{DGCat}_{\on{cont}}.
\end{equation}

In particular, for every $\CY\in \on{PreStk}_{\on{lft}}$ we have a well-defined DG category
$\Shv(\CY)$, and for a morphism $f:\CY_1\to \CY_2$ we have a continuous functor
$$f^!:\Shv(\CY_2)\to \Shv(\CY_1).$$

\sssec{}

The functor \eqref{e:shvs on prestacks} has a remarkable feature that it encodes not only
the !-pullback functor, but also the *-pushforward functor for schematic morphisms. 

\medskip

Namely, for a quasi-compact schematic map $f:\CY_1\to \CY_2$ , we have a well-defined functor
$$f_*:\Shv(\CY_1)\to \Shv(\CY_2)$$ 
that satisfies base change against !-pullbacks, and which is the left adjoint of $f^!$ is $f$ is proper. 

\medskip

Following \cite[Chapter 7]{GR1}, one can combine the !-pullback and *-pushforward functoriality in saying that the functor 
$\Shv$ uniquely extends to a functor
\begin{equation}  \label{e:Shv corr}
\on{Corr}(\on{PreStk}_{\on{lft}})^{\on{proper}}_{\on{sch},\on{all}}\to \on{DGCat}^{2\on{-Cat}}_{\on{cont}},
\end{equation}
where $\on{Corr}(\on{PreStk}_{\on{lft}})^{\on{proper}}_{\on{qc-sch},\on{all}}$ is the 2-category of correspondences,
whose objects are prestacks $\CY$ locally of finite type, 1-morphisms are diagrams
$$
\CD
\CY_{0,1} @>{g}>>  \CY_0 \\
@V{f}VV  \\
\CY_1
\endCD
$$
with $f$ schematic and quasi-compact, and where 2-morphisms are given by maps $h:\CY'_{0,1} \to \CY''_{0,1}$
that are schematic and proper. 

\sssec{}

The functor 
$$\Shv: \on{PreStk}_{\on{lft}}\to \on{DGCat}_{\on{cont}}$$ 
is \emph{lax} symmetric monoidal, i.e., for $\CY_1,\CY_2\in \on{PreStk}_{\on{lft}}$
we have a natural map
\begin{equation}  \label{e:external product}
\Shv(\CY_1)\otimes \Shv(\CY_2)\to \Shv(\CY_1\times \CY_2),
\end{equation}
given by the external product $\CF_1,\CF_2\mapsto \CF_1\otimes \CF_2$.

\sssec{}  \label{sss:external product}

The functor \eqref{e:external product} is in general not an equivalence. However, we have the following two observations:

\smallskip

\noindent(a) If $\Shv(-)=\Dmod(-)$, then \eqref{e:external product} is an equivalence for $\CY_1,\CY_2\in \affSch_{\on{ft}}$. 
This formally implies that \eqref{e:external product} is an equivalence any time either $\Shv(\CY_1)$ or $\Shv(\CY_2)$
is dualizable as a DG category, see \cite[Chapter 3, Proposition 3.1.7]{GR1}. In particular, this is the case
when $\CY_i$ is a scheme of finite type (or more generally, an ind-scheme, see \secref{sss:indsch} below).  

\medskip

\noindent(b) For sheaf theories (i) and (ii) in \secref{sss:shvs}, for $\CY_i=S_i\in \affSch_{\on{ft}}$, 
the functor \eqref{e:external product} sends compacts to compacts and is fully faithful. This formally implies 
that the same remains true for the functor \eqref{e:external product} for prestacks any time either 
$\Shv(\CY_1)$ or $\Shv(\CY_2)$ is dualizable as a DG category. 

\ssec{Ind-schemes}

\sssec{}   \label{sss:indsch}

We will be particularly interested in evaluating the functor \eqref{e:shvs on prestacks} on
$$\on{IndSch}_{\on{lft}}\subset \on{PreStk}_{\on{lft}}.$$

Recall that an object $\CY\in \on{PreStk}_{\on{lft}}$ is an ind-scheme if it can be written as a \emph{filtered} colimit
\begin{equation} \label{e:present ind-scheme}
\CY\simeq \underset{i}{\on{colim}}\, Y_i,
\end{equation}
where $Y_i$ are schemes of finite type, and for every arrow $i\overset{\alpha}\longrightarrow j$ in the category of indices, the corresponding map
$Y_i\overset{f_\alpha}\longrightarrow Y_j$ is a closed embedding.  

\sssec{}

Given an ind-scheme $\CY$, we can consider \emph{category} of its presentations as \eqref{e:present ind-scheme}. This category 
is contractible. In fact, it has a final object, where $I$ is the category of all closed subfunctors of $\CY$. 

\medskip

So, any two presentations of $\CY$ as \eqref{e:present ind-scheme} are essentially equivalent.

%\medskip
%
%\noindent NB: although the definitions of ind-schemes requires the index category to be filtered, for what follows this
%condition is not important. The corresponding more general objects are called in \cite{Ga6} ``pseudo-schemes". 

\sssec{}  \label{sss:limits and colimits}

Recall the following general phenomenon. Let 
$$I\to \on{DGCat}_{\on{cont}}, \quad i\mapsto \bC_i$$ 
be a functor, where $I$ is some index category. Suppose that for every arrow $(i\to j)\in I$, the corresponding functor
$\bC_i\to \bC_j$ admits a continuous right adjoint. By passing to the right adjoints we obtain another functor
$$I^{\on{op}}\to \on{DGCat}_{\on{cont}}.$$

\medskip

Then for every $i_0\in I$, the tautological functor
$$\on{ins}_{i_0}:\bC_{i_0}\to \underset{i\in I}{\on{colim}}\, \bC_i$$
admits a continuous right adjoint. Furthermore, the resulting functor 
$$\underset{i\in I}{\on{colim}}\, \bC_i \to \underset{i\in I^{\on{op}}}{\on{colim}}\, \bC_i$$
is an equivalence. 

\sssec{}  \label{sss:shv on indsch}

We apply the situation of \secref{sss:limits and colimits}, by setting 
$$\bC_i:=\Shv(Y_i),$$
where for $(i\overset{\alpha}\to j)\in I$ we have the functor
$$\Shv(Y_i) \overset{(f_\alpha)_*}\longrightarrow \Shv(Y_j)$$
and its right adjoint
$$\Shv(Y_i) \overset{(f_\alpha)^!}\longleftarrow \Shv(Y_j).$$

Hence, we obtain that $\Shv(\CY)$, which is, by definition, given as
$$\underset{i\in I^{\on{op}}}{\on{lim}}\, \Shv(Y_i),$$
with respect to !-pullbacks, can also be written as
\begin{equation} \label{e:categ on indsch as colim}
\underset{i\in I}{\on{colim}}\, \Shv(Y_i),
\end{equation} 
with respect to *-pushforwards. 

\medskip

In particular, $\Shv(\CY)$ is compactly generated by the essential images of $\Shv(Y_{i_0})_c$ under the tautological
functors
$$\on{ins}_{i_0}:\Shv(Y_{i_0})\to \Shv(\CY).$$

\sssec{}  \label{sss:duality}

If $\CY$ is an ind-scheme, the category $\Shv(\CY)$, being compactly generated, is also dualizable. However,
\secref{sss:shv on indsch} implies that $\Shv(\CY)$ is canonically self-dual. This self-duality can be described 
in the following equivalent ways:

\medskip

\noindent(i) Under the identifications 
$$\Shv(\CY)^\vee\simeq \Shv(\CY) \text{ and }\Shv(Y_i)^\vee\simeq \Shv(Y_i),$$ 
the functor dual to $\on{ins}_i:\Shv(Y_i)\to \Shv(\CY)$ is the evaluation functor $\Shv(\CY)\to \Shv(Y_i)$.

\medskip

\noindent(ii) The functors $\on{ins}_i$ are compatible with the contravariant equivalences
$$\BD_\CY:\Shv(\CY)^{\on{op}}_c\to \Shv(\CY)^{\on{op}}_c \text{ and } \BD_{Y_i}:\Shv(Y_i)_c\to \Shv(Y_i)_c.$$

\medskip

\noindent(iii) The pairing $\Shv(\CY)\otimes \Shv(\CY)\to \Vect$ is given, in terms of \eqref{e:categ on indsch as colim},
by the compatible family of pairings $\Shv(Y_i)\otimes \Shv(Y_i)\to \Vect$, corresponding to the usual self-duality of
each $\Shv(Y_i)$.

%\medskip

%\noindent(iv) The unit object in $\Shv(\CY)\otimes \Shv(\CY)$ is given, in terms of \eqref{e:categ on indsch as colim}, by
%$$\underset{i\in I}{\on{colim}}\, (\on{ins}_i\otimes \on{ins}_i)((\Delta_{Y_i})_*(\omega_{Y_i})),$$
%where we recall that 
%$$(\Delta_{Y_i})_*(\omega_{Y_i})\in \Shv(Y_i\times Y_i)\simeq  \Shv(Y_i)\otimes \Shv(Y_i)$$
%is the unit for the self-duality on $\Shv(Y_i)$. 

\sssec{}  \label{sss:t ind-sch}

Let $\CY$ be an ind-scheme. The presentation of $\Shv(\CY)$ as in \eqref{e:categ on indsch as colim}
shows that it carries a unique t-structure compatible with colimits, characterized by the property that the
functors 
$$\on{ins}_i:\Shv(Y_i)\to \Shv(\CY)$$
are t-exact.

\medskip

Equivalently, an object of $\Shv(\CY)$ is coconnective if and only if its restriction to every $Y_i$ is coconnective
as an object of $\Shv(Y_i)$. 

\ssec{Categories acted on by groups}

\sssec{}  \label{sss:definition of action}

Let $H$ be an algebraic group (of finite type). We consider $\Shv(H)$ as a monoidal category, where the monoidal 
operation is \emph{convolution}, i.e., 
$$\Shv(H)^{\otimes n}\overset{\boxtimes}\longrightarrow \Shv(H^n) \overset{\on{mult}^n_*}\longrightarrow \Shv(H).$$

\medskip

Verdier duality on $H$ defines an equivalence $\Shv(H)^\vee\simeq \Shv(H)$. This gives $\Shv(H)$ a comonoidal
structure. This comonoidal structure can be expressed in terms of the pullback
\begin{equation} \label{e:mult pullback}
\Shv(H) \overset{(\on{mult}^n)^!}\longrightarrow \Shv(H^n)
\end{equation}
as follows.

\medskip

The co-tensor product
$$\Shv(H) \to \Shv(H)^{\otimes n}$$
is the composite of \eqref{e:mult pullback} and the functor 
$$\Shv(H^n)\to \Shv(H)^{\otimes n}$$
right adjoint to the fully faithful functor
$\Shv(H)^{\otimes n}\to \Shv(H^n)$, see \secref{sss:external product}. 

\sssec{}  

By an action of $H$ on a DG category $\bC$ we shall mean an action on $\bC$ of the monoidal category
$\Shv(H)$. We denote the $\infty$-category of DG categories acted on by $H$ by $H\mmod$. 

%\medskip
%
%An important feature of the group-theoretic situation is the following:
%
%\begin{lem} \label{l:adj inv}
%Let $F:\bC_1\to \bC_2$ be a functor between categories acted on by $H$. Suppose that $F$ is equipped with a structure
%of lax/co-lax compatibility with the action of $H$. Then this compatibility is actually strict.
%\end{lem}
%
%\begin{proof}
%
%We can think of the datum of $H$-action on $\bC$ as a compatible family of
%actions of abstract groups $\Hom(S,H)$ on $\Shv(S)\otimes \bC$, for $S\in \affSch_{\on{ft}}$. 
%
%\medskip
%
%A structure of lax compatibility on $F$ means that for every such $S$ and $g\in \Hom(S,X)$, we have a 
%natural transformation
%$$g\cdot F(c) \overset{\alpha_{g,c}}\longrightarrow  F(g\cdot c), \quad c\in \bC_1.$$
%
%We need to show that the above map is actually an isomorphism. Let us write its inverse. To do so, it is sufficient
%to write the inverse of the map
%$$F(c)\simeq g^{-1}\cdot (g\cdot F(c))\to g^{-1}\cdot F(g\cdot c).$$
%
%The latter is given by
%$$g^{-1}\cdot F(g\cdot c) \overset{\alpha_{g^{-1},g\cdot c}}\longrightarrow F(g^{-1}\cdot (g\cdot c))\simeq F(c).$$
%
%\end{proof} 

\sssec{An example}  \label{sss:geom action}

Let $\CY$ be a prestack acted on by $H$. Then the operation of pushforward along the action map
defines an action of $H$ on 
$\Shv(\CY)$. 

\sssec{}   \label{sss:inv}

We shall say that an action is trivial if it factors through the augmentation 
$$\Shv(H)\to \Vect,\quad \CF\mapsto \on{C}^\cdot(H,\CF).$$

\medskip

Unless specified otherwise, we will regard $\Vect$ as equipped with the trivial action of $H$.

\medskip

For $\bC\in H\mmod$ we let 
$$\bC^H=\on{Funct}_{H\mmod}(\Vect,\bC).$$

Equivalently, using the self-duality of $\Shv(H)$, we can rewrite $\bC^H$ as the totalization of the cosimplicial
category $\bC^\bullet$ with terms
$$\bC^n:=\Shv(H)^{\otimes n}\otimes \bC.$$

\sssec{}   \label{sss:Averaging functors}

The cosimplicial DG category $\bC^\bullet$ of \secref{sss:inv} satisfies the \emph{comonadic Beck-Chevalley condition};
see \cite[Defn. C.1.3]{Ga7} for what this means. In particular, this implies that the forgetful functor
$$\oblv_H:\bC^H\to \bC$$
admits a right adjoint, to be denoted $\on{Av}^H_*$, and $\bC^H$ identifies with comodules in $\bC$ for the comonad
$\on{Av}^H_*\circ \oblv_H$. 

\medskip

It follows formally that the endo-functor $\on{Av}^H_*\circ \oblv_H$ of $\bC$ is given by
$$c\mapsto \sfe_H\star c.$$

\sssec{}  \label{sss:inv geom}

Let us return to the example of \secref{sss:geom action}. Note that one can give an a priori different definition 
of the category $\Shv(\CY)^H$, namely by setting it be equal to 
$$\Shv(\CY/H):=\on{Tot}(\CY_\bullet), \quad \CY_n=H^n\times \CY.$$

We claim, however, that the two definitions agree. Indeed, the right adjoints to the fully faithful functors
$$\Shv(H)^{\otimes n}\otimes \Shv(\CY)\to \Shv(H^n\times \CY)$$
define a map of cosimplicial categories
\begin{equation} \label{e:two versions tot}
\Shv(\CY_\bullet)\to \Shv(\CY)^\bullet,
\end{equation}
and we claim that the functor \eqref{e:two versions tot} induces an equivalence on totalizations:

\begin{proof}

One shows that the forgetful functor
$$\Shv(\CY/H)\to \Shv(\CY)$$
is also comonadic, and the map \eqref{e:two versions tot} induces
an isomorphism of the corresponding comonads. 

\end{proof} 

\sssec{}

Note that in the situation of \secref{sss:inv geom}, the functor $\on{Av}^H_*$ identifies with direct image 
along the map $\CY\to \CY/H$. Hence, the functor $\oblv_H$ identifies with its left adjoint, which is the functor
of *-pullback along the above map.

\medskip

The endo-functor $\on{Av}^H_*\circ \oblv_H$ of $\Shv(\CY)$ is given by
$$\CF\mapsto \on{act}_*\circ \on{pr}^*(\CF),$$
where 
$$\on{pr},\on{act}:H\times \CY\to \CY$$
are the projection and the action maps, respectively. 

\sssec{} \label{sss:coinvariants and limits}

Although we will not need this in the main body of the paper, we remark that the functor
$$\bC\mapsto \bC^H, \quad H\mmod\to \on{DGCat}_{\on{cont}}$$
has the following (non-tautological property): it commutes with colimits, see \corref{c:inv and colimits} for a proof. 

\medskip

This formally implies that if $\bC$ is dualizable as a plain category, then so is $\bC^H$. Moreover, we have a 
canonical identification
$$(\bC^H)^\vee\simeq (\bC^\vee)^H,$$
so that the dual of the functor $\oblv_H:\bC^H\to \bC$ is the functor $\on{Av}^H_*:\bC^\vee\to (\bC^\vee)^H$, and
the dual of the functor $\on{Av}^H_*:\bC\to \bC^H$ is the functor $\oblv_H:(\bC^\vee)^H\to \bC^\vee$. 

\sssec{}

The category $H\mmod$ carries a natural symmetric monoidal structure that commutes with
the forgetful functor to $\on{DGCat}_{\on{cont}}$. 

\medskip

Namely, for $\bC_1,\bC_2\in H\mmod$ we define the action of $\Shv(H)$ on $\bC_1\otimes \bC_2$ by precomposing
the natural action of $\Shv(H)\otimes \Shv(H)$ on $\bC_1\otimes \bC_2$ with the monoidal functor
$$\Shv(H) \overset{\Delta_*}\longrightarrow \Shv(H\times H) \to \Shv(H)\otimes \Shv(H),$$
where the second arrow is the right adjoint to $\Shv(H)\otimes \Shv(H)\to \Shv(H\times H)$. 

\ssec{The twisted case}

\sssec{Character sheaves} 

Let $\CL$ be a 1-dimensional local system on $H$, which is character sheaf, i.e., we have an isomorphism
$$\on{mult}^*(\CL)\simeq \CL\boxtimes \CL,$$
that behave associatively. 

\sssec{}

Given a character sheaf $\CL$, the functor
$$\Shv(H)\to \Vect, \quad \CF\mapsto \on{C}^\cdot(H,\CF\otimes \CL)$$
has a natural monoidal structure. 

\medskip

This defines on $\Vect$ a different structure
of category acted on by $H$; we denote the resulting object of $H\mmod$ by $\Vect_\CL$. 

\sssec{} 

For $\bC\in H\mmod$, we can twist the initial action of $H$ on $\bC$ by considering the object
$$\bC_\CL:=\bC\otimes \Vect_\CL\in H\mmod.$$ 

Explicitly, $\bC_\CL$ identifies with $\bC$ as a plain DG category, but the new action is given by
the formula
$$\CF\overset{\on{new}}\star c:= (\CF\otimes \CL) \overset{\on{old}}\star c.$$

\sssec{}  \label{sss:with coeff}

We denote
$$\bC^{H,\CL}:=\on{Funct}_{H\mmod}(\Vect_\CL,\bC)\simeq (\bC_{\CL^{-1}})^H.$$

We let
$$\oblv_{H,\CL}:\bC\rightleftarrows \bC^{H,\CL}:\on{Av}_*^{H,\CL}$$
denote the resulting adjoint pair.

\medskip

Note that the endo-functor $\oblv_{H,\CL}\circ \on{Av}_*^{H,\CL}$ of $\bC$ is 
given by
$$c\mapsto \CL^{-1}\star c.$$

\sssec{}

The basic examples of character sheaves are the Kummer sheaf on $\BG_m$ and the Artin-Schreier sheaf
on $\BG_a$, denoted $\on{A-Sch}$, the latter being more relevant for this paper.

\medskip

A priori, $\on{A-Sch}$ is defined either for $\Shv(-)=\Dmod$ in the guise of the exponential D-module,
or in the context of $\ell$-adic sheaves when the ground field $k$ has positive characteristic, in which
case it depends on the choice of a non-trivial character $\BF_p\to \sfe^\times$. 

\medskip

For a group $H$ and a homomorphism $\chi:H\to \BG_a$, we will often write 
$$\bC^{H,\chi}$$
instead of
$$\bC^{H,\chi^*(\on{A-Sch})}.$$

\ssec{Characteristic $0$ situation}

\sssec{}

The Artin-Schreier sheaf does not exist as a constructible sheaf if the ground field $k$ has 
characteristic $0$. E.g., it does not exist for the sheaf theory (iii) of \secref{sss:shvs}. So
for $\bC\in H\mmod$, and a homomorphism $\chi:H\to \BG_m$, the notation $\bC^{H,\chi}$
does not make sense. 

\medskip

However, one \emph{can} define a category equivalent to $\bC^{H,\chi}$, given some additional
data. 

\sssec{}

First, let us replace $\bC$ by $\bC^{H'}$, where $H':=\on{ker}(\chi)$, so we can assume that
we are dealing with a category acted on by $\BG_a$ itself. 

\medskip

Assume now that the $\BG_a$-action on $\bC$ extends to an action of the 
semi-direct product $\BG_m\ltimes \BG_a$, where $\BG_m$ acts on $\BG_a$ my dilations. 

\medskip

Consider the full subcategory
$$\on{ker}(\on{Av}^{\BG_a}_*)=:\bC' \subset \bC,$$
and set
$$\on{Kir}(\bC):=(\bC')^{\BG_m}.$$
(Here ``Kir" is a short-hand for the ``Kirillov model".)

\medskip

We have a tautological forgetful functor
$$\on{Kir}(\bC)\to \bC^{\BG_m},$$
which admits a left adjoint, given by
$$c\mapsto \on{Cone}(\on{Av}^{\BG_a}_*(c)\to c).$$

\sssec{}

We claim that $\on{Kir}(\bC)$ is a substitute of 
$$\Whit(\bC):=\bC^{\BG_a,\on{A-Sch}}$$
for all practical purposes. 

\medskip

For instance, in the situation when $\on{A-Sch}$ is defined, we claim that there exists a canonical
equivalence 
\begin{equation}  \label{e:Whit vs Kir}
\Whit(\bC)\simeq \on{Kir}(\bC).
\end{equation}

Namely, the functor $\to$ in \eqref{e:Whit vs Kir} is defined by
$$c\mapsto \on{Av}^{\BG_m}_*\circ \oblv_{\BG_a,\on{A-Sch}}(c).$$

The functor $\leftarrow$ in \eqref{e:Whit vs Kir} is defined by
$$c'\mapsto \on{Av}^{\BG_a,\on{A-Sch}}_!\circ \oblv_{\BG_m}(c'),$$
where one can show that $\on{Av}^{\BG_a,\on{A-Sch}}_!$ is defined and isomorphic to
$\on{Av}^{\BG_a,\on{A-Sch}}_*[2]$ on the essential image of $\oblv_{\BG_m}$.

\section{The local Whittaker category}  \label{s:loc Whit}

In this section we define the local Whittaker category and study its basic properties. 

\ssec{Definition of the local Whittaker category}  \label{ss:inv}

In this subsection we introduce the main object in this paper--the local Whittaker category.  We do this by imposing
equivariance with respect to the group indscheme $\fL(N)$. The ``ind" direction in $\fL(N)$ will cause objects of this 
category to be of substantially infinite-dimensional nature. 

\sssec{}

Consider the ind-scheme $\CY:=\fL(G)/K_n$ as acted on from the left by $\fL(G)$. We define the Whittaker
category $$\Whit(\CY):=\Shv(\CY)^{\fL(N),\chi}\subset \Shv(\CY)$$
to be the full subcategory that consists of $(\fL(N),\chi)$-equivariant objects. 

\medskip

Let us decipher what this means (we will essentially copy the definition from \cite[Sect. 1.2]{Ga1}).

\sssec{}  \label{sss:N alpha}

Recall the presentation \eqref{e:loop N as colim}. We set
\begin{equation} \label{e:intersect alpha} 
\Shv(\CY)^{\fL(N),\chi}:=\underset{\alpha}{\on{lim}}\, \Shv(\CY)^{N^\alpha,\chi},
\end{equation} 
where each $\Shv(\CY)^{N^\alpha,\chi}$ is a full subcategory of $\Shv(\CY)$, and for $(\alpha'\to \alpha'')\in A$, we have
$$\Shv(\CY)^{N^{\alpha''},\chi}\subset \Shv(\CY)^{N^{\alpha'},\chi}$$
as full subcategories in $\Shv(\CY)$. Note that the limit in \eqref{e:intersect alpha} amounts to the intersection
$$\underset{\alpha}\cap\, \Shv(\CY)^{N^\alpha,\chi}$$
as full subcategories in $\Shv(\CY)$.

\medskip

Let us now explain what the subcategories 
$$\Shv(\CY)^{N^\alpha,\chi}\subset \Shv(\CY)$$
are. 

\sssec{}

For a fixed index $\alpha$, the ind-scheme $\CY$, when viewed as equipped with an action of $N^\alpha$,
is naturally an ind-object in the category of schemes equipped with an action of $N^\alpha$.

\medskip

I.e., we can represent $\CY$ as 
\begin{equation}  \label{e:indsch as colim}
\underset{i\in I}{\on{colim}}\, Y_i,\quad Y_i\overset{f_{i,i'}}\longrightarrow Y_{i'}
\end{equation}
where each $Y_i$ is stable under the $N^\alpha$-action and $f_{i,i'}$ are closed embeddings
(automatically compatible with the $N^\alpha$-actions). 

\medskip

We set
$$\Shv(\CY)^{N^\alpha,\chi}:=\underset{i\in I}{\on{lim}}\, \Shv(Y_i)^{N^\alpha,\chi},$$
viewed as a full subcategory of 
$$\Shv(\CY)\simeq \underset{i\in I}{\on{lim}}\, \Shv(Y_i).$$

Thus, it remains to explain what we mean by
$$\Shv(Y_i)^{N^\alpha,\chi}\subset \Shv(Y_i)$$
for each $\alpha$ and $i$, so that for $(i\to i')$, the corresponding functor 
$$\Shv(Y_{i'}) \overset{f_{i,i'}^!}\longrightarrow \Shv(Y_i)$$
sends $\Shv(Y_{i'})^{N^\alpha,\chi}$ to $\Shv(Y_i)^{N^\alpha,\chi}$.

\sssec{}

Recall the presentation \eqref{e:arc N as lim}. 
For any $\beta\in B_{\alpha,i}$, we can consider the corresponding equivariant category
$\Shv(Y_i)^{N^\alpha_\beta,\chi}$. Since $N^\alpha_\beta$ is unipotent, the forgetful functor
$$\Shv(Y_i)^{N^\alpha_\beta,\chi}\to \Shv(Y_i)$$ is fully faithful, and  
for every $(\beta'\to \beta'')\in B_{\alpha,i}$, we have
$$\Shv(Y_i)^{N^\alpha_{\beta'},\chi}=\Shv(Y_i)^{N^\alpha_{\beta''},\chi}$$
as subcategories of $\Shv(Y_i)$. 

\medskip

We set $\Shv(Y_i)^{N^\alpha,\chi}\subset \Shv(Y_i)$ to be $\Shv(Y_i)^{N^\alpha_\beta,\chi}$ 
for some/any $\beta\in B_{\alpha,i}$. 

\sssec{}

Going back, it is clear that for a map $(i\to i')\in I$, the corresponding functor
$$\Shv(Y_{i'}) \overset{f_{i,i'}^!}\longrightarrow \Shv(Y_i)$$
indeed sends $\Shv(Y_{i'})^{N^\alpha,\chi}$ to $\Shv(Y_i)^{N^\alpha,\chi}$.

\medskip

It is also clear that for a map $(\alpha'\to \alpha'')\in A$, we have
$$\Shv(\Gr)^{N^{\alpha''},\chi}\subset \Shv(\Gr)^{N^{\alpha'},\chi}$$
as full subcategories of $\Shv(\Gr)$. 

\medskip

This completes the definition of $\Shv(\CY)^{\fL(N),\chi}$ as a full subcategory of $\Shv(\CY)$. 

\ssec{Structure of the local Whittaker category}

In this subsection we will discuss the very first general properties of the local Whittaker category
$\Whit(\CY)$. 

\sssec{}  \label{sss:Av alpha}

For what follows, let us note that for every fixed $\alpha$, the forgetful functor 
$$\oblv_{N^\alpha,\chi}:\Shv(\CY)^{N^\alpha,\chi}\to \Shv(\CY)$$
admits a \emph{continuous} right adjoint $\Av^{N^\alpha,\chi}_*$. 

\medskip

Let us describe the functor $\Av^{N^\alpha,\chi}_*$ explicitly.
Writing $\CY$ as \eqref{e:indsch as colim}, the functor $\Av^{N^\alpha,\chi}_*$ corresponds to a compatible
family of functors
\begin{equation} \label{e:Av alpha i}
\Av^{N^\alpha,\chi}:\Shv(Y_i)\to \Shv(Y_i)^{N^\alpha,\chi}.
\end{equation}

For every individual $i$, writing $N^\alpha$ as \eqref{e:arc N as lim}, the corresponding functor \eqref{e:Av alpha i}
equals the functor 
\begin{equation} \label{e:Av alpha beta i}
\Av^{N^\alpha_\beta,\chi}_*:\Shv(Y_i)\to \Shv(Y_i)^{N^\alpha_\beta,\chi}=\Shv(Y_i)^{N^\alpha,\chi}.
\end{equation}

\sssec{} \label{sss:*-av}

The functor $\oblv_{\fL(N),\chi}$, being continuous, admits a right adjoint, to be denoted 
$\on{Av}^{\fL(N),\chi}_*$. But the functor $\on{Av}^{\fL(N),\chi}_*$ is \emph{discontinuous}. Explicitly, we have
$$\on{Av}^{\fL(N),\chi}_*(\CF) \simeq \underset{\alpha}{\on{lim}}\, \on{Av}^{N^\alpha,\chi}_*(\CF),$$
where the limit is taken in $\Shv(\CY)$. 

\medskip

Even though each individual functor $\on{Av}^{N^\alpha,\chi}_*$ is continuous, the inverse limit destroys
this property. 

\sssec{} \label{sss:!-av}

Consider again the forgetful functor
$$\oblv_{\fL(N),\chi}:\Shv(\CY)^{\fL(N),\chi}\to \Shv(\CY).$$

As any functor, it admits a partially defined left adjoint\footnote{For a functor $F:\bC\to \bD$, we shall say that
its partially defined adjoint $F^L$ is defined on $d\in \bD$, if the functor $c\mapsto \CHom_\bD(d,F(c))$ is co-representable.
In this case we set $F^L(d)\in \bC$ to be the co-representing object.}, to be denoted 
$\on{Av}^{\fL(N),\chi}_!$.

\medskip

We \emph{do not} claim that $\on{Av}^{\fL(N),\chi}_!$ is defined on all of $\Shv(\CY)$
(however, this is the case when $n=0$).  

\medskip

Nevertheless, it will turn out
that the functor $\on{Av}^{\fL(N),\chi}_!$ \emph{is} defined on a sufficiently large class of objects
of $\Shv(\CY)$ to ensure that the category $\Shv(\CY)^{\fL(N),\chi}$ is well-behaved,
see \thmref{t:enough} below.  In particular, in \secref{sss:proof compact} we will prove:

\begin{thm}  \label{t:comp gen}
The category $\Shv(\CY)^{\fL(N),\chi}$ is compactly generated.
\end{thm}

\sssec{}

Note that the definition of the Whittaker category $\Shv(\CY)^{\fL(N),\chi}$ has a variant 
$$(\Shv(\CY)\otimes \bC)^{\fL(N),\chi},$$
where $\bC$ is an arbitrary DG category. 

\medskip

We have the forgetful functor
$$(\oblv_{\fL(N),\chi}\otimes \on{Id}_\bC):\Shv(\CY)^{\fL(N),\chi}\otimes \bC \to \Shv(\CY)\otimes \bC,$$
whose essential image is easily seen to belong to $(\Shv(\CY)\otimes \bC)^{\fL(N),\chi}$.
Hence, we obtain a functor:
\begin{equation} \label{e:Whit and tens}
F_\bC :\Shv(\CY)^{\fL(N),\chi}\otimes \bC\to (\Shv(\CY)\otimes \bC)^{\fL(N),\chi}.
\end{equation}

In \secref{ss:enough} we will prove:

\begin{thm}  \label{t:Whit and tens}
The functor $F_\bC$ of \eqref{e:Whit and tens} is an equivalence for any $\bC$.
\end{thm} 

\noindent \emph{Warning}: the assertion of \thmref{t:Whit and tens} is not at all tautological. 

\sssec{}

Recall (see \secref{sss:t ind-sch}) that the category $\Shv(\CY)$ is equipped with a t-structure. A feature that makes
$\Shv(\CY)^{\fL(N),\chi}$ ``very non-classical" is that the objects of this subcategory are ``invisible" from the point
of view of this t-structure. Namely, we will prove: 

\begin{prop} \label{p:invis}
Every $\CF\in \Shv(\CY)^{\fL(N),\chi}$ is \emph{infinitely connective}, i.e., lies in $(\Shv(\CY))^{\leq -n}$ for every $n$.
\end{prop}

\ssec{A stratification}  \label{ss:on stratum}

The stratification of the affine Grassmannian $\Gr_{G,x}$ by $\fL(N)$-orbits gives rise to a stratification of $\CY$. 
This will define a stratification of $\Whit(\CY)$ but some more easily understood categories. 

%In order to develop a ``feel" for what $\Shv(\CY)^{\fL(N),\chi}$ looks like, we will describe its variant 
%on certain locally closed ind-subschemes of $\CY$.

\sssec{}   \label{sss:strata Y}

Consider the projection
\begin{equation} \label{e:proj to Gr}
\CY\to \fL(G)/K_0=\Gr_{G,x}.
\end{equation}

Recall that $\fL(N)$-orbits on $\Gr_{G,x}$ are in bijection with elements of the coweight lattice $\Lambda$; for each
$\mu\in \Lambda$, let us denote by $S^\mu$ the corresponding orbit, i.e.,
$$S^\mu=\fL(N)\cdot t^\mu.$$

\medskip

Let $\CY^\mu$ denote the preimage of $S^\mu$ under \eqref{e:proj to Gr}. Let 
$$\iota^\mu:\CY^\mu\hookrightarrow \CY$$
denote the corresponding locally closed embedding. 

\sssec{}

Consider the corresponding full subcategory
$$\Shv(\CY^\mu)^{\fL(N),\chi}\subset \Shv(\CY^\mu),$$
defined in the same way as
$$\Shv(\CY)^{\fL(N),\chi}\subset \Shv(\CY).$$

The functors
$$(\iota^\mu)_*:\Shv(\CY^\mu)\to \Shv(\CY) \text{ and }
(\iota^\mu)^!:\Shv(\CY)\to \Shv(\CY^\mu)$$
restrict to functors on the corresponding $(\fL(N),\chi)$-equivariant subcategories.

\medskip

We will prove:

\begin{prop}  \label{p:generated} \hfill

\smallskip

\noindent{\em(a)} The category
$\Shv(\CY^\mu)^{\fL(N),\chi}$ is non-zero only when $\mu+n\cdot \rho\in \Lambda^+_\BQ$.

\smallskip

\noindent{\em(b)} The category $\Shv(\CY)^{\fL(N),\chi}$ is generated by the essential images of the functors 
$(\iota^\mu)_*$ for $\mu+n\cdot \rho\in \Lambda^+_\BQ$. 

\smallskip

\noindent{\em(c)} An object of $\Shv(\CY)$ belongs to $\Shv(\CY)^{\fL(N),\chi}$ if and only if its !-restriction to
each $\CY^\mu$ belongs to $\Shv(\CY^\mu)^{\fL(N),\chi}\subset \Shv(\CY^\mu)$. 

\end{prop}

\sssec{Example}

For $n=0$, we obtain that the category $\Shv(\CY^\mu)^{\fL(N),\chi}$ is non-zero if and only if $\mu$ is dominant,
and in the latter case it is equivalent to $\Vect$. 

\medskip

When we go from $n$ to $n+1$ more and more strata $\Shv(\CY^\mu)^{\fL(N),\chi}$ become non-zero. The geometric reason for that
is that the stabilizers become smaller. 

\sssec{}  \label{sss:Y mu}

For what follows we will need some more notation: 

\medskip

For a fixed $\mu\in \Lambda$, let $Y^\mu\subset \CY^\mu$ denote the preimage of the point $t^\mu\in \Gr_{G,x}$ 
under \eqref{e:proj to Gr}. Denote 
$$N^\mu:=\on{Ad}_{t^\mu}(\fL^+(N))\subset \fL(N).$$

\medskip

The action of $N^\mu$ on $\CY^\mu$ preserves $Y^\mu$. Moreover, it is 
easy to see that we can find a group ind-scheme $N'\subset \fL(N)$ so that
$$\fL(N)=N^\mu \cdot N', \quad N^\mu\cap N'=\{1\}.$$
Hence, we can identify 
$$\CY^\mu\simeq Y^\mu\times N'.$$

\begin{lem}  \label{l:reduce to fd}
Restriction along $Y^\mu\hookrightarrow \CY^\mu$ defines an equivalence
$$\Shv(\CY^\mu)^{\fL(N),\chi}\simeq \Shv(Y^\mu)^{N^\mu,\chi},$$
so that the forgetful functor
$$\oblv_{\fL(N),\chi}:\Shv(\CY^\mu)^{\fL(N),\chi}\to \Shv(\CY^\mu)$$
identifies with
$$\Shv(Y^\mu)^{N_\mu,\chi}\overset{\oblv_{N^\mu,\chi}} \longrightarrow 
\Shv(Y^\mu)\overset{-\boxtimes \chi^!(\on{A-Sch})}\longrightarrow 
\Shv(Y^\mu\times N')=\Shv(\CY^\mu).$$
\end{lem}

\begin{proof}

We can choose the family of subgroups $N^\alpha$ to be of the form $N^\mu \cdot N'_\alpha$
for $N'_\alpha\subset N'$. We have:
\begin{multline*}
\Shv(\CY^\mu)^{N^\mu\cdot N',\chi}=
\Shv(Y^\mu\times N')^{N^\mu\cdot N',\chi}\simeq
\underset{\alpha}{\on{lim}}\, \Shv(Y^\mu\times N')^{N^\mu\cdot N'_\alpha,\chi}\simeq \\
\simeq \underset{\alpha}{\on{lim}}\, \underset{\alpha_1\geq \alpha}{\on{lim}}\,  
\Shv(Y^\mu\times N'_{\alpha_1})^{N^\mu\cdot N'_\alpha,\chi}
\simeq  \underset{\alpha_1\geq \alpha}{\on{lim}}\, \Shv(Y^\mu\times N'_{\alpha_1})^{N^\mu\cdot N'_\alpha,\chi}
\end{multline*}

Since the diagonal $\{\alpha_1=\alpha\}$ is cofinal in the poset of $\{\alpha_1\leq \alpha\}$, the above limit
identifies with 
$$\underset{\alpha}{\on{lim}}\, \Shv(Y^\mu\times N'_\alpha)^{N^\mu\cdot N'_\alpha,\chi}.$$
Now, it is clear that for each $\alpha$, the restriction functor 
$$\Shv(Y^\mu\times N'_\alpha)^{N^\mu\cdot N'_\alpha,\chi}\to \Shv(Y^\mu)^{N^\mu,\chi}$$
is an equivalence and the forgetful functor
$$\oblv_{N^\mu\cdot N'_\alpha,\chi}: \Shv(Y^\mu\times N'_\alpha)^{N^\mu\cdot N'_\alpha,\chi}\to 
\Shv(Y^\mu\times N'_\alpha)$$
identifies with 
$$\Shv(Y^\mu)^{N_\mu,\chi}\overset{\oblv_{N^\mu,\chi}} \longrightarrow 
\Shv(Y^\mu)\overset{-\boxtimes \chi^!(\on{A-Sch})}\longrightarrow 
\Shv(Y^\mu\times N'_\alpha).$$

\medskip

Hence,
$$\underset{\alpha}{\on{lim}}\, \Shv(Y^\mu\cdot N'_\alpha)^{N^\mu\times N'_\alpha,\chi}\to 
\Shv(Y^\mu)^{N^\mu,\chi}$$
is also an equivalence, and the statement concerning $\oblv_{\fL(N),\chi}$ follows as well.

\end{proof} 

We can now prove \propref{p:invis}:

\begin{proof}[Proof of \propref{p:invis}]

Since each (finite-dimensional) closed subscheme $Y\subset \CY$ intersects finitely many of the strata $\CY^\mu$,
it is enough to show that the !-restriction 
of $\CF$ to any $\CY^\mu$ is infinitely connective.

\medskip

However, this follows immediately from \lemref{l:reduce to fd}, since 
 $$\chi^!(\on{A-Sch})\in \Shv(N')$$
is infinitely connective. Indeed, its further restriction to every $N^\alpha\cap N'$ lives in the (perverse) cohomological
degree equal to $-\dim(N^\alpha\cap N')$.

\end{proof} 

Finally, let us prove \propref{p:generated}:

\begin{proof} 

Suppose $\mu+n\cdot \rho\notin \Lambda^+_\BQ$. Then there exists a simple
root $\check\alpha_i$ such that
$$\mu(\check\alpha_i)<n.$$

Then the subgroup 
$$\BG_a \overset{t^{-1}}\to \fL(\BG_a) \overset{\check\alpha_i}\to \fL(N)$$
belongs to $N^\mu$, and acts trivially on $Y^\mu$. Since the restriction of $\chi$
to the above subgroup is non-trivial, this implies that 
$$\Shv(Y^\mu)^{\BG_a}=0\,\Rightarrow\, \Shv(Y^\mu)^{N^\mu,\chi}=0.$$

This implies point (a) of the proposition using \lemref{l:reduce to fd}.

\medskip

Point (a) formally implies point (b). Indeed, for every connected component of $\CY$, every subset of $\mu$'s for 
which stratum $\CY^\mu$ intersects this component and for which $\mu+n\cdot \rho\in \Lambda^+_\BQ$
has a minimal element with respect to the standard order relation 
$$\mu_1\leq \mu_2\, \Leftrightarrow\, \mu_2-\mu_1\in \Lambda^{\on{pos}}.$$

\medskip

Point (c) follows similarly. 

\end{proof} 

\ssec{Adapted objects} \label{ss:enough}

In this subsection we will describe a procedure to construct a particularly manageable family of compact objects inside $\Shv(\CY)^{\fL(N),\chi}$. 

\sssec{}

We will say that an object $\CF\in \Shv(\CY)$ is ``$(\fL(N),\chi)$-adapted" if:

\begin{itemize}

\item For any DG category $\bC$ and $\bc\in \bC$, the partially defined 
functor $\on{Av}^{\fL(N),\chi}_!$, left adjoint to 
$$\oblv_{\fL(N),\chi}: (\Shv(\CY)\otimes \bC)^{\fL(N),\chi}\to \Shv(\CY)\otimes \bC,$$
is defined on $\CF\otimes \bc\in \Shv(\CY)\otimes \bC$ (this condition is automatic in the constructible contexts);

\medskip

\item The tautological map
$$\on{Av}^{\fL(N),\chi}_!(\CF\otimes \bc)\to F_\bC(\on{Av}^{\fL(N),\chi}_!(\CF)\otimes \bc)$$
(where $F_\bC$ is as in \eqref{e:Whit and tens}), is an isomorphism.

\end{itemize}

\medskip

Since $\oblv_{\fL(N),\chi}$ is continuous, if $\CF\in \Shv(\CY)$ is compact and $(\fL(N),\chi)$-adapted, then
$$\on{Av}^{\fL(N),\chi}_!(\CF)\in \Shv(\CY)^{\fL(N),\chi}$$
is compact.

\medskip

We will prove:

\begin{thm}  \label{t:enough} 
For any DG category $\bC$, the category $(\Shv(\CY)\otimes \bC)^{\fL(N),\chi}$ is generated by objects of the form 
$\on{Av}^{\fL(N),\chi}_!(\CF\otimes \bc)$ with
$\CF\in \Shv(\CY)$ compact and $(\fL(N),\chi)$-adapted.
\end{thm}

We will prove this theorem in \secref{ss:Rask1}. We will deduce it from the simplest part of S.~Raskin's paper \cite{Ras},
namely, Sect. 2.11 of {\it loc. cit.} (i.e., the case of Theorem 2.7.1(1) of {\it loc.cit.} for $m=\infty$). 

\sssec{}   \label{sss:proof compact}

Note that \thmref{t:enough} (for $\bC=\Vect$) immediately implies \thmref{t:comp gen}.

\sssec{}

The rest of this subsection is devoted to the proof of \thmref{t:Whit and tens}. 

\medskip

First off, \thmref{t:enough} readily implies that the essential image of the functor 
$$F_\bC:\Shv(\CY)^{\fL(N),\chi}\otimes \bC\to (\Shv(\CY)\otimes \bC)^{\fL(N),\chi}$$
generates the target category. Hence, it remains to show that $F_\bC$ is fully faithful.

\medskip

For the latter, it suffices to show that for $\CF\in \Shv(\CY)$ compact and $(\fL(N),\chi)$-adapted, any $\bc\in \bC$
and any $\wt\CF\in \Shv(\CY)^{\fL(N),\chi}\otimes \bC$, the map 
\begin{multline}  \label{e:ff tensor}
\CHom_{\Shv(\CY)^{\fL(N),\chi}\otimes \bC}(\on{Av}^{\fL(N),\chi}_!(\CF)\otimes \bc,\wt\CF)\to  \\
\to \CHom_{(\Shv(\CY)\otimes \bC)^{\fL(N),\chi}}(F_\bC(\on{Av}^{\fL(N),\chi}_!(\CF)\otimes \bc),\wt\CF)\simeq \\
\simeq  \CHom_{(\Shv(\CY)\otimes \bC)^{\fL(N),\chi}}(\on{Av}^{\fL(N),\chi}_!(\CF\otimes \bc),\wt\CF)\simeq
\CHom_{\Shv(\CY)\otimes \bC}(\CF\otimes \bc,\oblv_{\fL(N),\chi}(\wt\CF))
\end{multline} 
is an isomorphism. 

\sssec{}

Consider the following general paradigm:

\medskip

Let $\bD$ be a compactly generated DG category, and $\bd \in \bD^c$. For another DG category $\bC$, consider
the (continuous) functor
$$\CHom_{\bD}(\bd,-):\bD\otimes \bC\to \bC, \quad \bd_1\otimes \bc_1\mapsto \CHom_\bD(\bd,\bd_1)\otimes \bc_1.$$

\medskip

We have:

\begin{lem} \label{l:internal Hom}
For any $\bc\in \bC$ and $\wt\bd\in \bD\otimes \bC$,
we have a canonical isomorphism
$$\CHom_{\bD\otimes \bC}(\bd\otimes \bc,\wt\bd)\simeq \CHom_{\bC}(\bc, \CHom_{\bD}(\bd,\wt\bd)).$$
\end{lem} 

\begin{proof}

Follows by interpreting $\bD\otimes \bC$ as 
$$\on{Funct}((\bD^c)^{\on{op}},\bC).$$

\end{proof}

\sssec{}

We apply \lemref{l:internal Hom} to the two sides in \eqref{e:ff tensor}. We obtain that the left-hand side identifies with
$$\CHom_{\bC}(\bc,\CHom_{\Shv(\CY)^{\fL(N),\chi}}(\on{Av}^{\fL(N),\chi}_!(\CF),\wt\CF)),$$
and the right-hand side with 
$$\CHom_{\bC}(\bc,\CHom_{\Shv(\CY)}(\CF,\oblv_{\fL(N),\chi}(\wt\CF))).$$

\medskip

Hence, it remains to show that the two functors $\Shv(\CY)^{\fL(N),\chi}\otimes \bC\to \bC$, given by
$$\wt\CF\mapsto \CHom_{\Shv(\CY)^{\fL(N),\chi}}(\on{Av}^{\fL(N),\chi}_!(\CF),\wt\CF) \text{ and }
\wt\CF\mapsto \CHom_{\Shv(\CY)}(\CF,\oblv_{\fL(N),\chi}(\wt\CF))$$
are isomorphic.

\medskip

For that it suffices to identify the corresponding functors $\Shv(\CY)^{\fL(N),\chi}\times \bC\rightrightarrows \bC$
that send $\CF_1\times \bc_1$ to 
$$\CHom_{\Shv(\CY)^{\fL(N),\chi}}(\on{Av}^{\fL(N),\chi}_!(\CF),\CF_1)\otimes \bc_1 \text{ and }
\CHom_{\Shv(\CY)}(\CF,\oblv_{\fL(N),\chi}(\CF_1))\otimes \bc_1,$$
respectively.

\medskip

Now the assertion follows from the $(\on{Av}^{\fL(N),\chi}_!,\oblv_{\fL(N),\chi})$-adjunction. 

\ssec{Proof of \thmref{t:enough}}  \label{ss:Rask1}

The proof of \thmref{t:enough}, given below, is based on a geometric construction due to S.~Raskin. 

\sssec{}  \label{sss:adolesc}

For $j\geq 1$ let $\overset{\circ}I{}^j$ denote the subgroup of $\fL^+(G)$ consisting of points that belong to the preimage of 
$$\fL^+(N)_j\subset \fL^+(G)_j$$
under the projection
$$\fL^+(G)\to \fL^+(G)_j:=\fL^+(G)/K_j.$$
(I.e., this is the subgroup consisting of points that belong to $N$ modulo $t^j$.) 

\medskip

Note that for $j=1$, the subgroup $\overset{\circ}I{}^1$ is the unipotent radical of the Iwahori subgroup. By convention, for $j=0$
we set $\overset{\circ}I{}^0=\fL^+(G)$. 

\medskip

Denote 
$$I^j:=\on{Ad}_{-j\check\rho(t)}(\overset{\circ}I{}^j)\subset \fL(G).$$

\sssec{}

Consider the intersection
$$I^j\cap \fL(N).$$

We claim that the character $\chi|_{I^j\cap \fL(N)}$ can be canonically extended to all of $I^j$. Namely, $\chi|_{I^j\cap \fL(N)}$
factors through the projection
$$I^j\cap \fL(N) \overset{\on{Ad}_{j\check\rho(t)}}\longrightarrow \overset{\circ}I{}^j\cap \fL(N)=\fL^+(N)\to  \fL^+(N)_j,$$
and we define the sought-for extension (to be denoted also by $\chi$) 
to be the restriction of the resulting character on $\fL^+(N)_j$ along
$$I^j\overset{\on{Ad}_{j\check\rho(t)}}\longrightarrow \overset{\circ}I{}^j\to  \fL^+(N)_j.$$

\sssec{}

For any $j$ we can consider the category $\Shv(\CY)^{I^j,\chi}$, or more generally 
$$(\Shv(\CY)\otimes \bC)^{I^j,\chi}$$
for a test DG category $\bC$.

\medskip

Note that for $j\geq 1$, the group-scheme $I^j$ is pro-unipotent, and so $(\Shv(\CY)\otimes \bC)^{I^j,\chi}$
is a full subcategory of $\Shv(\CY)\otimes \bC$. 

\medskip

We have the functor
\begin{equation} \label{e:from infty to j}
\Av^{I^j,\chi}_*\circ \oblv_{\fL(N),\chi}:(\Shv(\CY)\otimes \bC)^{\fL(N),\chi}\to (\Shv(\CY)\otimes \bC)^{I^j,\chi}.
\end{equation} 

Note that for $j\geq 1$ we have $I^j=(I^j\cap \fL^+(B^-))\cdot (I^j\cap \fL(N))$, and so the above functor identifies with
$$\Av^{I^j\cap \fL^+(B^-)}_*\circ \oblv_{\fL(N),\chi}.$$

\medskip

The functor $\Av^{I^j,\chi}_*\circ \oblv_{\fL(N),\chi}$ considered above has a partially defined left adjoint given by
\begin{equation} \label{e:from j to infty}
\Av^{\fL(N),\chi}_!\circ \oblv_{I^j,\chi}.
\end{equation} 

\sssec{}

We have the following key result due to S.~Raskin (this is the case of $m=\infty$ in Theorem 2.7.1 in \cite{Ras}, which
is the most elementary part of that paper): 

\begin{thm} \label{t:Raskin} 
Any object in the essential image of $\oblv_{I^j,\chi}$ is $(\fL(N),\chi)$-adapted.
\end{thm}

As an immediate corollary, we obtain:

\begin{cor} 
The left adjoint \eqref{e:from j to infty} of \eqref{e:from infty to j} is defined.
\end{cor} 

For completeness, we will sketch the proof of \thmref{t:Raskin} in \secref{s:proof of Rask}.

\sssec{} \label{sss:proof of enough}

Let us now use \thmref{t:Raskin} to prove \thmref{t:enough}. 

\medskip

First off, the category $\Shv(\CY)^{I^j,\chi}$ is compactly generated (e.g., by \cite{DrGa}). 
Moreover, the functor 
$$\oblv_{I^j,\chi}:\Shv(\CY)^{I^j,\chi}\to \Shv(\CY)$$
sends compacts to compacts (being a left adjoint of the continuous functor $\Av^{I^j,\chi}_*$). 

\medskip

Thus, it remains to see that the essential images of the functors \eqref{e:from j to infty} (for all $j$)
generate the category $(\Shv(\CY)\otimes \bC)^{\fL(N),\chi}$. This is equivalent to saying that the
intersection of the kernels of the functors \eqref{e:from infty to j} is zero. 

\medskip

We will take $j\geq 1$. We will show that the intersection of the kernels of the functors
$\Av^{I^j\cap \fL^+(B^-)}_*$ is zero on all of $\Shv(\CY)\otimes \bC$. The latter assertion
is equivalent to the fact that the essential images of the functors
$$\oblv_{I^j\cap \fL^+(B^-)}:(\Shv(\CY)\otimes \bC)^{I^j\cap \fL^+(B^-)}\to \Shv(\CY)\otimes \bC$$
generates $\Shv(\CY)\otimes \bC$. 

\medskip

However, the latter is obvious, as $I^j\cap \fL^+(B^-)$ shrink as $j\to \infty$.

\qed[\thmref{t:enough}]

\section{A dual definition of the local Whittaker category}  \label{s:dual}

In this section we will define another version of the local Whittaker category, by following a 
procedure dual to that used in the definition of $\Whit(\CY)$: instead of invariants we will
use coinvariants. 

\medskip

We will eventually see that the new category, denoted $\Whit(\CY)_{\on{co}}$, is equivalent to
the original $\Whit(\CY)$. But the functor establishing this equivalence will be something non-tautological. 

\ssec{Digression: invariant functors and categorical \emph{coinvariants}}  \label{ss:coinv}

In order to prepare for the dual definition of the local Whittaker category, we will first consider the
finite-dimensional situation. 

\sssec{}

First, let $N'$ be a unipotent group equipped with a character $\chi:N'\to \BG_a$, and an action on a scheme $Y$. 

\medskip

For a DG category $\bC$, we let
\begin{equation} \label{e:inv functors}
\on{Funct}_{\on{cont}}(\Shv(Y),\bC)^{N',\chi}\subset \on{Funct}_{\on{cont}}(\Shv(Y),\bC)
\end{equation}
be the full subcategory that consists of continuous functors $F:\Shv(Y)\to \bC$, for which the natural transformation
$$F\circ \on{Av}^{N',\chi}_*\to F$$
is an isomorphism.

\medskip

The inclusion \eqref{e:inv functors} admits a right adjoint, given by
$$F\mapsto \on{Av}^{N',\chi}_*\circ F.$$

\sssec{}  \label{sss:tensor prod and inv}

Note that using the Verdier self-duality 
$$\Shv(Y)\simeq \Shv(Y)^\vee, \quad \langle \CF,\CF'\rangle:=\Gamma(Y,\CF\sotimes \CF')$$
we can identify 
$$\Shv(Y)\otimes \bC\simeq \on{Funct}_{\on{cont}}(\Shv(Y),\bC), \quad \CF\otimes \bc\mapsto 
(\CF'\mapsto \langle \CF,\CF'\rangle \otimes \bc).$$

\medskip

In terms of this identification, we have
$$(\Shv(Y)\otimes \bC)^{N',-\chi}\simeq 
\on{Funct}_{\on{cont}}(\Shv(Y),\bC)^{N',\chi},$$
where the LHS is understood in the sense of \secref{sss:with coeff}.

\sssec{}

We define the category $\Shv(Y)_{N',\chi}$ to be universal among DG categories $\bC$ equipped with a functor 
$$F:\Shv(Y)\to \bC, \quad F\in \on{Funct}_{\on{cont}}(\Shv(Y),\bC)^{N',\chi}.$$

Denote the resulting universal functor 
$$\Shv(Y)\to \Shv(Y)_{N',\chi}$$
by $p^{N',\chi}$.

\medskip

We claim, however:

\begin{prop}  \label{p:coinv as inv}
There exists a canonical identification of pairs $$(\Shv(Y)_{N',\chi},p^{N',\chi})\simeq (\Shv(Y)^{N',\chi},\on{Av}^{N',\chi}_*).$$ 
\end{prop}

\begin{proof}

We need to establish an equivalence
$$\on{Funct}_{\on{cont}}(\Shv(Y),\bC)^{N',\chi}\simeq \on{Funct}_{\on{cont}}(\Shv(Y)^{N',\chi},\bC),$$
such that the forgetful functor \eqref{e:inv functors} corresponds to
\begin{equation} \label{e:inv again}
\on{Funct}_{\on{cont}}(\Shv(Y)^{N',\chi},\bC) \overset{-\circ \on{Av}^{N',\chi}_*}\longrightarrow \on{Funct}_{\on{cont}}(\Shv(Y),\bC),
\end{equation} 
in a way functorial in $\bC$.

\medskip

Note that \eqref{e:inv again} is fully faithful and admits a right adjoint given by restriction along $\oblv_{N',\chi}$.
Hence, it is enough to show that the corresponding two idempotents on 
$\on{Funct}_{\on{cont}}(\Shv(Y),\bC)$ match up. However, they are both given by pre-composition with 
$\Av^{N',\chi}_*$. 

\end{proof} 

\begin{cor}  \label{c:coinv as inv} \hfill

\smallskip

\noindent{\em(a)}
The composite functor
$$\Shv(Y)^{N',\chi}\overset{\oblv_{N',\chi}}\longrightarrow  \Shv(Y)\overset{p^{N',\chi}}\longrightarrow \Shv(Y)_{N',\chi}$$
is an equivalence. 

\smallskip

\noindent{\em(b)} 
The inverse equivalence, precomposed with $p^{N',\chi}$, identifies with $\Av^{N',\chi}_*$. 

\end{cor}

\begin{proof}
In terms of the identification of \propref{p:coinv as inv}, the functor in question corresponds to the endofunctor
$$\Av^{N',\chi}_*\circ \oblv_{N',\chi}$$
of $\Shv(Y)^{N',\chi}$, which is isomorphic to the identity.
\end{proof} 

\sssec{}  \label{sss:coinv as quot}

The pair $(\Shv(Y)_{N',\chi},p^{N',\chi})$ can be also described as a Verdier quotient. 

\medskip

Namely, it is obtained by taking the quotient 
of $\Shv(Y)$ by the full DG subcategory consisting of annihilated by the functor $\on{Av}^{N',\chi}_*$. 

\ssec{The dual local Whittaker category}

We will now define a dual version of the Whittaker category, to be denoted 
$$\Whit(\CY)_{\on{co}}:=\Shv(\CY)_{\fL(N),\chi}.$$

\sssec{}

For given a DG category $\bC$, we can consider the DG category
$$\on{Funct}_{\on{cont}}(\Shv(\CY),\bC).$$

We define the full subcategory
$$\on{Funct}_{\on{cont}}(\Shv(\CY),\bC)^{\fL(N),\chi}\subset \on{Funct}_{\on{cont}}(\Shv(\CY),\bC),$$
essentially by mimicking the procedure in \secref{ss:inv}: 

\medskip

Namely, for $\fL(N)$ written as in \eqref{e:loop N as colim}, we set
\begin{multline*}
\on{Funct}_{\on{cont}}(\Shv(\CY),\bC)^{\fL(N),\chi}:=
\underset{\alpha}{\on{lim}}\, \on{Funct}_{\on{cont}}(\Shv(\CY),\bC)^{N^\alpha,\chi}\simeq \\
\simeq \underset{\alpha}\cap\, \on{Funct}_{\on{cont}}(\Shv(\CY),\bC)^{N^\alpha,\chi}\subset \on{Funct}_{\on{cont}}(\Shv(\CY),\bC),
\end{multline*}
so we have to make sense of $\on{Funct}_{\on{cont}}(\Shv(\CY),\bC)^{N^\alpha,\chi}\subset \on{Funct}_{\on{cont}}(\Shv(\CY),\bC)$.

\medskip

Using \secref{sss:limits and colimits}, we have
$$\on{Funct}_{\on{cont}}(\Shv(\CY),\bC)=\on{Funct}_{\on{cont}}(\underset{i}{\on{colim}}\, \Shv(Y_i),\bC)\simeq
\underset{i}{\on{lim}}\, \on{Funct}_{\on{cont}}(\Shv(Y_i),\bC),$$
and in terms of this equivalence, we set
$$\on{Funct}_{\on{cont}}(\Shv(\CY),\bC)^{N^\alpha,\chi}=\underset{i}{\on{lim}}\, 
\on{Funct}_{\on{cont}}(\Shv(Y_i),\bC)^{N^\alpha,\chi}
\subset \underset{i}{\on{lim}}\, \on{Funct}_{\on{cont}}(\Shv(Y_i),\bC).$$

Thus, it remains to define
$$\on{Funct}_{\on{cont}}(\Shv(Y_i),\bC)^{N^\alpha,\chi} \subset \on{Funct}_{\on{cont}}(\Shv(Y_i),\bC).$$

We set $$\on{Funct}_{\on{cont}}(\Shv(Y_i),\bC)^{N^\alpha,\chi}:=
\on{Funct}_{\on{cont}}(\Shv(Y_i),\bC)^{N^\alpha_\beta,\chi}\subset \on{Funct}_{\on{cont}}(\Shv(Y_i),\bC),$$
for $N^\alpha$ presented as in \eqref{e:arc N as lim}.

\medskip

This completes the definition of the full subcategry
$$\on{Funct}_{\on{cont}}(\Shv(\CY),\bC)^{\fL(N),\chi}\subset
\on{Funct}_{\on{cont}}(\Shv(\CY),\bC).$$ 

\sssec{}

We are now ready to define $\Shv(\CY)_{\fL(N),\chi}$. Namely, we let it be the universal among DG categories $\bC$ equipped with a functor 
$$F:\Shv(\CY)\to \bC, \quad F\in \on{Funct}_{\on{cont}}(\Shv(\CY),\bC)^{\fL(N),\chi}.$$

\medskip

Denote the resulting universal functor
$$\Shv(\CY)\to \Shv(\CY)_{\fL(N),\chi}$$
by $p^{\fL(N),\chi}$.

\sssec{}

It follows from the definitions that $\Shv(\CY)_{\fL(N),\chi}$ identifies tautologically with the colimit in $\DGCat_{\on{cont}}$ 
$$\underset{\alpha}{\on{colim}}\, \Shv(\CY)_{N^\alpha,\chi},$$
where the colimit is taken in $\DGCat_{\on{cont}}$. 

\sssec{}

Using \secref{sss:coinv as quot}, we can also describe $\Shv(\CY)_{\fL(N),\chi}$ as the quotient of $\Shv(\CY)$
by the full DG subcategory \emph{generated} by objects
$$\{\CF\,|\, \exists \alpha\, \text{ such that } \on{Av}^{N^\alpha,\chi}_*(\CF)=0\}.$$

\ssec{Properties of the dual Whittaker category}

We will now discuss some basic properties of $\Whit(\CY)_{\on{co}}$. We will see that it is really the dual category
of $\Whit(\CY)$ (up to replacing $\chi$ by its inverse). An essential feature of our infinite-dimensional setting
is that the tautological composite functor
$$\Whit(\CY)\to \Shv(\CY)\to \Whit(\CY)_{\on{co}}$$
is identically equal to zero, by stark contrast with the finite-dimensional situation. 

\sssec{}

By \secref{sss:tensor prod and inv}, we obtain:

\begin{lem} \label{l:simple Verdier duality}
For any $\bC$, under the Verdier duality identification
$$\on{Funct}_{\on{cont}}(\Shv(\CY),\bC)\simeq \Shv(\CY)\otimes \bC,$$
the full subcategory
$$\on{Funct}_{\on{cont}}(\Shv(\CY)_{\fL(N),\chi},\bC)\hookrightarrow \on{Funct}_{\on{cont}}(\Shv(\CY),\bC)$$
corresponds to 
$$(\Shv(\CY)\otimes \bC)^{\fL(N),-\chi}\subset \Shv(\CY)\otimes \bC.$$
\end{lem} 

\medskip

Combined with \thmref{t:Whit and tens}:

\begin{cor}  \label{c:dual dualizable} 
The category $\Shv(\CY)_{\fL(N),\chi}$ identifies with the dual of $\Shv(\CY)^{\fL(N),-\chi}$, so that the functor
$$p^{\fL(N),\chi}:\Shv(\CY)\to \Shv(\CY)_{\fL(N),\chi}$$
is the dual of
$$\oblv_{\fL(N),\chi}:\Shv(\CY)^{\fL(N),-\chi}\to \Shv(\CY).$$
\end{cor}

Since $\Shv(\CY)^{\fL(N),-\chi}$ is is compactly generated, we further obtain:

\begin{cor}  \label{c:dual compact}  \hfill

\smallskip

\noindent{\em(a)}
The category $\Shv(\CY)_{\fL(N),\chi}$ is compactly generated.

\smallskip

\noindent{\em(b)} Let $\CF\in \Shv(\CY)^c$ be such that the functor 
$\Av^{\fL(N),-\chi}_!$ is defined on $\CF$. Then 
$$p^{\fL(N),\chi}(\BD^{\on{Verdier}}(\CF))\in \Shv(\CY)_{\fL(N),\chi}$$
is compact, and 
$$\CHom_{\Shv(\CY)_{\fL(N),\chi}}(p^{\fL(N),\chi}(\BD^{\on{Verdier}}(\CF)),-)\simeq
\langle \Av^{\fL(N),-\chi}_!,-\rangle,$$
where 
$$\BD^{\on{Verdier}}:(\Shv(\CY)^c)^{\on{op}}\to \Shv(\CY)^c$$
denotes the Verdier duality functor and $\langle -,-\rangle$ denotes the canonical pairing
$$\Shv(\CY)^{\fL(N),-\chi}\otimes \Shv(\CY)_{\fL(N),\chi}\to \Vect.$$

\end{cor} 

\sssec{}

In order to develop a ``feel" for what $\Shv(\CY)_{\fL(N),\chi}$ is like, let us describe the corresponding category
$\Shv(\CY^\mu)_{\fL(N),\chi}$, where $\CY^\mu$ is, as in \secref{ss:on stratum}.

\medskip

We have the following counterpart of \lemref{l:reduce to fd} (with the same proof):

\begin{lem}  \label{l:reduce to fd co}
We have a canonical identification 
$$ \Shv(\CY^\mu)_{\fL(N),\chi}\simeq \Shv(Y^\mu)_{N^\mu,\chi}$$
so that the projection functor 
$$\Shv(\CY^\mu) \overset{p^{\fL(N),\chi}}\longrightarrow \Shv(\CY^\mu)_{\fL(N),\chi}$$
goes over to 
$$\Shv(\CY^\mu) \simeq \Shv(Y^\mu\times N')   \overset{\sotimes -\chi^!(\on{A-Sch})}\longrightarrow \Shv(Y^\mu\times N')
\to \Shv(Y^\mu) \overset{p^{N^\mu,\chi}}\longrightarrow
\Shv(Y^\mu)_{N^\mu,\chi},$$
where the third arrow is the functor of *-direct image. 
\end{lem} 

\sssec{}

Consider the composite functor
\begin{equation} \label{e:stupid functor}
\Shv(\CY)^{\fL(N),\chi}\overset{\oblv_{\fL(N),\chi}}\longrightarrow 
\Shv(\CY)\overset{p^{\fL(N),\chi}}\longrightarrow  \Shv(\CY)_{\fL(N),\chi}.
\end{equation} 

We note that \emph{contrary} to the finite-dimensional situation described by \corref{c:coinv as inv}(a), we have:

\begin{prop}
The functor \eqref{e:stupid functor} is identically equal to zero.
\end{prop}

\begin{proof}

We have a commutative diagram
$$
\CD
\Shv(\CY^\mu)^{\fL(N),\chi}  @>{\oblv_{\fL(N),\chi}}>>    \Shv(\CY^\mu)  @>{p^{\fL(N),\chi}}>>
\Shv(\CY^\mu)_{\fL(N),\chi}  \\
@V{(\iota^\mu)_*}VV  @V{(\iota^\mu)_*}VV  @VV{(\iota^\mu)_*}V \\
\Shv(\CY)^{\fL(N),\chi}  @>{\oblv_{\fL(N),\chi}}>>    \Shv(\CY)  @>{p^{\fL(N),\chi}}>>
\Shv(\CY)_{\fL(N),\chi}.  \\
\endCD
$$

By \propref{p:generated}(b), it suffices to show that the corresponding functor  
\begin{equation} \label{e:stupid functor lambda}
\Shv(\CY^\mu)^{\fL(N),\chi} \overset{\oblv_{\fL(N),\chi}}\longrightarrow  \Shv(\CY^\mu)
\overset{p^{\fL(N),\chi}}\longrightarrow \Shv(\CY^\mu)_{\fL(N),\chi}
\end{equation}
is zero. 

\medskip

Using Lemmas \ref{l:reduce to fd} and \ref{l:reduce to fd co}, it suffices to show that the functor 
$$\Shv(Y^\mu) \overset{-\boxtimes \chi^!(\on{A-Sch})}\longrightarrow 
\Shv(Y^\mu\times N')  \overset{\sotimes -\chi^!(\on{A-Sch})}\longrightarrow \Shv(Y^\mu\times N')\to
\Shv(Y^\mu)$$
is zero. 

\medskip

However, the latter functor is given by tensoring with
$$\Gamma(N',\omega_{N'})\simeq \underset{\alpha}{\on{colim}}\, \Gamma(N'_\alpha,\omega_{N'_\alpha})\in \Vect,$$
and the latter is zero, as it is infinitely connective. 

\end{proof}

\ssec{The pseudo-identity functor} \label{ss:pseudo-id}

As we have just seen, an analog of \corref{c:coinv as inv}(a) completely fails in our situation: the corresponding composite
functor is identically equal to $0$. 

\medskip

However, we will be able
to salvage \corref{c:coinv as inv}(b). Namely, we will define a (renormalized) analog of the functor 
of *-averaging with respect to $(\fL(N),\chi)$ that would factor through $\Shv(\CY)_{\fL(N),\chi}$ and give rise
to an equivalence $\Shv(\CY)_{\fL(N),\chi}\to \Shv(\CY)^{\fL(N),\chi}$. 

\medskip

The definition of this functor depends on the choice of a lattice $N_0\subset \fL(N)$; a natural such choice is
$N_0=\fL^+(N)$. 

\sssec{} 

Choose a presentation of $\fL(N)$ as in \eqref{e:loop N as colim}. With no restriction of generality,
we can assume that $N_0\subset N^\alpha$ for all $\alpha$. 

\medskip

For each $\alpha$ we consider the endofunctor 
$$\oblv_{N^\alpha,\chi}\circ \Av^{N^\alpha,\chi}_*[2\dim(N^\alpha/N_0)]$$ of $\Shv(\CY)$.
We claim that the assignment
$$\alpha\rightsquigarrow \oblv_{N^\alpha,\chi}\circ \Av^{N^\alpha,\chi}_*[2\dim(N^\alpha/N_0)]$$
lifts to a functor
$$A\to \on{Funct}_{\on{cont}}(\Shv(\CY),\Shv(\CY)),$$
i.e., we have a homotopy-coherent system of natural transformations 
$$\oblv_{N^{\alpha'},\chi}\circ \Av^{N^{\alpha'},\chi}_*[2\dim(N^{\alpha'}/N_0)]\to 
\oblv_{N^{\alpha''},\chi}\circ \Av^{N^{\alpha''},\chi}_*[2\dim(N^{\alpha''}/N_0)]$$
for $N^{\alpha'}\subset N^{\alpha''}$. 

\sssec{}

Namely, in terms of the action of $\Shv(\fL(N))$ on $\Shv(\CY)$
(see \secref{sss:loop group acting}), the functor $$\oblv_{N^\alpha,\chi}\circ \Av^{N^\alpha,\chi}_*$$ is given by convolution with the object
$$\sfe_{N^\alpha}\sotimes \chi^!(\on{A-Sch})\in \Shv(\fL(N^\alpha))\subset \Shv(\fL(N)).$$

Now, we claim that the assignment
$$\alpha\rightsquigarrow \sfe_{N^\alpha}[2\dim(N^\alpha/N_0)]$$
extends to a functor 
$$A\to \Shv(\fL(N)).$$

\medskip

Indeed, the object $\sfe_{N^\alpha}$ is the *-pullback of $\sfe_{N^\alpha/N_0}\in \Shv(\fL(N)/N_0)$ along $\fL(N)\to \fL(N)/N_0$,
while since $N^\alpha/N_0$ is smooth, we have
$$\sfe_{N^\alpha/N_0}\simeq \omega_{N^\alpha/N_0}[-2\dim(N^\alpha/N_0)].$$

Now, the desired functor comes from the functor
$$A\to \Shv(\fL(N)/N_0), \quad \alpha\mapsto \omega_{N^\alpha/N_0}, \quad 
N^{\alpha'}\subset N^{\alpha''} \mapsto (\omega_{N^{\alpha'}/N_0}\to \omega_{N^{\alpha''}/N_0}).$$

\sssec{} \label{sss:ren inv}

We define 
$$\Av^{\fL(N),\chi}_{*,\on{ren}}:=\underset{\alpha\in A}{\on{colim}}\, 
\oblv_{N^\alpha,\chi}\circ \Av^{N^\alpha,\chi}[2\dim(N^\alpha/N_0)].$$

\medskip

We claim that the essential image of $\Av^{\fL(N),\chi}_{*,\on{ren}}$ is contained in the essential image of 
$\oblv_{\fL(N),\chi}$.

\medskip

Indeed, by definition, we need to show that the essential image of $\Av^{\fL(N),\chi}_{*,\on{ren}}$ is contained in the essential image of 
$\oblv_{N_{\alpha'},\chi}$ for every $\alpha'\in A$. However, for every $\CF\in \Shv(\CY)$ and $\alpha'\in A$, 
the objects $\oblv_{N^\alpha,\chi}\circ \Av^{N^\alpha,\chi}[2\dim(N^\alpha/N_0)](\CF)$
belong to the essential image of $\oblv_{N_{\alpha'},\chi}$ for $\alpha\geq \alpha'$. 

\begin{rem}  \label{r:ren av}
One can view $\Av^{\fL(N),\chi}_{*,\on{ren}}$ as a renormalized version of *-averaging with respect to $(\fL(N),\chi)$ in the
following sense:

\medskip  

In the situation of \secref{ss:coinv} (say, for the trivial character), the functor $\Av^{N'}_*$ is given by 
$$\on{act}_*\circ p^*,$$
where
$$\on{act},p:N'\times Y\rightrightarrows Y$$
are the action and the projection maps. Set
$$\Av^{N'}_{*,\on{ren}}:=\on{act}_*\circ p^!.$$

We have:
$$\Av^{N'}_{*,\on{ren}}\simeq \Av^{N'}_*[2\dim(N')].$$

Now, in the situation when $N'$ is a group ind-scheme \emph{of ind-finite type}, the functor $p^*$ makes no sense
(or, rather, defines a pro-object). So we have to use $p^!$, and we get a well-defined functor 
$\Av^{N'}_{*,\on{ren}}$. 

\medskip

When $N'$ is a group ind-scheme not of ind-finite type, such as $\fL(N)$, in order to have
a well-defined $p^!$, we need a choose a lattice $N_0\subset N'$. This leads to the definition of 
$\Av^{\fL(N),\chi}_{*,\on{ren}}$ given above.

\end{rem}

\sssec{}   \label{sss:PsId}

For the same reason as in \secref{sss:ren inv}, we have:
$$\Av^{\fL(N),\chi}_{*,\on{ren}}\in \on{Funct}_{\on{cont}}(\Shv(\CY),\Shv(\CY))^{\fL(N),\chi}.$$

\medskip

Hence, we obtain that the functor $\Av^{\fL(N),\chi}_{*,\on{ren}}$ factors as
$$\Shv(\CY) \overset{p^{\fL(N),\chi}}\longrightarrow \Shv(\CY)_{\fL(N),\chi} \overset{\on{Ps-Id}_{\Whit}}\longrightarrow
\Shv(\CY)^{\fL(N),\chi}.$$
for a uniquely defined functor
$$\on{Ps-Id}_{\Whit}:\Shv(\CY)_{\fL(N),\chi} \to \Shv(\CY)^{\fL(N),\chi}.$$

\sssec{An example}
Consider the functor $\Av^{\fL(N)}_{*,\on{ren}}$ applied to the category $\Shv(\fL(N)/N')$,
where $N'\subset N_0$ is a group-subscheme of finite codimension. We have the canonical
identifications
$$\Vect\simeq \Shv(\fL(N)/N_0)^{\fL(N)}, \quad \sfe\mapsto \omega_{\fL(N)/N_0}$$
and
$$\Shv(\fL(N)/N_0)_{\fL(N)}\simeq \Vect, \quad \CF\mapsto \Gamma(\fL(N)/N_0,\CF).$$
With respect to the above identifications, the functor $\Av^{\fL(N)}_{*,\on{ren}}$, viewed
as an endo-functor of $\Vect$ is the cohomological shift by $[-2(\dim(N_0/N'))]$.

\medskip

By contrast, if we apply the functor $\Av^{\fL(N)}_{*,\on{ren}}$ to $\on{Shv}(\on{pt})\simeq \Vect$, 
we obtain the zero functor.

\sssec{}

We have the following key statement that replaces \corref{c:coinv as inv}(b) in our infinite-dimensional setting:

\begin{thm} \label{t:inv vs coinv}
The functor $\on{Ps-Id}_{\Whit}(\CF)$ is an equivalence.
\end{thm}

\thmref{t:inv vs coinv}, as stated above, is due to S.~Raskin. It had been conjectured by the author
in 2008 and established by him for $n=0$ (unpublished). The case of an arbitrary $n$ and $G=GL_r$ was done by
D.~Beraldo in \cite{Be}. The general case was established by S.~Raskin using a new geometric insight.

\medskip

In this paper we will give an alternative proof of \thmref{t:inv vs coinv}, see \corref{c:inv vs coinv}. 
However, our proof is not altogether disjoint from that of Raskin: we will use the main geometric tool of \cite{Ras}, 
namely the subgroups $I^j$ introduced in \secref{ss:Rask1}.
Yet, we will use only the simplest part of \cite{Ras}, incarnated by \thmref{t:Raskin} (or rather its Ran space version). 

\section{The global Whittaker category}   \label{s:global}

In this section we fix a smooth and complete curve $X$ and a point $x\in X$. We will define a \emph{global} version of the Whittaker
category, using various enhancements of the moduli stack $\Bun_G$ of $G$-bundles on $X$. 
The idea is to mimic the definition of the global Whittaker space in the classical theory of automorphic functions. 

\medskip

We will ultimately prove that the global Whittaker category is equivalent to the local one. The corresponding phenomenon
in the classical theory is that the global Whittaker space splits as the tensor product of local Whittaker spaces. 

\ssec{Drinfeld's compactification}  \label{ss:Drinf comp}

In this subsection we recall the definition of the \emph{Drinfeld compactification}, which is an (ind)-algebraic stack
used in the definition of the global Whittaker category .  

\sssec{}

Let $\BunNbx$ be the version of Drinfeld's compactification
introduced in \cite{Ga3}. Namely, $\BunNbx$ is the prestack that classifies the data of a $G$-bundle $\CP_G$ on $X$ 
equipped with injective maps of coherent sheaves
\begin{equation}  \label{e:Plucker maps}
\kappa^{\clambda}:(\omega^{\frac{1}{2}})^{\langle \clambda,2\rho\rangle}\to \CV^{\clambda}_{\CP_G}(\infty \cdot x), \quad \clambda\in \cLambda^+
\end{equation}
(here $\CV^{\clambda}$ denotes the Weyl module of highest weight $\clambda$), such that the maps $\kappa^{\clambda}$ satisfy
the Pl\"ucker relations, i.e., they define a reduction of $\CP_G$ to $B$ at the generic point of $X$. 

\begin{rem}
When the derived group of $G$ is not simply connected, in addition to the Pl\"ucker relations one imposes another closed condition,
restricting the possible defect of the maps \eqref{e:Plucker maps}, see \cite[Sect. 7]{Sch}. However, for the purposes of defining the global Whittaker
category, the difference is material, as the objects satisfying the Whittaker condition will be supported on the closed substack in question.
\end{rem} 

\sssec{}

For $\mu\in \Lambda$, let
$$(\BunNb)_{\leq\mu\cdot x}\subset\BunNbx$$
be the closed subfunctor where we require that for every $\clambda\in \cLambda^+$, the corresponding map \eqref{e:Plucker maps}
has a pole of order $\leq \langle \mu,\clambda \rangle$, i.e., corresponds to a regular map
\begin{equation}  \label{e:Plucker mu}
(\omega^{\frac{1}{2}})^{\langle \clambda,2\rho\rangle}\to \CV^{\clambda}_{\CP_G}(\langle \mu,\clambda \rangle\cdot x). 
\end{equation}

\medskip

For 
\begin{equation} \label{e:order}
\mu_2\leq \mu_1 \, \Leftrightarrow \mu_2-\mu_1\in \Lambda^{\on{pos}} 
\end{equation} 
we have an inclusion
$$(\BunNb)_{\leq\mu_1\cdot x} \subset (\BunNb)_{\leq\mu_2\cdot x},$$
and
\begin{equation} \label{e:BunNb as union}
(\BunNb)_{\infty\cdot x}\simeq \underset{\mu\in \Lambda}{\on{colim}}\, (\BunNb)_{\leq\mu\cdot x},
\end{equation} 
where $\Lambda^{\on{pos}}$ is understood as a poset with the standard order relation, i.e., \eqref{e:order}. 

\medskip

For each $\mu$, the prestack $(\BunNb)_{\leq\mu\cdot x}$ is an algebraic stack; thus \eqref{e:BunNb as union} shows
that $\BunNbx$ is an ind-algebraic stack.

\begin{rem}
Although the poset $\Lambda$ is \emph{not} filtered, its subset corresponding to
those $\mu$, for which $(\BunNb)_{\leq\mu\cdot x}$ intersects a given connected component of $(\BunNb)_{\infty\cdot x}$, is filtered. 
\end{rem}

\ssec{Stratifications of $\BunNbx$}

In this subsection we review various stratifications of Drinfeld's compactification, which will be used in the analysis 
of the structure of the global Whittaker category. 
 
\sssec{}

We denote by 
$$(\BunNb)_{=\mu\cdot x}\subset (\BunNb)_{\leq\mu\cdot x}$$
the open substack, where we require that for for every $\clambda\in \cLambda^+$, the corresponding map \eqref{e:Plucker maps}
has a pole of order \emph{equal} to $\langle \mu,\clambda \rangle$ at $x$. I.e., the map \eqref{e:Plucker mu} is a bundle map \emph{on a neighborhood}
of $x$. 

\sssec{}

One can further subdivide each $(\BunNb)_{=\mu\cdot x}$ into locally closed substacks, according to the order of vanishing of 
the maps \eqref{e:Plucker maps} away from $x$.

\medskip

Namely, let
$$(\BunNb)_{=\mu\cdot x,\on{good\,elswhr}}\subset (\BunNb)_{=\mu\cdot x}$$
be the open substack where we require that the maps \eqref{e:Plucker maps} do not vanish away from $x$, i.e., 
\eqref{e:Plucker mu} is a bundle map. 

\medskip

For each $\lambda\in \Lambda^{\on{pos}}$, let 
$$(\BunNb)_{=\mu\cdot x,\on{def}=\lambda}\subset (\BunNb)_{=\mu\cdot x}$$
be the locally closed substack where each of the maps \eqref{e:Plucker maps} factors as
$$(\omega^{\frac{1}{2}})^{\langle \clambda,2\rho\rangle}\to 
(\omega^{\frac{1}{2}})^{\langle \clambda,2\rho\rangle}(D)\to \CV^{\clambda}_{\CP_G}(\langle \mu,\clambda \rangle\cdot x),$$
where $D$ is a divisor of degree $\langle \lambda,\check\lambda\rangle$ on $X-x$, and the second map is a bundle map. 

\medskip

We have a well-defined map
$$(\BunNb)_{=\mu\cdot x,\on{def}=\lambda}\to (X-x)^\lambda,$$
where for $\lambda=\Sigma\, n_i\cdot \alpha_i$ (here $\alpha_i$'s are the positive coroots) we have
$$(X-x)^\lambda:=\underset{i}\Pi\, (X-x)^{(n_i)}.$$

\medskip

We have
$$(\BunNb)_{=\mu\cdot x,\on{good\,elswhr}}=(\BunNb)_{=\mu\cdot x,\on{def}=0}$$ and
$$(\BunNb)_{=\mu\cdot x}=\underset{\lambda\in \Lambda^{\on{pos}}}\cup\, (\BunNb)_{=\mu\cdot x,\on{def}=\lambda}.$$

\begin{rem} \label{r:not a substack}
In addition to the locally closed substacks
$$(\BunNb)_{=\mu\cdot x,\on{good\,elswhr}}\subset (\BunNb)_{=\mu\cdot x}$$
for an individual $\mu$, 
can define $(\BunNb)_{\infty\cdot x,\on{good\,elswhr}}$ as a subfunctor of 
$(\BunNb)_{\infty\cdot x}$. 

\medskip

The caveat here is that  
$$(\BunNb)_{\infty\cdot x,\on{good\,elswhr}}\hookrightarrow (\BunNb)_{\infty\cdot x}$$
is \emph{not} a locally closed embedding, and $(\BunNb)_{\infty\cdot x,\on{good\,elswhr}}$ is not
even an algebraic stack. 
\end{rem}

\sssec{}

Note that $(\BunNb)_{=\mu\cdot x}$ is \emph{not} quasi-compact. 

\medskip

Let 
$$(\BunNb)_{=\mu\cdot x,\on{def}\leq \lambda}\subset (\BunNb)_{=\mu\cdot x}$$
be the open substack equal to 
$$\underset{0\leq \lambda'\leq \lambda}\cup\, (\BunNb)_{=\mu\cdot x,\on{def}=\lambda'}.$$

\medskip

Then each $(\BunNb)_{=\mu\cdot x,\on{def}\leq \lambda}$ is quasi-compact. 

\sssec{}

Let $\ul{y}=\{y_1,...,y_m\}$ be a finite collection of points on $X-x$. We define an open subfunctor
$$(\BunNb)_{\infty\cdot x,\on{good\,at\,}\ul{y}}\subset \BunNbx$$
by requiring that the maps  \eqref{e:Plucker maps} do not vanish at the points $y_1,...,y_m$.

\medskip

We will use the notation
$$(\BunNb)_{\leq \mu\cdot x,\on{good\,at\,}\ul{y}}:=(\BunNb)_{\leq\mu\cdot x}\cap 
(\BunNb)_{\infty\cdot x,\on{good\,at\,}\ul{y}},$$
$$(\BunNb)_{=\mu\cdot x,\on{good\,at\,}\ul{y}}:=(\BunNb)_{=\mu\cdot x}\cap 
(\BunNb)_{\infty\cdot x,\on{good\,at\,}\ul{y}},$$
$$(\BunNb)_{=\mu\cdot x,\on{def}=\lambda,\on{good\,at\,}\ul{y}}:=
(\BunNb)_{=\mu\cdot x,\on{def}=\lambda} \cap 
(\BunNb)_{\infty\cdot x,\on{good\,at\,}\ul{y}},$$
etc. 

\medskip

Note that the open subfunctors $(\BunNb)_{\infty\cdot x,\on{good\,at\,}\ul{y}}$ for $\ul{y}$ being singletons $\ul{y}=\{y\}$
cover $(\BunNb)_{\infty\cdot x}$.  We have
$$(\BunNb)_{\infty\cdot x,\on{good\,at\,}\ul{y}}=\underset{i=1,...,m}\cap\, (\BunNb)_{\infty\cdot x,\on{good\,at\,}y_i}.$$

\ssec{Adding a level structure}

In order to find an analog of the Whittaker category on $\CY=\fL(G)/K$ where $K\subsetneq \fL^+(G)$, we will
need to introduce a variant of Drinfeld's compactification that has to do with $G$-level structures at $x$. 

\sssec{}

Let $\Bun_G^{G\on{-level}_{n\cdot x}}$ is the moduli stack of $G$-bundles with structure of level $n$ at $x$.

\medskip

The forgetful map
$$\Bun_G^{G\on{-level}_{n\cdot x}}\to \Bun_G$$
is a $\fL^+_x(G)_n$-torsor. 

\sssec{}

Consider the forgetful map
$$\BunNbx\to \Bun_G,$$
and denote
$$\BunNbx^{G\on{-level}_{n\cdot x}}:=\BunNbx\underset{\Bun_G}\times \Bun_G^{G\on{-level}_{n\cdot x}}.$$

\medskip

We will denote by
$$(\BunNb)_{=\mu\cdot x}^{G\on{-level}_{n\cdot x}}\subset \BunNbx^{G\on{-level}_{n\cdot x}}$$
the corresponding locally closed substack, and similarly for 
$$(\BunNb)_{=\mu\cdot x}^{G\on{-level}_{n\cdot x}},\,\, (\BunNb)_{=\mu\cdot x,\on{good\,elswhr}}^{G\on{-level}_{n\cdot x}},\,\,
(\BunNb)_{=\mu\cdot x,\on{def}=\lambda}^{G\on{-level}_{n\cdot x}},\,\, 
(\BunNb)_{=\mu\cdot x,\on{def}\leq\lambda}^{G\on{-level}_{n\cdot x}},\,\,
(\BunNb)_{\infty\cdot x,\on{good\,at\,}\ul{y}}^{G\on{-level}_{n\cdot x}},$$
etc. 

\sssec{}

Note that for a fixed $\mu\in \Lambda$, we have a well-defined map
\begin{equation}  \label{e:eval x}
(\BunNb)^{G\on{-level}_{n\cdot x}}_{=\mu\cdot x}\to
(\fL^+_x(G)_n/\fL^+_x(N)_n)\overset{\fL^+_x(T)}\times \CP^{\omega^\rho}_T(\mu\cdot x).
\end{equation}

\ssec{Action of the loop groupoid away from the level} 

We will now introduce a key tool needed for the definition of the global Whittaker category: the 
action of the loop group $\fL(N)$ by ``regluing". A feature if this construction is that it takes place
at points of the curve \emph{different} from $x$, which is our point of interest. 

\sssec{}  \label{sss:twist of N}

Given $\ul{y}$ as above, we can consider the usual loop (resp., arc) groups $\fL_{\ul{y}}(N)$, $\fL_{\ul{y}}(B)$ and $\fL_{\ul{y}}(G)$
(resp., $\fL^+_{\ul{y}}(N)$,  $\fL^+_{\ul{y}}(B)$ and $\fL^+_{\ul{y}}(G)$).
However, we will change the notation slightly and will use the above symbols to denote 
certain twists of these objects.

\medskip

Namely, consider the $T$-torsor induced from the line bundle $\omega^{\frac{1}{2}}$ by means of the homomorphism
$2\rho:\BG_m\to T$; denote it $\CP^{\omega^\rho}_T$.  Using a (chosen) splitting $T\to B$, we can consider the $B$- and $G$-torsors
$$\CP^{\omega^\rho}_B:=B\overset{T}\times \omega^\rho \text{ and }
\CP^{\omega^\rho}_G:=G\overset{T}\times \omega^\rho$$
over $X$.  Let $B^{\omega^\rho}$ (resp., $G^{\omega^\rho}$) be the group-scheme of automorphisms of $\CP^{\omega^\rho}_B$ 
(resp., $\CP^{\omega^\rho}_G$). In other words, $B^{\omega^\rho}$ (resp., $G^{\omega^\rho}$) is the inner twist of 
the constant group-scheme with fiber $B$ (resp., $G$) over $X$ by means of $\CP^{\omega^\rho}_B$ (resp., $\CP^{\omega^\rho}_G$). 

\medskip

Let $N^{\omega^\rho}$ be the group-scheme of automorphisms of $\CP^{\omega^\rho}_B$ that
project to the identity automorphism of $\CP^{\omega^\rho}_T$ (in other words, $N^{\omega^\rho}$ is the twist of the constant
group-scheme with fiber $N$ over $X$ by means of the $T$-torsor $\omega^\rho$ using the adjoint action of $T$ on $N$).

\medskip

From now on, we will use the symbol
$\fL^+_{\ul{y}}(N)$ (resp., $\fL_{\ul{y}}(N)$) to denote the group-scheme (resp., group ind-scheme) of 
sections of $N$ over the formal (resp., formal punctured) disc around $\ul{y}$. And similarly for 
$\fL^+_{\ul{y}}(B)$ and $\fL^+_{\ul{y}}(G)$ (resp., $\fL_{\ul{y}}(B)$ and $\fL_{\ul{y}}(G)$). 

\medskip

The above twist is made in order to have a canonical character
$$\chi_{\ul{y}}:\fL_{\ul{y}}(N)\to \BG_a,$$
which is trivial on $\fL^+_{\ul{y}}(N)$.

%\sssec{}

%In what follows we will replace the loop group $\fL(G)$ considered previously by its version attached to $x\in X$,
%and denoted $\fL_x(G)$, and similarly for $\fL^+_x(G)$. We will denote by $K_{n,x}\subset \fL^+_x(G)$
%the $n$-th congruence subgroup. 

%\medskip

%Of course, by choosing a coordinate near $x$ we can identify these twisted versions of the loop group with the
%original ones. 

%\medskip

%The forgetful map
%$$\BunNbx^{G\on{-level}_{n\cdot x}}\to \BunNbx$$
%is a torsor for the group $\fL_x^+(G)_n:=\fL_x^+(G)/K_{n,x}$. 

\sssec{}

Note that a point of $(\BunNb)_{\infty\cdot x,\on{good\,at\,}\ul{y}}$ defines a $B$-torsor on the formal disc around $\ul{y}$, with the 
induced $T$-torsor identified with $\CP^{\omega^\rho}_T$. 

\medskip

Let 
$$(\BunNb)^{N\on{-level}_{\infty\cdot \ul{y}}}_{\infty\cdot x,\on{good\,at\,}\ul{y}}$$
denote the moduli space that classifies the data of a point of $(\BunNb)_{\infty\cdot x,\on{good\,at\,}\ul{y}}$
plus the data of isomorphism of the above $B$-torsor with $\CP^{\omega^\rho}_B$ that induces the identity 
automorphism on $\CP^{\omega^\rho}_T$. 

\medskip

The forgetful map
$$(\BunNb)^{N\on{-level}_{\infty\cdot \ul{y}}}_{\infty\cdot x,\on{good\,at\,}\ul{y}}\to
(\BunNb)_{\infty\cdot x,\on{good\,at\,}\ul{y}}$$
is a $\fL^+_{\ul{y}}(N)$-torsor. 

\medskip

%We will also consider have the corresponding (locally) closed subfunctors
%$$(\BunNb)^{N\on{-level}_{\infty\cdot \ul{y}}}_{=\lambda\cdot x,\on{good\,at\,}\ul{y}}\subset
%(\BunNb)^{N\on{-level}_{\infty\cdot \ul{y}}}_{\leq \lambda\cdot x,\on{good\,at\,}\ul{y}}\subset 
%(\BunNb)^{N\on{-level}_{\infty\cdot \ul{y}}}_{\infty\cdot x,\on{good\,at\,}\ul{y}},$$
%etc.

%\medskip

\sssec{}  \label{sss:extend to loop group}

Denote
$$(\BunNb)^{G\on{-level}_{n\cdot x},N\on{-level}_{\infty\cdot \ul{y}}}_{\infty\cdot x,\on{good\,at\,}\ul{y}}:=
(\BunNb)^{N\on{-level}_{\infty\cdot \ul{y}}}_{\infty\cdot x,\on{good\,at\,}\ul{y}} 
\underset{\Bun_G}\times \Bun_G^{G\on{-level}_{n\cdot x}}.$$

\medskip

A crucial piece of structure is that the $\fL^+_{\ul{y}}(N)$-action on 
$(\BunNb)^{G\on{-level}_{n\cdot x},N\on{-level}_{\infty\cdot \ul{y}}}_{\infty\cdot x,\on{good\,at\,}\ul{y}}$
extends to an action of the group ind-scheme $\fL_{\ul{y}}(N)$.

\medskip

In particular, 
\begin{equation} \label{e:groupoid}
\fL^+_{\ul{y}}(N)\backslash \fL_{\ul{y}}(N)\overset{\fL^+_{\ul{y}}(N)}\times 
(\BunNb)^{G\on{-level}_{n\cdot x},N\on{-level}_{\infty\cdot \ul{y}}}_{\infty\cdot x,\on{good\,at\,}\ul{y}}
\end{equation}
has a natural structure of groupoid acting on $(\BunNb)^{G\on{-level}_{n\cdot x}}_{\infty\cdot x,\on{good\,at\,}\ul{y}}$. 

\sssec{}

Recall the map \eqref{e:eval x}. For future use we note the following:

\begin{lem}   \label{l:skinny strata} 
For a fixed $\mu\in \Lambda$ and $\lambda\in \Lambda^{\on{pos}}$, the group ind-scheme $\fL_{\ul{y}}(N)$
acts transitively along the fibers of the map
\begin{multline*}
(\BunNb)^{G\on{-level}_{n\cdot x},N\on{-level}_{\infty\cdot \ul{y}}}_{=\mu\cdot x,\on{def}=\lambda,\on{good\,at\,}\ul{y}}\to 
(\BunNb)^{G\on{-level}_{n\cdot x}}_{=\mu\cdot x,\on{def}=\lambda,\on{good\,at\,}\ul{y}}\overset{\text{\eqref{e:eval x}}}\longrightarrow \\
\to \left((\fL^+_x(G)_n/\fL^+_x(N)_n)\overset{\fL^+_x(T)}\times \CP^{\omega^\rho}_T(\mu\cdot x)\right)
\times (X-\{x\cup \ul{y}\})^\lambda.
\end{multline*}
\end{lem}

Furthermore, by Riemann-Roch, we have:

\begin{lem}   \label{l:skinny strata bis}
For every integer $k$ there exists a large enough group sub-scheme of $\fL_{\ul{y}}(N)$
such that for $\mu\in \Lambda$ and $\lambda\in \Lambda^{\on{pos}}$ satisfying 
$\langle \mu-\lambda,\check\rho\rangle\leq k$, this subgroup acts transitively along the orbits 
of $\fL_{\ul{y}}(N)$ on $(\BunNb)^{G\on{-level}_{n\cdot x},N\on{-level}_{\infty\cdot \ul{y}}}_{=\mu\cdot x,\on{def}=\lambda,\on{good\,at\,}\ul{y}}$
(i.e., along the fibers of the map in \lemref{l:skinny strata}). 
\end{lem}

\ssec{Definition of the global Whittaker category (with an auxiliary point)} 

Our goal is to define a certain full subcategory
$$\Whit(\BunNbx^{G\on{-level}_{n\cdot x}})\subset \Shv(\BunNbx^{G\on{-level}_{n\cdot x}}).$$

\medskip

We will first do so on the locus $(\BunNb)_{\infty\cdot x,\on{good\,at\,}\ul{y}}^{G\on{-level}_{n\cdot x}}$, 
i.e., we will define 
$$\Whit((\BunNb)_{\infty\cdot x,\on{good\,at\,}\ul{y}}^{G\on{-level}_{n\cdot x}})
\subset \Shv((\BunNb)_{\infty\cdot x,\on{good\,at\,}\ul{y}}^{G\on{-level}_{n\cdot x}}).$$

This will be done by imposing an equivariance condition with respect to the action of the groupoid \eqref{e:groupoid}. 

\sssec{}

Note that the operation of *-pullback defines an equivalence
$$\Shv((\BunNb)_{\infty\cdot x,\on{good\,at\,}\ul{y}}^{G\on{-level}_{n\cdot x}})\simeq 
\Shv((\BunNb)^{G\on{-level}_{n\cdot x},N\on{-level}_{\infty\cdot \ul{y}}}_{\infty\cdot x,\on{good\,at\,}\ul{y}})^{\fL^+_{\ul{y}}(N)}
\subset \Shv((\BunNb)^{G\on{-level}_{n\cdot x},N\on{-level}_{\infty\cdot \ul{y}}}_{\infty\cdot x,\on{good\,at\,}\ul{y}}).$$

\medskip

We define $\Whit((\BunNb)_{\infty\cdot x,\on{good\,at\,}\ul{y}}^{G\on{-level}_{n\cdot x}})$ to be the full subcategory of 
$\Shv((\BunNb)_{\infty\cdot x,\on{good\,at\,}\ul{y}}^{G\on{-level}_{n\cdot x}})$ that maps under the above equivalence to 
$$\Shv((\BunNb)^{G\on{-level}_{n\cdot x},N\on{-level}_{\infty\cdot \ul{y}}}_{\infty\cdot x,\on{good\,at\,}\ul{y}})^{\fL_{\ul{y}}(N),\chi_{\ul{y}}}
\subset 
\Shv((\BunNb)^{G\on{-level}_{n\cdot x},N\on{-level}_{\infty\cdot \ul{y}}}_{\infty\cdot x,\on{good\,at\,}\ul{y}})^{\fL^+_{\ul{y}}(N)}.$$

Let us rewrite this definition in terms that only involve algebro-geometric objects locally of finite type. 

\sssec{}

Let us write 
$$\fL_{\ul{y}}(N)\simeq \underset{\alpha\in A}{\on{colim}}\, N^\alpha_{\ul{y}}$$
as in \eqref{e:loop N as colim}.  With no restriction of generality we can assume that
$$\fL^+_{\ul{y}}(N) \subset  N^\alpha_{\ul{y}},\,\,\forall \alpha.$$

\medskip

First off, we have:
$$\Shv((\BunNb)^{G\on{-level}_{n\cdot x},N\on{-level}_{\infty\cdot \ul{y}}}_{\infty\cdot x,\on{good\,at\,}\ul{y}})^{\fL_{\ul{y}}(N),\chi_{\ul{y}}}=
\underset{\alpha}{\on{lim}}\,
\Shv((\BunNb)^{G\on{-level}_{n\cdot x},N\on{-level}_{\infty\cdot \ul{y}}}_{\infty\cdot x,\on{good\,at\,}\ul{y}})^{N^\alpha_{\ul{y}},\chi_{\ul{y}}},$$
where each 
$$\Shv((\BunNb)^{G\on{-level}_{n\cdot x},N\on{-level}_{\infty\cdot \ul{y}}}_{\infty\cdot x,\on{good\,at\,}\ul{y}})^{N^\alpha_{\ul{y}},\chi_{\ul{y}}}$$
is a full subcategory of 
$$\Shv((\BunNb)^{G\on{-level}_{n\cdot x},N\on{-level}_{\infty\cdot \ul{y}}}_{\infty\cdot x,\on{good\,at\,}\ul{y}})^{\fL^+_{\ul{y}}(N)}
\subset \Shv((\BunNb)^{G\on{-level}_{n\cdot x},N\on{-level}_{\infty\cdot \ul{y}}}_{\infty\cdot x,\on{good\,at\,}\ul{y}}).$$

In particular, the above $\underset{\alpha}{\on{lim}}$ amounts to the intersection of these subcategories. We will now describe
$\Shv((\BunNb)^{G\on{-level}_{n\cdot x},N\on{-level}_{\infty\cdot \ul{y}}}_{\infty\cdot x,\on{good\,at\,}\ul{y}})^{N^\alpha_{\ul{y}},\chi_{\ul{y}}}$
as a full subcategory of $\Shv((\BunNb)_{\infty\cdot x,\on{good\,at\,}\ul{y}}^{G\on{-level}_{n\cdot x}})$. 

\sssec{}

For each $\alpha$ let $N^{-\alpha}_{\ul{y}}\subset N^\alpha_{\ul{y}}$ be a normal subgroup of finite codimension contained in 
$\fL^+_{\ul{y}}(N)$. Then the character $\chi_{\ul{y}}|_{N^\alpha_{\ul{y}}}$ factors through
$$N^\alpha_{\ul{y}}\twoheadrightarrow N^\alpha_{\ul{y}}/N^{-\alpha}_{\ul{y}}.$$

\medskip

Consider the ind-algebraic stack
$$N^{-\alpha}_{\ul{y}}\backslash (\BunNb)^{G\on{-level}_{n\cdot x},N\on{-level}_{\infty\cdot \ul{y}}}_{\infty\cdot x,\on{good\,at\,}\ul{y}};$$
the action of $\fL^+_{\ul{y}}(N)/N^{-\alpha}_{\ul{y}}$ on it extends to an action of $N^\alpha_{\ul{y}}/N^{-\alpha}_{\ul{y}}$.

\medskip

Then we have: 
$$\Shv((\BunNb)^{G\on{-level}_{n\cdot x},N\on{-level}_{\infty\cdot \ul{y}}}_{\infty\cdot x,\on{good\,at\,}\ul{y}})^{N^\alpha_{\ul{y}},\chi_{\ul{y}}}\simeq
\Shv(N^{-\alpha}_{\ul{y}}\backslash (\BunNb)^{G\on{-level}_{n\cdot x},N\on{-level}_{\infty\cdot \ul{y}}}_{\infty\cdot x,\on{good\,at\,}\ul{y}})
^{N^\alpha_{\ul{y}}/N^{-\alpha}_{\ul{y}},\chi_{\ul{y}}},$$
where we identify the RHS with a full subcategory of $\Shv((\BunNb)_{\infty\cdot x,\on{good\,at\,}\ul{y}}^{G\on{-level}_{n\cdot x}})$
as follows:

\medskip

We have:
\begin{multline*} 
\Shv(N^{-\alpha}_{\ul{y}}\backslash (\BunNb)^{G\on{-level}_{n\cdot x},N\on{-level}_{\infty\cdot \ul{y}}}_{\infty\cdot x,\on{good\,at\,}\ul{y}})
^{N^\alpha_{\ul{y}}/N^{-\alpha}_{\ul{y}},\chi_{\ul{y}}}\subset \\
\subset \Shv(N^{-\alpha}_{\ul{y}}\backslash (\BunNb)^{G\on{-level}_{n\cdot x},N\on{-level}_{\infty\cdot \ul{y}}}_{\infty\cdot x,\on{good\,at\,}\ul{y}})
^{\fL^+_{\ul{y}}(N)/N^{-\alpha}_{\ul{y}}},
\end{multline*} 
whereas  *-pullback along 
$$N^{-\alpha}_{\ul{y}}\backslash (\BunNb)^{G\on{-level}_{n\cdot x},N\on{-level}_{\infty\cdot \ul{y}}}_{\infty\cdot x,\on{good\,at\,}\ul{y}}\to 
(\BunNb)_{\infty\cdot x,\on{good\,at\,}\ul{y}}^{G\on{-level}_{n\cdot x}}$$
identifies 
$$\Shv((\BunNb)_{\infty\cdot x,\on{good\,at\,}\ul{y}}^{G\on{-level}_{n\cdot x}})\simeq 
\Shv(N^{-\alpha}_{\ul{y}}\backslash (\BunNb)^{G\on{-level}_{n\cdot x},N\on{-level}_{\infty\cdot \ul{y}}}_{\infty\cdot x,\on{good\,at\,}\ul{y}})
^{\fL^+_{\ul{y}}(N)/N^{-\alpha}_{\ul{y}}}.$$

%\medskip

%This completes the definition of $\Whit((\BunNb)_{\infty\cdot x,\on{good\,at\,}\ul{y}}^{G\on{-level}_{n\cdot x}})$. 

\ssec{Properties of the global Whittaker category (with an auxiliary point)}

In this subsection we will study some basic properties of the category $(\BunNb)_{\infty\cdot x,\on{good\,at\,}\ul{y}}^{G\on{-level}_{n\cdot x}}$;
in particular, its (in)dependence of $\ul{y}$.

\sssec{}

Replacing $(\BunNb)_{\infty\cdot x,\on{good\,at\,}\ul{y}}$ by its (locally) closed substacks
$$(\BunNb)_{=\mu\cdot x,\on{def}=\lambda,\on{good\,at\,}\ul{y}}^{G\on{-level}_{n\cdot x}}\subset 
(\BunNb)_{=\mu\cdot x,\on{good\,at\,}\ul{y}}\subset (\BunNb)_{\leq \mu\cdot x,\on{good\,at\,}\ul{y}}$$
we can similarly define the corresponding full categories
$$\Whit((\BunNb)_{=\mu\cdot x,\on{def}=\lambda,\on{good\,at\,}\ul{y}}^{G\on{-level}_{n\cdot x}})\subset 
\Shv((\BunNb)_{=\mu\cdot x,\on{def}=\lambda,\on{good\,at\,}\ul{y}}^{G\on{-level}_{n\cdot x}});$$
$$\Whit((\BunNb)_{=\mu\cdot x,\on{good\,at\,}\ul{y}})\subset \Shv((\BunNb)_{=\mu\cdot x,\on{good\,at\,}\ul{y}});$$
$$\Whit((\BunNb)_{\leq \mu\cdot x,\on{good\,at\,}\ul{y}})\subset \Shv((\BunNb)_{\leq \mu\cdot x,\on{good\,at\,}\ul{y}}).$$

\medskip

The corresponding !-pullback and *-pushforward functors maps these full subcategories to one another.
Furthermore, since we are dealing with unipotent groups, we have:

\begin{lem} \label{l:Whit via strata y}
An object of $\Shv((\BunNb)_{\infty\cdot x,\on{good\,at\,}\ul{y}}^{G\on{-level}_{n\cdot x}})$ belongs to
$\Whit((\BunNb)_{\infty\cdot x,\on{good\,at\,}\ul{y}}^{G\on{-level}_{n\cdot x}})$ if (and only if) its !-restrictions to
the locally closed substacks $(\BunNb)_{=\mu\cdot x,\on{good\,at\,}\ul{y}}^{G\on{-level}_{n\cdot x}}$
and/or $(\BunNb)_{=\mu\cdot x,\on{def}=\lambda,\on{good\,at\,}\ul{y}}^{G\on{-level}_{n\cdot x}}$
belong to  $\Whit((\BunNb)_{=\mu\cdot x,\on{good\,at\,}\ul{y}}^{G\on{-level}_{n\cdot x}})$
and/or $\Whit((\BunNb)_{=\mu\cdot x,\on{def}=\lambda,\on{good\,at\,}\ul{y}}^{G\on{-level}_{n\cdot x}})$.
\end{lem}

\sssec{}

Next, we have the following stabilization result:

\begin{prop} \label{p:alpha is enough}
For a fixed $\mu\in \Lambda$ there exists a large enough subgroup $N^\alpha_{\ul{y}}\subset \fL_{\ul{y}}(N)$ such that in order to test
that an object of $\Shv((\BunNb)_{\leq \mu\cdot x,\on{good\,at\,}\ul{y}}^{G\on{-level}_{n\cdot x}})$ belongs to 
$\Whit(\Shv((\BunNb)_{\leq \mu\cdot x,\on{good\,at\,}\ul{y}}^{G\on{-level}_{n\cdot x}})$, it sufficient to test that 
its pullback
to $(\BunNb)^{G\on{-level}_{n\cdot x},N\on{-level}_{\infty\cdot \ul{y}}}_{\leq \mu\cdot x,\on{good\,at\,}\ul{y}}$ belongs to
$$\Shv((\BunNb)^{G\on{-level}_{n\cdot x},N\on{-level}_{\infty\cdot \ul{y}}}_{\leq \mu\cdot x,\on{good\,at\,}\ul{y}})^{N^\alpha_{\ul{y}},\chi_{\ul{y}}}.$$
\end{prop}

\begin{proof}

It suffices to show that there exists $N^\alpha_{\ul{y}}\subset \fL_{\ul{y}}(N)$ which acts transitively on every 
$\fL_{\ul{y}}(N)$-orbit on $(\BunNb)^{G\on{-level}_{n\cdot x},N\on{-level}_{\infty\cdot \ul{y}}}_{\leq \mu\cdot x,\on{good\,at\,}\ul{y}}$.

\medskip

The existence of such $N^\alpha_{\ul{y}}\subset \fL_{\ul{y}}(N)$ follows from \lemref{l:skinny strata bis}. 

\end{proof}

\begin{cor}  \label{c:right adj glob}
The inclusions
$$\Whit((\BunNb)_{\infty\cdot x,\on{good\,at\,}\ul{y}}^{G\on{-level}_{n\cdot x}})\hookrightarrow
\Shv((\BunNb)_{\infty\cdot x,\on{good\,at\,}\ul{y}}^{G\on{-level}_{n\cdot x}}),$$
$$\Whit((\BunNb)_{\leq \mu\cdot x,\on{good\,at\,}\ul{y}}^{G\on{-level}_{n\cdot x}})\hookrightarrow
\Shv((\BunNb)_{\leq \mu\cdot x,\on{good\,at\,}\ul{y}}^{G\on{-level}_{n\cdot x}}),$$
$$\Whit((\BunNb)_{=\mu\cdot x,\on{good\,at\,}\ul{y}}^{G\on{-level}_{n\cdot x}})\hookrightarrow
\Shv((\BunNb)_{=\mu\cdot x,\on{good\,at\,}\ul{y}}^{G\on{-level}_{n\cdot x}})$$
all admit continuous right adjoints. These right adjoints commute with the corresponding !-pullback and
*-pushforward functors. 
\end{cor}

\begin{proof}
We have:
$$\Shv((\BunNb)_{\infty\cdot x,\on{good\,at\,}\ul{y}}^{G\on{-level}_{n\cdot x}})\simeq 
\underset{\mu\in \Lambda}{\on{lim}}\, \Shv((\BunNb)_{\leq \mu\cdot x,\on{good\,at\,}\ul{y}}^{G\on{-level}_{n\cdot x}})$$
(with respect to the !-restriction functors), and 
$$\Whit((\BunNb)_{\infty\cdot x,\on{good\,at\,}\ul{y}}^{G\on{-level}_{n\cdot x}})\simeq 
\underset{\mu\in \Lambda}{\on{lim}}\, \Whit((\BunNb)_{\leq \mu\cdot x,\on{good\,at\,}\ul{y}}^{G\on{-level}_{n\cdot x}}).$$

\medskip

So it is enough to prove the assertion of the proposition for a fixed 
$(\BunNb)_{\leq \mu\cdot x,\on{good\,at\,}\ul{y}}^{G\on{-level}_{n\cdot x}}$ and the substacks
$$(\BunNb)_{=\mu'\cdot x,\on{good\,at\,}\ul{y}}^{G\on{-level}_{n\cdot x}}\subset
(\BunNb)_{\leq \mu'\cdot x,\on{good\,at\,}\ul{y}}^{G\on{-level}_{n\cdot x}}\subset 
(\BunNb)_{\leq \mu\cdot x,\on{good\,at\,}\ul{y}}^{G\on{-level}_{n\cdot x}}$$
for $\mu'\leq \mu$.

\medskip

However, now the assertion follows from \propref{p:alpha is enough}: the required right adjoint is given 
on the corresponding
$$N^{-\alpha}_{\ul{y}}\backslash 
(\BunNb)^{G\on{-level}_{n\cdot x},N\on{-level}_{\infty\cdot \ul{y}}}_{\leq \mu\cdot x,\on{good\,at\,}\ul{y}}$$
by the functor $\Av_*^{N^{\alpha}_{\ul{y}}/N^{-\alpha}_{\ul{y}},\chi_{\ul{y}}}$.

\end{proof}

\sssec{}

Let now $\ul{y}$ be equal to $\ul{y}'\sqcup \ul{y}''$. Note that 
$$(\BunNb)^{G\on{-level}_{n\cdot x}}_{\infty\cdot x,\on{good\,at\,}\ul{y}}=
(\BunNb)^{G\on{-level}_{n\cdot x}}_{\infty\cdot x,\on{good\,at\,}\ul{y}'}\cap
(\BunNb)^{G\on{-level}_{n\cdot x}}_{\infty\cdot x,\on{good\,at\,}\ul{y}''}.$$

\medskip

We claim:

\begin{prop}  \hfill \label{p:add point} 

\smallskip

\noindent{\em(a)}
The restriction functor 
$$\Shv((\BunNb)^{G\on{-level}_{n\cdot x}}_{\infty\cdot x,\on{good\,at\,}\ul{y}'})\to
\Shv((\BunNb)^{G\on{-level}_{n\cdot x}}_{\infty\cdot x,\on{good\,at\,}\ul{y}})$$
sends 
$$\Whit((\BunNb)^{G\on{-level}_{n\cdot x}}_{\infty\cdot x,\on{good\,at\,}\ul{y}'})\to
\Whit((\BunNb)^{G\on{-level}_{n\cdot x}}_{\infty\cdot x,\on{good\,at\,}\ul{y}}).$$

\smallskip

\noindent{\em(b)}
The diagram
$$
\CD
\Shv((\BunNb)^{G\on{-level}_{n\cdot x}}_{\infty\cdot x,\on{good\,at\,}\ul{y}'})   @>>> 
\Shv((\BunNb)^{G\on{-level}_{n\cdot x}}_{\infty\cdot x,\on{good\,at\,}\ul{y}})  \\
@VVV   @VVV   \\
\Whit((\BunNb)^{G\on{-level}_{n\cdot x}}_{\infty\cdot x,\on{good\,at\,}\ul{y}'})   @>>> 
\Whit((\BunNb)^{G\on{-level}_{n\cdot x}}_{\infty\cdot x,\on{good\,at\,}\ul{y}})  
\endCD
$$
where the vertical arrows are the right adjoints to the inclusions, also commutes.
\end{prop}

\begin{proof}

To prove point (a), it suffices to show that the action of $\fL_{\ul{y}'}(N)$ along the orbits of 
$\fL_{\ul{y}}(N)$ on $(\BunNb)^{G\on{-level}_{n\cdot x},N\on{-level}_{\infty\cdot \ul{y}}}_{\infty\cdot x,\on{good\,at\,}\ul{y}}$
is transitive. However, this follows from \lemref{l:skinny strata}. 

\medskip

To prove point (b), it is enough to do so for the embedding
$$(\BunNb)^{G\on{-level}_{n\cdot x}}_{\leq \mu\cdot x,\on{good\,at\,}\ul{y}}\hookrightarrow
(\BunNb)^{G\on{-level}_{n\cdot x}}_{\leq \mu\cdot x,\on{good\,at\,}\ul{y}'}.$$

\medskip

Now the assertion follows from \lemref{l:skinny strata bis}: the right adjoint in question is given by the functor
$\Av_*^{N^{\alpha}_{\ul{y}'},\chi_{\ul{y}'}}$ for a large enough subgroup 
$$N^{\alpha}_{\ul{y}'}\subset \fL_{\ul{y}'}(N).$$

\end{proof}

\sssec{}

The above discussion was \emph{not} specific to the fact that we were dealing with a non-degenerate character 
$\chi_{\ul{y}}$; in particular it equally applies to the case when the character is trivial. 

\medskip

However, the following assertion \emph{is} specific to the non-degenerate case:

\begin{lem} \label{l:good strata y} \hfill

\smallskip

\noindent{\em(a)}
Any object of $\Whit((\BunNb)^{G\on{-level}_{n\cdot x}}_{=\mu\cdot x,\on{good\,at\,}\ul{y}})$ supported outside of 
$$(\BunNb)^{G\on{-level}_{n\cdot x}}_{=\mu\cdot x,\on{good\,elswhr}}
\subset (\BunNb)^{G\on{-level}_{n\cdot x}}_{=\mu\cdot x,\on{good\,at\,}\ul{y}}$$
is zero.

\smallskip

\noindent{\em(b)}
The category $\Whit((\BunNb)^{G\on{-level}_{n\cdot x}}_{=\mu\cdot x,\on{good\,at\,}\ul{y}})$ is zero
unless $\mu+n\cdot \rho\in \Lambda^+_\BQ$.

\end{lem}

For the proof see \cite[Lemma 6.2.4]{FGV}.

\ssec{Definition of the global Whittaker category}

In this section we will finally define the sought-after category $\Whit(\BunNbx^{G\on{-level}_{n\cdot x}})$. I.e.,
we will show how to get rid of the auxiliary point(s) $\ul{y}$.  

\sssec{}

We define
$$\Whit(\BunNbx^{G\on{-level}_{n\cdot x}})\subset \Shv(\BunNbx^{G\on{-level}_{n\cdot x}})$$
to be the full subcategory that consists of objects such that their restriction to
$$(\BunNb)^{G\on{-level}_{n\cdot x}}_{\on{good\,at\,}\ul{y}}\subset \BunNbx^{G\on{-level}_{n\cdot x}}$$
belongs to
$$\Whit((\BunNb)^{G\on{-level}_{n\cdot x}}_{\on{good\,at\,}\ul{y}})\subset
\Shv((\BunNb)^{G\on{-level}_{n\cdot x}}_{\on{good\,at\,}\ul{y}})$$
for any finite non-empty collection of points $\ul{y}$.

\medskip

By \propref{p:add point}, it is enough to check this condition for $\ul{y}$ being singletons $\{y\}$. Note also that every quasi-compact 
algebraic substack of $\BunNbx^{G\on{-level}_{n\cdot x}}$ is contained in a union of 
$(\BunNb)^{G\on{-level}_{n\cdot x}}_{\on{good\,at\,}y})$ for finitely many points $y$.

\sssec{}

We define the full subcategories
$$\Whit((\BunNb)^{G\on{-level}_{n\cdot x}}_{\leq \mu\cdot x})\subset \Shv((\BunNb)^{G\on{-level}_{n\cdot x}}_{\leq \mu\cdot x}),$$
$$\Whit((\BunNb)^{G\on{-level}_{n\cdot x}}_{=\mu\cdot x})\subset \Shv((\BunNb)^{G\on{-level}_{n\cdot x}}_{=\mu\cdot x}),$$
$$\Whit((\BunNb)^{G\on{-level}_{n\cdot x}}_{=\mu\cdot x,\on{def}=\lambda})\subset 
\Shv((\BunNb)^{G\on{-level}_{n\cdot x}}_{=\mu\cdot x,\on{def}=\lambda})$$
by the same principle.

\medskip

The corresponding !-pullback and *-pushforward functors map these full subcategories to one another.
From \lemref{l:Whit via strata y} we obtain:

\begin{cor} \label{c:Whit via strata}
An object of $\Shv((\BunNb)_{\infty\cdot x}^{G\on{-level}_{n\cdot x}})$ belongs to
$\Whit((\BunNb)_{\infty\cdot x}^{G\on{-level}_{n\cdot x}})$ if (and only if) its !-restrictions to
the locally-closed substacks $(\BunNb)_{=\mu\cdot x}^{G\on{-level}_{n\cdot x}}$
(resp., $(\BunNb)_{=\mu\cdot x,\on{def}=\lambda}^{G\on{-level}_{n\cdot x}}$)
belong to $\Whit((\BunNb)_{=\mu\cdot x}^{G\on{-level}_{n\cdot x}})$
(resp., $\Whit((\BunNb)_{=\mu\cdot x,\on{def}=\lambda}^{G\on{-level}_{n\cdot x}})$).
\end{cor}

From \propref{p:add point}(b) and \corref{c:right adj glob}, we obtain: 

\begin{cor}  
The inclusions
$$\Whit((\BunNb)_{\infty\cdot x}^{G\on{-level}_{n\cdot x}})\hookrightarrow
\Shv((\BunNb)_{\infty\cdot x}^{G\on{-level}_{n\cdot x}}),$$
$$\Whit((\BunNb)_{\leq \mu\cdot x}^{G\on{-level}_{n\cdot x}})\hookrightarrow
\Shv((\BunNb)_{\leq \mu\cdot x}^{G\on{-level}_{n\cdot x}}),$$
$$\Whit((\BunNb)_{=\mu\cdot x}^{G\on{-level}_{n\cdot x}})\hookrightarrow
\Shv((\BunNb)_{=\mu\cdot x}^{G\on{-level}_{n\cdot x}})$$
all admit continuous right adjoints. These right adjoints commute with the corresponding !-pullback and
*-pushforward functors. 
\end{cor}

We will denote the right adjoint(s) appearing in the above corollary by $\Av_{*,\on{glob}}^{\Whit}$. 

\sssec{}

The above assertions are \emph{not} specific to the fact that we were dealing with a non-degenerate character.
In the non-degenerate case, from \lemref{l:good strata y} we obtain:

\begin{cor} \label{c:good strata} \hfill

\smallskip

\noindent{\em(a)} The restriction functor
$\Whit((\BunNb)^{G\on{-level}_{n\cdot x}}_{\infty\cdot x})\to \Whit((\BunNb)^{G\on{-level}_{n\cdot x}}_{\infty\cdot x,\on{good\,at\,}\ul{y}})$
is an equivalence for any $\ul{y}$.

\smallskip

\noindent{\em(b)}
The category $\Whit((\BunNb)^{G\on{-level}_{n\cdot x}}_{=\mu\cdot x})$ is zero
unless $\mu+n\cdot \rho\in \Lambda^+_\BQ$.

\smallskip

\noindent{\em(c)} The restriction functor
$$\Whit((\BunNb)^{G\on{-level}_{n\cdot x}}_{=\mu\cdot x})\to \Whit((\BunNb)_{=\mu\cdot x,\on{good\,elswhr}})$$
is an equivalence. 

\smallskip

\noindent{\em(d)} An object of $\Shv((\BunNb)_{\infty\cdot x}^{G\on{-level}_{n\cdot x}})$ belongs to
$\Whit((\BunNb)_{\infty\cdot x}^{G\on{-level}_{n\cdot x}})$ if (and only if) for every $\mu$:

\smallskip

\noindent{\em(i)} Its !-restriction to 
$(\BunNb)^{G\on{-level}_{n\cdot x}}_{=\mu\cdot x}-
(\BunNb)^{G\on{-level}_{n\cdot x}}_{=\mu\cdot x,\on{good\,elswhr}}$
is zero;

\smallskip

\noindent{\em(ii)} Its !-restriction to $(\BunNb)^{G\on{-level}_{n\cdot x}}_{=\mu\cdot x,\on{good\,elswhr}}$
belongs to $\Whit((\BunNb)^{G\on{-level}_{n\cdot x}}_{=\mu\cdot x,\on{good\,elswhr}})$. 

\end{cor}

\begin{rem}
In \thmref{t:loc pullback mu}(b) we will give an explicit local description of the categories 
$\Whit((\BunNb)^{G\on{-level}_{n\cdot x}}_{=\mu\cdot x,\on{good\,elswhr}})$.
\end{rem}

\sssec{}

From \corref{c:good strata} we obtain:

\begin{cor}  \hfill  \label{c:glob Whit comp gen}

\smallskip

\noindent{\em(a)} For a given $\mu$, 
every object of $\Whit((\BunNb)^{G\on{-level}_{n\cdot x}}_{\leq \mu\cdot x})$ 
is a clean extension from a quasi-compact substack.

\smallskip

\noindent{\em(b)} The category $\Whit((\BunNb)^{G\on{-level}_{n\cdot x}}_{\infty\cdot x})$ 
is compactly generated; the forgetful functor
$$\Whit((\BunNb)^{G\on{-level}_{n\cdot x}}_{\infty\cdot x})\to 
\Shv((\BunNb)^{G\on{-level}_{n\cdot x}}_{\infty\cdot x})$$ sends compacts to compacts.

\end{cor}

\begin{proof}

Point (a) follows from the fact that for a given $\mu$ the set of $\mu'\in \Lambda$ that satisfy
$$\mu' \in \mu-\Lambda \text{ and } \mu'+n\cdot \rho\in \Lambda^+$$
is finite.

\medskip

To prove point (b), it suffices to show that for a fixed $\mu$, the category
$\Whit((\BunNb)^{G\on{-level}_{n\cdot x}}_{\leq \mu\cdot x,\on{good\,at\,}\ul{y}})$ 
for some/any non-empty $\ul{y}$, is compactly generated. However, by point (a),
the latter is equivalent to the category of (twisted) sheaves on a quasi-compact algebraic stack,
and the assertion follows from \cite{DrGa}.

\end{proof}

\ssec{Duality for the global Whittaker category}

In this subsection we will show that the global Whittaker category (unlike its local counterpart) is
more or less tautologically self-dual, up to replacing $\chi$ by its inverse. 

\sssec{}

As was mentioned above, the algebraic stacks $(\BunNb)^{G\on{-level}_{n\cdot x}}_{\leq \mu\cdot x}$ are not quasi-compact. Hence,
the functor
$$\Gamma((\BunNb)^{G\on{-level}_{n\cdot x}}_{\leq \mu\cdot x},-):\Shv((\BunNb)^{G\on{-level}_{n\cdot x}}_{\leq \mu\cdot x})\to \Vect$$
is not continuous, and we \emph{do not} have a Verdier duality identification of $\Shv((\BunNb)^{G\on{-level}_{n\cdot x}}_{\leq \mu\cdot x})$ 
with its dual.

\medskip

However, if $\CF\in (\BunNb)^{G\on{-level}_{n\cdot x}}_{\leq \mu\cdot x}$ is a *-extension from a quasi-compact substack, the functor 
$$\Gamma((\BunNb)^{G\on{-level}_{n\cdot x}}_{\leq \mu\cdot x},\CF\sotimes -):
\Shv((\BunNb)^{G\on{-level}_{n\cdot x}}_{\leq \mu\cdot x})\to \Vect$$
is continuous.

\medskip

In particular, it follows from \corref{c:glob Whit comp gen} that for 
$\CF\in \Whit((\BunNb)^{G\on{-level}_{n\cdot x}}_{\infty\cdot x})$, the functor 
$$\Gamma((\BunNb)^{G\on{-level}_{n\cdot x}}_{\infty\cdot x},\CF\sotimes -):\Shv((\BunNb)^{G\on{-level}_{n\cdot x}}_{\infty\cdot x})\to \Vect$$
is continuous.

\sssec{}

We claim:

\begin{prop}  \label{p:Verdier for glob Whit}
The category $\Whit((\BunNb)^{G\on{-level}_{n\cdot x}}_{\infty\cdot x})$ is canonically dual to 
a similar category defined using the opposite character; this duality is uniquely defined by the property that 
for $\CF\in \Whit((\BunNb)^{G\on{-level}_{n\cdot x}}_{\infty\cdot x})$ and $\CF'\in 
\Shv((\BunNb)^{G\on{-level}_{n\cdot x}}_{\infty\cdot x})$ we have a functorial isomorphism
\begin{equation} \label{e:Av duality}
\langle \CF,\Av_{*,\on{glob}}^{\Whit}(\CF')\rangle=
\Gamma((\BunNb)^{G\on{-level}_{n\cdot x}}_{\infty\cdot x},\CF\sotimes \CF').
\end{equation} 
\end{prop}

\begin{proof}

It suffices to define a contravariant equivalence between the corresponding subcategories of compact objects.

\medskip

Every compact object $\CF\in \Whit((\BunNb)^{G\on{-level}_{n\cdot x}}_{\infty\cdot x})$ is supported on 
some $(\BunNb)^{G\on{-level}_{n\cdot x}}_{\leq \mu \cdot x}$, and by \corref{c:glob Whit comp gen}(a)
is a clean extension from some quasi-compact open. Hence, $\BD^{\on{Verdier}}(\CF)$ is a compact object
in $\Shv((\BunNb)^{G\on{-level}_{n\cdot x}}_{\leq \mu \cdot x})$ and belongs to 
$\Whit((\BunNb)^{G\on{-level}_{n\cdot x}}_{\leq \mu \cdot x})$ with the opposite character.

\medskip

For $\CF'\in \Shv((\BunNb)^{G\on{-level}_{n\cdot x}}_{\infty\cdot x})$ we have a canonical isomorphism
\begin{multline*}
\Gamma((\BunNb)^{G\on{-level}_{n\cdot x}}_{\infty\cdot x},\CF\sotimes \CF')\simeq
\CHom_{\Shv((\BunNb)^{G\on{-level}_{n\cdot x}}_{\infty\cdot x})}(\BD^{\on{Verdier}}(\CF),\CF')\simeq \\
\simeq \CHom_{\Whit((\BunNb)^{G\on{-level}_{n\cdot x}}_{\infty\cdot x})}(\BD^{\on{Verdier}}(\CF),\Av_{*,\on{glob}}^{\Whit}(\CF'))=:
\langle \CF,\Av_{*,\on{glob}}^{\Whit}(\CF')\rangle,
\end{multline*}
as required. 

\end{proof}

\section{The local vs global comparison}  \label{s:global-to-local}

In this section we will compare the local and global definitions of the Whittaker category. Our main result,
\thmref{t:main}, will say that they are equivalent. 

\ssec{The local-to-global map}

In this subsection we will introduce a map between geometries that will eventually let us compare the local
and the global definitions of the Whittaker category. 

\sssec{}

We will now introduce the twisted versions the group ind-scheme $\fL(N)$ and the ind-scheme 
$\CY=\fL(G)/K_n$ that it acts on.

\medskip

Instead of $\fL(N)$, we will use the group ind-scheme $\fL_x(N)$, defined following the recipe in \secref{sss:twist of N}. We let
$\chi_x$ denote the canonical character on $\fL_x(N)$. 

\medskip

For $\CY$ we will keep the same notation, but we will mean the moduli space of triples
$$(\CP_G,\gamma,\epsilon),$$
where $\CP_G$ is a $G$-bundle on the formal disc around $x$, $\gamma$ is an identification of $\CP_G$ with 
$\CP^{\omega^\rho}_G$ on the formal punctured disc, and $\epsilon$ is a structure of level $n$ at $x$ on $\CP_G$. 

\medskip

The group of automorphisms of $\CP^{\omega^\rho}_G$ on the formal punctured disc, %denoted $\fL_x(G)$, 
acts on $\CY$; in particular we have a $\fL_x(N)$ -action on $\CY$.

\sssec{}  \label{sss:map pi}

Recall that according to the Beauville-Laszlo theorem, the data of $(\CP_G,\gamma)$ in the definition of $\CY$
can be reinterpreted by letting $\CP_G$ be a $G$-bundle over $X$ and $\gamma$ an identification of $\CP_G$ with 
$\CP^{\omega^\rho}_G$ on $X-x$. 

\medskip

The $G$-bundle $\CP^{\omega^\rho}_G$ comes equipped with a tautological Pl\"ucker data \eqref{e:Plucker maps}.
From here we obtain a map
\begin{equation} \label{e:map pi}
\pi:\CY\to \BunNbx^{G\on{-level}_{n\cdot x}}.
\end{equation} 

\sssec{}

Our first goal is to prove:

\begin{thm}  \label{t:loc pullback} \hfill

\smallskip

\noindent{\em(a)}
The functor $\pi^!$ sends $\Whit((\BunNb)^{G\on{-level}_{n\cdot x}}_{\infty\cdot x})$ to $\Whit(\CY)$.

\smallskip

\noindent{\em(b)}
Vice versa, if $\CF\in \Shv((\BunNb)^{G\on{-level}_{n\cdot x}}_{\infty\cdot x})$ is such that $\pi^!(\CF)\in \Whit(\CY)$,
and its !-restriction to the locally closed subsets
$$(\BunNb)^{G\on{-level}_{n\cdot x}}_{=\mu\cdot x}-
(\BunNb)^{G\on{-level}_{n\cdot x}}_{=\mu\cdot x,\on{good\,elswhr}}$$
vanishes (for all $\mu$), then $\CF\in \Whit((\BunNb)^{G\on{-level}_{n\cdot x}}_{\infty\cdot x})$.

\end{thm}

\ssec{A strata-wise equivalence}

In this subsection we will show that the map $\pi$ of \eqref{e:map pi} defines a \emph{strata-wise} equivalence between
the local and the global Whittaker categories. 

\medskip

The discussion in this subsection applies equally well to the situation with the trivial character. 

\sssec{}

Let $\CY^\mu$ be as in \secref{sss:strata Y}. Note that for a given $\mu\in \Lambda$, the map $\pi$ restricts to a map
$$\pi_\mu:\CY^\mu\to \Whit((\BunNb)^{G\on{-level}_{n\cdot x}}_{=\mu\cdot x,\on{good\,elswhr}}).$$

\medskip

We will deduce \thmref{t:loc pullback} from the following more precise assertion:

\begin{thm}  \label{t:loc pullback mu} \hfill

\smallskip

\noindent{\em(a)} For every $\mu$, the functor $\pi_\mu^!$ sends 
$\Whit((\BunNb)^{G\on{-level}_{n\cdot x}}_{=\mu\cdot x,\on{good\,elswhr}})$ to $\Whit(\CY^\mu)$.

\smallskip

\noindent{\em(b)} The resulting functor 
$\Whit((\BunNb)^{G\on{-level}_{n\cdot x}}_{=\mu\cdot x,\on{good\,elswhr}})\to \Whit(\CY^\mu)$
is an equivalence.

\smallskip

\noindent{\em(c)} If $\CF\in \Shv((\BunNb)^{G\on{-level}_{n\cdot x}}_{=\mu\cdot x,\on{good\,elswhr}})$ is
such that $\pi_\mu^!(\CF)\in \Whit(\CY^\mu)$, then 
$$\CF\in \Whit((\BunNb)^{G\on{-level}_{n\cdot x}}_{=\mu\cdot x,\on{good\,elswhr}}).$$

\end{thm} 

The implication \thmref{t:loc pullback mu}(a) $\Rightarrow$ \thmref{t:loc pullback}(a) follows from
\propref{p:generated}(c). The implication \thmref{t:loc pullback mu}(c) $\Rightarrow$ \thmref{t:loc pullback}(b)
follows from \corref{c:good strata}(d). 

\medskip

The rest of this subsection is devoted to the proof of \thmref{t:loc pullback mu}.

\sssec{}

Choose a point $y\in X$ different from $x$. Let $N_{X-x}$ (resp., $N_{X-(x,y)}$) denote the group ind-scheme of sections of 
$N^{\omega^\rho}$ over $X-x$ (resp., $X-(x,y)$). 
%Consider the
%prestack quotient $\CY^\mu/N_{X-x}$, and note that the map $\pi_\mu$ factors via a map 
%\begin{equation} \label{e:quot by out}
%\CY^\mu/N_{X-x}\to (\BunNb)^{G\on{-level}_{n\cdot x}}_{=\mu\cdot x}.
%\end{equation} 

%We claim that the map \eqref{e:quot by out} is in fact an isomorphism. Indeed, this is just the statement that
%any $N^{\omega^\rho}$-bundle on $\text{affine test scheme }\times (X-x)$ admits a trivialization.

\medskip

Restriction to the formal punctured discs around $x$ and $y$ defines embeddings
$$N_{X-x}\hookrightarrow \fL_x(N)$$
and
$$\fL_y(N) \hookleftarrow N_{X-(x,y)}\hookrightarrow \fL_x(N).$$

By the sum of residues formula, we have
$$\chi_x|_{N_{X-(x,y)}}=-\chi_y|_{N_{X-(x,y)}}.$$

\medskip

Note that the map $\pi$ extends to a map
$$\CY\hookrightarrow \fL_y(N)/\fL^+_y(N)\times \CY \overset{\pi_{y}} \longrightarrow
(\BunNb)^{G\on{-level}_{n\cdot x}}_{\infty\cdot x,\on{good\,elswhr}}.$$

Moreover, the above map $\pi_y$ lifts to a map 
$$\pi^{\on{level}}_y:
\fL_y(N) \times \CY \to (\BunNb)^{G\on{-level}_{n\cdot x},N\on{-level}_{\infty\cdot y}}_{\infty\cdot x,\on{good\,elswhr}},$$
which is equivariant with respect to the $\fL_y(N)$-actions: we consider the $\fL_y(N)$-action by \emph{right} multiplication on 
the $\fL_y(N)$-factor in $\fL_y(N) \times \CY$ and the $\fL_y(N)$-action on 
$(\BunNb)^{G\on{-level}_{n\cdot x},N\on{-level}_{\infty\cdot y}}_{\infty\cdot x,\on{good\,elswhr}}$ from \secref{sss:extend to loop group}.

\medskip

Denote by $\pi_{y,\mu}$ and $\pi_{y,\mu}^{\on{level}}$ the corresponding maps
$$\fL_y(N)/\fL^+_y(N)\times \CY^\mu \to
(\BunNb)^{G\on{-level}_{n\cdot x}}_{=\mu\cdot x,\on{good\,elswhr}} \text{ and }
\fL_y(N) \times \CY^\mu \to (\BunNb)^{G\on{-level}_{n\cdot x},N\on{-level}_{\infty\cdot y}}_{=\mu\cdot x,\on{good\,elswhr}}.$$

\medskip

We have:

\begin{prop} \label{p:out} \hfill

\smallskip

\noindent{\em(a)} Pullback along $\pi_\mu$ defines an equivalence 
$$\Shv((\BunNb)^{G\on{-level}_{n\cdot x}}_{=\mu\cdot x,\on{good\,elswhr}}) \simeq \Shv(\CY^\mu)^{N_{X-x}}$$

\smallskip

\noindent{\em(b)} 
Pullback along $\pi_{y,\mu}$ defines an equivalence
$$\Shv((\BunNb)^{G\on{-level}_{n\cdot x}}_{=\mu\cdot x,\on{good\,elswhr}}) \simeq \Shv(\fL_y(N)/\fL^+_y(N)\times \CY^\mu)^{N_{X-(x,y)}}.$$
%\simeq \Shv(\CY^\mu)^{N_{X-x}}$$

\smallskip

\noindent{\em(c)} 
Pullback along $\pi_{y,\mu}^{\on{level}}$ defines an equivalence
$$\Shv((\BunNb)^{G\on{-level}_{n\cdot x},N\on{-level}_{\infty\cdot y}}_{=\mu\cdot x,\on{good\,elswhr}})\simeq 
\Shv(\fL_y(N) \times \CY^\mu)^{N_{X-(x,y)}}.$$
\end{prop}

\begin{proof}

For point (a), we claim that the map
$$\CY^\mu \to (\BunNb)^{G\on{-level}_{n\cdot x}}_{=\mu\cdot x,\on{good\,elswhr}}$$
identifies $(\BunNb)^{G\on{-level}_{n\cdot x}}_{=\mu\cdot x,\on{good\,elswhr}}$ with the prestack quotient of 
$\CY^\mu$ with respect to the action of $N_{X-x}$.

\medskip

Indeed, this is just the statement that
any $N^{\omega^\rho}$-bundle on $\{\text{affine test scheme}\}\times (X-x)$ admits a trivialization.

\medskip

Points (b) and (c) are proved similarly. 

\end{proof}

\begin{rem}
The same proof shows that the functor $\pi^!$ defines an equivalence
$$\Shv((\BunNb)^{G\on{-level}_{n\cdot x}}_{\infty\cdot x,\on{good\,elswhr}}) \simeq \Shv(\CY)^{N_{X-x}},$$
where 
$$(\BunNb)^{G\on{-level}_{n\cdot x}}_{\infty\cdot x,\on{good\,elswhr}}:=
(\BunNb)_{\infty\cdot x,\on{good\,elswhr}}\underset{\Bun_G}\times \Bun_G^{G\on{-level}_{n\cdot x}},$$
and where $(\BunNb)_{\infty\cdot x,\on{good\,elswhr}}$ is as Remark \ref{r:not a substack}. 

\medskip

This statement is not as useful for us because $(\BunNb)^{G\on{-level}_{n\cdot x}}_{\infty\cdot x,\on{good\,elswhr}}$
is not an algebraic stack, so we cannot say much about the category of sheaves on it. 

\end{rem}

Let us also observe: 

\begin{lem} \label{l:out bis} \hfill

\smallskip

\noindent{\em(a)} With respect to the equivalence of of \propref{p:out}(c),
objects of $\Shv((\BunNb)^{G\on{-level}_{n\cdot x},N\on{-level}_{\infty\cdot y}}_{=\mu\cdot x,\on{good\,elswhr}})$
that are $(\fL_y(N),\chi_y)$-equivariant correspond to objects of $\Shv(\fL_y(N)\times \CY^\mu)$ that are $N_{X-(x,y)}$-equivariant 
with respect to the diagonal action \emph{and} that are $(\fL_y(N),\chi_y)$-equivariant on the $\fL_y(N)$-factor by \emph{right}
multiplication.

\smallskip

\noindent{\em(a')} Same as {\em(a)}, but we replace $\chi_y$ by $-\chi_y$ and instead of right multiplication we consider
left multiplication. 

\smallskip

\noindent{\em(b)} With respect to the equivalence of 
\propref{p:out}(b), objects of $\Shv((\BunNb)^{G\on{-level}_{n\cdot x}}_{=\mu\cdot x,\on{good\,elswhr}})$
that belong to $\Whit((\BunNb)^{G\on{-level}_{n\cdot x}}_{=\mu\cdot x,\on{good\,elswhr}})$ correspond to objects of 
$\Shv(\fL_y(N)/\fL^+_y(N)\times \CY^\mu)$ that are $N_{X-(x,y)}$-equivariant with respect to the diagonal action \emph{and} 
that are $(\fL_y(N),-\chi_y)$-equivariant on the $\fL_y(N)/\fL^+_y(N)$-factor. 

\end{lem}

\sssec{}

We are now ready to prove \thmref{t:loc pullback mu}. 

\medskip

Let us observe that for an object 
$$\CF\in \Shv(\fL_y(N)/\fL^+_y(N)\times \CY^\mu)$$ that is 
$N_{X-(x,y)}$-equivariant wth respect to the \emph{diagonal} action of $N_{X-(x,y)}$ on $\fL_y(N)/\fL^+_y(N)\times \CY^\mu$,
the following extra conditions are equivalent:

\smallskip

\noindent(i) $\CF$ is $(N_{X-(x,y)},\chi_x)$-equivariant with respect to the action on the $\CY^\mu$-factor;

\smallskip

\noindent(ii) $\CF$ is $(N_{X-(x,y)},-\chi_y)$-equivariant with respect to the action on the $\fL_y(N)/\fL^+_y(N)$-factor;

\smallskip

\noindent(iii)  Both conditions (i) and (ii);

\smallskip

\noindent(iv) The restriction of $\CF$ to $1\times \CY^\mu$ is $(N_{X-(x,y)},\chi_x)$-equivariant.

\medskip

Moreover, restriction as in point (iv) defines an equivalence from the category spanned by such objects to
$$\Shv(\CY^\mu)^{N_{X-(x,y)},\chi_x}.$$

\medskip

Hence, using \lemref{l:out bis}(b), it remains to prove the next assertion: 

\begin{prop}   \label{p:skinny equiv} \hfill

\smallskip{\em(a)} 
The forgetful functor 
$$\Shv(\CY^\mu)^{\fL_x(N),\chi_x}\to \Shv(\CY^\mu)^{N_{X-(x,y)},\chi_x}$$
is an equivalence. 

\smallskip{\em(b)}
For any prestack $Z$, the forgetful functor 
$$\Shv(\fL_y(N)/\fL^+_y(N)\times Z)^{\fL_y(N),-\chi_y}\to \Shv(\fL_y(N)/\fL^+_y(N)\times Z)^{N_{X-(x,y)},-\chi_y}$$
is an equivalence.

\end{prop} 

\begin{proof}[Proof of \propref{p:skinny equiv}]

We will prove point (a), as point (b) is similar. 
The idea of the proof is that $N_{X-(x,y)}$ is ``dense" in $\fL_x(N)$. Here is how we spell this out:

\medskip

Let $Y^\mu$ be as in \secref{sss:Y mu}. Since the action of $N_{X-(x,y)}$ on $S^\mu$ is transitive
(this is one incarnation of the density of $N_{X-(x,y)}$ in $\fL_x(N)$), in \lemref{l:reduce to fd}, we obtain that restriction 
along $Y^\mu\hookrightarrow \CY^\mu$ defines
an equivalence
$$\Shv(\CY^\mu)^{N_{X-(x,y)},\chi_x}\to \Shv(Y^\mu)^{N_{X-(x,y)}\cap N^\mu,\chi_x}.$$

Hence, it remains to see that the restriction functor
$$\Shv(Y^\mu)^{N^\mu,\chi_x}\to \Shv(Y^\mu)^{N_{X-(x,y)}\cap N^\mu,\chi_x}$$
is an equivalence.

\medskip

To show this we note that we can find a normal group sub-scheme $N''\subset N^\mu$ of finite codimension such that
its action on $Y^\mu$ is trivial. Hence, it suffices to show that the functor
$$\Shv(Y^\mu)^{N^\mu/N',\chi_x}\to \Shv(Y^\mu)^{N_{X-(x,y)}\cap N^\mu/N_{X-(x,y)}\cap N',\chi_x}$$
is an equivalence.  

\medskip

However, for any $N'$ of finite codimension in $N^\mu$, the map
$$N_{X-(x,y)}\cap N^\mu/N_{X-(x,y)}\cap N'\to N^\mu/N'$$
is an isomorphism (again, by the density of $N_{X-(x,y)}$ in $\fL_x(N)$). 

\end{proof}

For future use, we note that the above argument also proves the following:

\begin{lem} \label{l:out bis bis}
Under the equivalence of 
of \propref{p:out}(a), objects of $\Shv((\BunNb)^{G\on{-level}_{n\cdot x}}_{=\mu\cdot x,\on{good\,elswhr}})$
that belong to $\Whit((\BunNb)^{G\on{-level}_{n\cdot x}}_{=\mu\cdot x,\on{good\,elswhr}})$ correspond to objects of 
$\Shv(\CY^\mu)$, for which $N_{X-x}$-equivariance extends to $(\fL_x(N),\chi_x)$-equivariance. 
\end{lem} 

\ssec{Local-to-global functor and duality}

Above we have considered the functor $\pi^!$ that maps the global Whittaker category to $\Whit(\CY)$. In this
subsection we will define a functor that maps $\Whit(\CY)_{\on{co}}$ to the global version. 

\sssec{}

Recall that according to \corref{c:dual dualizable} the dual of $\Whit(\CY)$ is the category $\Whit(\CY)_{\on{co}}$
(defined using the opposite character). Similarly, according to \propref{p:Verdier for glob Whit}, the category dual to 
$\Whit((\BunNb)^{G\on{-level}_{n\cdot x}}_{\infty\cdot x})$ is again $\Whit((\BunNb)^{G\on{-level}_{n\cdot x}}_{\infty\cdot x})$
(defined using the opposite character). 

\medskip

Let us describe the resulting functor
\begin{equation} \label{e:dual functor}
\pi_{*,\Whit}:\Whit(\CY)_{\on{co}}:=\Shv(\CY)_{\fL_x(N),\chi_x}\to \Whit((\BunNb)^{G\on{-level}_{n\cdot x}}_{\infty\cdot x})
\end{equation} 
dual to 
$$\pi^!: \Whit((\BunNb)^{G\on{-level}_{n\cdot x}}_{\infty\cdot x})\to \Shv(\CY)^{\fL_x(N),\chi_x}.$$

\sssec{}

Note that the morphism $\pi$ is ind-schematic, so the functor 
$$\pi_*:\Shv(\CY)\to \Shv((\BunNb)^{G\on{-level}_{n\cdot x}}_{\infty\cdot x})$$
is well-defined.

\medskip

We claim:

\begin{lem} \label{l:dual pi}
The composite
$$\Shv(\CY) \overset{p^{\fL_x(N),\chi_x}} \longrightarrow 
\Whit(\CY)_{\on{co}} \overset{\pi_{*,\Whit}}\longrightarrow \Whit((\BunNb)^{G\on{-level}_{n\cdot x}}_{\infty\cdot x})$$
identifies canonically with the functor $\Av_{*,\on{glob}}^{\Whit}\circ \pi_*$.
\end{lem}

(We remind that the functor $\Av_{*,\on{glob}}^{\Whit}$ that appears in the lemma is the right adjoint to the embedding
$\Whit((\BunNb)^{G\on{-level}_{n\cdot x}}_{\infty\cdot x})\hookrightarrow \Shv((\BunNb)^{G\on{-level}_{n\cdot x}}_{\infty\cdot x})$.) 

\begin{proof}

The assertion follows from \eqref{e:Av duality} and the fact that for any $\CF\in \Shv(\CY)^c$, the object 
$$\pi_*(\CF)\in \Shv((\BunNb)^{G\on{-level}_{n\cdot x}}_{\infty\cdot x})$$ is *-extended from a quasi-compact
substack so that for any $\CF'\in \Shv((\BunNb)^{G\on{-level}_{n\cdot x}}_{\infty\cdot x})$ we have
$$\Gamma((\BunNb)^{G\on{-level}_{n\cdot x}}_{\infty\cdot x},\pi_*(\CF)\sotimes \CF')\simeq
\Gamma(\CY,\CF\sotimes \CF').$$

\end{proof}

\begin{rem}
In \secref{p:averagings}(b) we will describe the functor $\Av_{*,\on{glob}}^{\Whit}\circ \pi_*$ in terms local at $x$.
\end{rem}

\ssec{Comparing the averaging procedures}

In this subsection we will see that the composite functor
$$\Whit(\CY)_{\on{co}}\to \Whit((\BunNb)^{G\on{-level}_{n\cdot x}}_{\infty\cdot x})\to \Whit(\CY)$$
essentially coincides with the functor $\on{Ps-Id}_{\Whit}$ of \secref{sss:PsId}. 

\sssec{}

Recall that $N_{X-x}$ denotes the group ind-scheme of sections of of $N^{\omega^\rho}$ over $X-x$. Note, however,
that the image of $N_{X-x}\hookrightarrow \fL_x(N)$ is no longer dense. 
Let $N'\subset \fL_x(N)$ be a large enough group subscheme so that $N'\cdot N_{X-x}=\fL_x(N)$.

\medskip

We claim:

\begin{prop}  \hfill  \label{p:averagings}

\smallskip

\noindent{\em(a)} $\pi^!\circ \Av_{*,\on{glob}}^{\Whit}: \Shv((\BunNb)^{G\on{-level}_{n\cdot x}}_{\infty\cdot x})\to
\Whit(\CY)$ identifies canonically with $\Av^{N',\chi_x}_*\circ \pi^!$.

\smallskip

\noindent{\em(b)} The functor $\Av_{*,\on{glob}}^{\Whit}\circ \pi_*$ identifies with $\pi_*\circ \Av^{N',\chi_x}_*$. 
\end{prop}

\begin{proof}

For point (a), it suffices to prove the corresponding assertion for the functor
$$\pi_\mu^!:\Shv((\BunNb)^{G\on{-level}_{n\cdot x}}_{=\mu\cdot x,\on{good\,elswhr}})\to \Whit(\CY^\mu)$$
for all $\mu\in \Lambda$.

\medskip

For point (b), it suffices to prove the corresponding assertion for the functor
$$(\pi_\mu)_*:\Whit(\CY^\mu)\to \Shv((\BunNb)^{G\on{-level}_{n\cdot x}}_{=\mu\cdot x,\on{good\,elswhr}})$$
for all $\mu\in \Lambda$.

\medskip

Now both assertions follow from the equivalence of \lemref{l:out bis bis}.

\end{proof}

\sssec{}

Consider the action of the group ind-scheme $N_{X-x}$ on $\CY$, and consider the corresponding functor
$$\Av^{N_{X-x}}_{*,\on{ren}}:=\on{act}_*\circ p^!:\Shv(\CY)\to \Shv(\CY),$$
see Remark \ref{r:ren av}.

\medskip

We claim:

\begin{prop}  \label{p:pi and pi}
The composite $\pi^!\circ \pi_*:\Shv(\CY)\to \Shv(\CY)$ identifies with the above functor $\Av^{N_{X-x}}_{*,\on{ren}}$.
\end{prop}

\begin{proof}

Follows by base change from the fact that the action of $N_{X-x}$ on $\CY$ defines an isomorphism
$$N_{X-x}\times \CY\simeq \CY\underset{(\BunNb)^{G\on{-level}_{n\cdot x}}}\times \CY.$$

\end{proof}

As a consequence, we obtain: 

\begin{cor} \label{c:PsId glob}
The functor $$\pi^!\circ \pi_{*,\Whit}:\Whit(\CY)_{\on{co}}\to \Whit(\CY)$$ identifies canonically with 
$\on{Ps-Id}_{\Whit}[-2d]$, where $d=\dim(N_0\backslash \fL_x(N)/N_{X-x})$.
\end{cor}

\begin{proof}

By \lemref{l:dual pi}, we need to show that the functor
$$\pi^!\circ \Av_{*,\on{glob}}^{\Whit}\circ \pi_*$$
identifies with $\Av_{*,\on{ren}}^{\fL_x(N),\chi_x}$. 

\medskip

Combining Propositions \ref{p:averagings}(b) and \ref{p:pi and pi}, we obtain that the functor $\pi^!\circ \Av_{*,\on{glob}}^{\Whit}\circ \pi_*$
is given by
$$\Av^{N',\chi_x}_*\circ \Av_{*,\on{ren}}^{N_{X-x}},$$
where $N'$ is as in \propref{p:averagings}. 

\medskip

However, unwinding the definitions, it is easy to see that $\Av^{N',\chi_x}_*\circ \Av_{*,\on{ren}}^{N_{X-x}}$ 
identifies canonically with 
$$\Av_{*,\on{ren}}^{\fL_x(N),\chi_x}[2d]=:\on{Ps-Id}_{\Whit}[-2d].$$

\end{proof} 

\ssec{Statement of the local-to-global equivalence}

In this subsection we finally state the local-to-global comparison theorem. 

\sssec{}

We are now ready to state the main result of this paper:

\begin{thm}  \label{t:main}
The functor 
$$\pi^!:\Whit((\BunNb)^{G\on{-level}_{n\cdot x}}_{\infty\cdot x})\to \Whit(\CY)$$
is an equivalence.
\end{thm}

The proof will be given in \secref{ss:proof of main}.

\sssec{}  \label{sss:ff enough}

Some remarks are in order. Note that for every $\mu$ we have a commutative diagram
$$
\CD
\Whit(\CY^\mu) @>>> \Whit(\CY) \\
@AAA   @AAA   \\
\Whit((\BunNb)^{G\on{-level}_{n\cdot x}}_{=\mu\cdot x})  @>>>  \Whit((\BunNb)^{G\on{-level}_{n\cdot x}}_{\infty\cdot x}),
\endCD
$$ 
where the horizontal arrows are *-direct image functors. Since the left vertical arrows are equivalences for all $\mu$ 
(by \thmref{t:loc pullback mu}(b) and \corref{c:good strata}(c)), we see that the functor in the theorem is a ``stratum-wise equivalence". 
So the challenge of the theorem is to show that these strata glue in the same way in the source and the target.

\medskip

Given \propref{p:generated}(b), we obtain that \thmref{t:main} is equivalent to the statement that the functor 
$$\pi^!:\Shv((\BunNb)^{G\on{-level}_{n\cdot x}}_{\infty\cdot x})\to \Shv(\CY)$$
is fully faithful, when restricted to 
$$\Whit((\BunNb)^{G\on{-level}_{n\cdot x}}_{\infty\cdot x})\subset \Shv((\BunNb)^{G\on{-level}_{n\cdot x}}_{\infty\cdot x}).$$

\sssec{}

From \thmref{t:main} and \lemref{l:dual pi}, we obtain:

\begin{cor}
The functor 
$$\pi_{*,\Whit}:\Whit(\CY)_{\on{co}}\to \Whit((\BunNb)^{G\on{-level}_{n\cdot x}}_{\infty\cdot x})$$
is an equivalence.
\end{cor} 

Finally, combining with \corref{c:PsId glob} we obtain:

\begin{cor}  \label{c:inv vs coinv}
The functor $\on{Ps-Id}_{\Whit}$ is an equivalence.
\end{cor} 

Thus, we obtain a proof of \thmref{t:inv vs coinv}. 

\section{Ran version and the proof of the main theorem}  \label{s:Ran}

The proof of \thmref{t:main} is based on considering the \emph{Ran space} version of the local
Whittaker category. 

\medskip

We will show that the the pullback functor from the global version
to the Ran version is fully faithful (this will be a geometric assertion not related to the
specifics of the Whittaker situation). Then we will show that the original local Whittaker
category (at one point of the curve) is equivalent to the Ran version. 

\ssec{Ran geometry}

In this subsection we recall the definition of the Ran space and various geometric objects associated with it. 

\sssec{}

Recall that the Ran space of $X$, denoted $\Ran(X)$, is the functor that associates to an affine test scheme $S$
the \emph{set} of finite non-empty subsets $\CI\subset \Hom(S,X)$. 

\medskip

Explicitly, 
$$\Ran(X)\simeq \underset{I}{\on{colim}}\, X^I,$$
where the colimit is taken in $\on{PreStk}$, and the index category is opposite to that of finite non-empty subsets and surjective maps;
to a surjection $\phi:I_1\twoheadrightarrow I_2$ we associate the correspondind diagonal map
$$\Delta_\phi:X^{I_2}\to X^{I_1}.$$

\sssec{}  \label{sss:marked}

We will consider a version of $\Ran(X)$ with a marked point, denoted $\Ran(X)_x$. By definition, $\Ran(X)_x$ associates 
to an affine test scheme $S$ the set of finite non-empty subsets $\CI\subset \Hom(S,X)$ with a distinguished element
corresponding to the map
$$S\to \on{pt}\overset{x}\to X.$$

\medskip

Explicitly, 
$$\Ran(X)_x\simeq \underset{I}{\on{colim}}\, (X^I\underset{X}\times \{x\}),$$
where the colimit is taken in $\on{PreStk}$, and the index category is opposite to that of finite non-empty subsets 
equipped with a distinguished element and surjective maps that preserve distinguished elements.

\medskip

Note that we have a map $\Ran(X)\to \Ran(X)_x$, given by adding the distinguished element.
 
\sssec{}

Let $\Gr_{G,\Ran}$ denote the (following slightly twisted version of the) Ran Grassmannian of $G$:

\medskip

By definition, $\Gr_{G,\Ran}$ 
attaches to an affine test scheme $S$ the set of triples $(\CI,\CP_G,\gamma)$, where:

\begin{itemize}

\item $\CI$ is a finite non-empty subset of $\Hom(S,X)$;

\item $\CP_G$ is a $G$-bundle on $S\times X$; 

\item $\gamma$ is an identification of $\CP_G$ with the pullback of $\CP_G^{\omega^\rho}$
over the open subset of $S\times X$ equal to the complement of the union of graphs of the maps $S\to X$ that comprise $\CI$. 

\end{itemize} 

\sssec{}

Let $S^0_\Ran\subset \ol{S}{}^0_\Ran$ be the (locally) closed subfunctors of $\Gr_\Ran$ defined by the following
conditions:

\medskip

For $\ol{S}{}^0_\Ran$, we require that the composite \emph{meromorphic}  maps
$$(\omega^{\frac{1}{2}})^{\langle \clambda,2\rho\rangle}\to \CV^{\clambda}_{\CP^{\omega^\rho}_G} \overset{\gamma}\to \CV^{\clambda}_{\CP_G}$$
be \emph{regular} on $S\times X$ for all $\check\lambda\in \check\Lambda^+$. 

\medskip

For $S^0_\Ran$ we require that these maps be an injective bundle maps. 

\sssec{}

We will now introduce a Ran version of the space $\CY$, denoted $\CY_{\Ran_x} $. By definition $\CY_{\Ran_x} $ 
attaches to an affine test scheme $S$ the set of quadruples $(\CI,\CP_G,\gamma,\epsilon)$, where: 

\begin{itemize}

\item $\CI$ is a finite non-empty subset of $\Hom(S,X)$ with a distinguished element
corresponding to the map $S\to \on{pt}\overset{x}\to X$;

\item $\CP_G$ is a $G$-bundle on $S\times X$; 

\item $\gamma$ is an identification of $\CP_G$ with the pullback of $\CP_G^{\omega^\rho}$
over the open subset of $S\times X$ equal to the complement of the union of graphs of the maps $S\to X$ that comprise $\CI$,
subject to the condition that the composite \emph{meromorphic}  maps
\begin{equation} \label{e:new Plucker}
(\omega^{\frac{1}{2}})^{\langle \clambda,2\rho\rangle}\to \CV^{\clambda}_{\CP^{\omega^\rho}_G} \overset{\gamma}\to \CV^{\clambda}_{\CP_G}
\end{equation}
be \emph{regular} on $S\times (X-x)$ for all $\check\lambda\in \check\Lambda^+$. 

\item $\epsilon$ is a structure of level $n$ on $\CP_G$ along $S\times \{x\}$.

\end{itemize} 

\sssec{}

As in \secref{sss:map pi}, we have a naturally defined map
$$\pi_\Ran:\CY_{\Ran_x}  \to \BunNbx^{G\on{-level}_{n\cdot x}}.$$

We have the following basic geometric assertion:

\begin{thm} \label{t:contr}
The pullback functor
$$\pi_\Ran^!:\Shv(\BunNbx^{G\on{-level}_{n\cdot x}})\to \Shv(\CY_{\Ran_x} )$$
is fully faithful. 
\end{thm}

The proof repeats verbatim the proof of \cite[Theorem 3.4.4]{Ga2}. 

\ssec{The Ran version of the Whittaker category}

In this subsection we state the key result, \thmref{t:unit}, which says that the Ran version of the
Whittaker category is (essentially) equivalent to the local one (at point point $x\in X$). 

\sssec{}

The definition of the Whittaker category $\Whit(\CY)$ has a Ran version, denoted 
$$\Whit(\CY_{\Ran_x} ):=\Shv(\CY_{\Ran_x} )^{\fL_\Ran(N),\chi_\Ran}.$$

We refer the reader to \cite[Sect. 1.2]{Ga2}, where the definition is spelled out for the trivial character; the case of
the non-degenerate character is no different. See also \secref{ss:proof of unit}, below. 

\medskip

As in \thmref{t:loc pullback} one shows:

\begin{prop} \label{p:Ran pullback}
The functor $\pi_\Ran^!$ sends $\Whit((\BunNb)^{G\on{-level}_{n\cdot x}}_{\infty\cdot x})$ to $\Whit(\CY_{\Ran_x} )$.
\end{prop}

Combined with \thmref{t:contr}, we obtain: 

\begin{cor}
The functor
$$\pi_\Ran^!:\Whit((\BunNb)^{G\on{-level}_{n\cdot x}}_{\infty\cdot x})\to \Whit(\CY_{\Ran_x} )$$
is fully faithful.
\end{cor}

\sssec{}

Note now that we have a naturally defined map (in fact, a closed embedding)
$$\on{unit}_\Ran:\Ran(X)_x\times \CY\to \CY_{\Ran_x} .$$

Observe also that the definition of $\Whit(\CY)\subset \Shv(\CY)$ has a variant 
$$\Whit(\CZ\times \CY)\subset \Shv(\CZ\times \CY)$$
for an arbitrary prestack $\CY$, where we use the action of $\fL_x(N)$ on $\CZ\times \CY$ coming from the $\CY$-factor.

\medskip

In \secref{ss:proof of unit} we will prove:

\begin{thm} \label{t:unit}
The functor
$$\on{unit}_\Ran^!:\Shv(\CY_{\Ran_x} )\to \Shv(\Ran(X)_x\times \CY)$$
defines an equivalence 
$$\Whit(\CY_{\Ran_x} )\simeq \Whit(\Ran(X)_x\times \CY).$$
\end{thm}

\ssec{Proof of the main theorem}  \label{ss:proof of main}

In this subsection we will deduce \thmref{t:main} from \thmref{t:unit}. 

\sssec{}

As was explained in \secref{sss:ff enough}, it suffices to show that the functor
$$\pi^!:\Shv((\BunNb)^{G\on{-level}_{n\cdot x}}_{\infty\cdot x})\to \Shv(\CY)$$
is fully faithful, when restricted to 
$$\Whit((\BunNb)^{G\on{-level}_{n\cdot x}}_{\infty\cdot x})\subset \Shv((\BunNb)^{G\on{-level}_{n\cdot x}}_{\infty\cdot x}).$$

\sssec{}

Note that the map $\pi:\CY\to (\BunNb)^{G\on{-level}_{n\cdot x}}_{\infty\cdot x}$ equals the composition
$$\CY\to \Ran(X)_x\times \CY \overset{\on{unit}_\Ran}\longrightarrow \CY_{\Ran_x} \overset{\pi_\Ran}\longrightarrow 
(\BunNb)^{G\on{-level}_{n\cdot x}}_{\infty\cdot x},$$
where the first arrow corresponds to the tautological map 
$$\on{pt}\overset{\{x\}}\to \Ran(X)_x.$$

\medskip

Hence, the functor $\pi^!$, restricted to $\Whit((\BunNb)^{G\on{-level}_{n\cdot x}}_{\infty\cdot x})$ is the composition
\begin{equation} \label{e:factor pi}
\Whit((\BunNb)^{G\on{-level}_{n\cdot x}}_{\infty\cdot x}) \overset{\pi_\Ran^!}\longrightarrow \Whit(\CY_{\Ran_x} ) 
\overset{\on{unit}_\Ran}\longrightarrow \Whit(\Ran(X)_x\times \CY)  \to \Whit(\CY).
\end{equation}

\medskip

According to \thmref{t:contr}, the first arrow in \eqref{e:factor pi} is fully faithful, and the second arrow is an equivalence by
\thmref{t:unit}. Hence, the functor
$$(\pi_\Ran \circ \on{unit}_\Ran)^!:
\Whit((\BunNb)^{G\on{-level}_{n\cdot x}}_{\infty\cdot x})\to \Whit(\Ran(X)_x\times \CY)$$
is fully faithful.

\sssec{}

Note now that the map
$$\pi_\Ran \circ \on{unit}_\Ran:\Ran(X)_x\times \CY\to (\BunNb)^{G\on{-level}_{n\cdot x}}_{\infty\cdot x}$$
factors as
$$\Ran(X)_x\times \CY\to \CY \overset{\pi}\longrightarrow (\BunNb)^{G\on{-level}_{n\cdot x}}_{\infty\cdot x},$$
where the first map is the projection on the $\CY$-factor.

\medskip

Hence, the functor $\pi_\Ran^!$ is a \emph{retract} of the functor $(\pi_\Ran \circ \on{unit}_\Ran)^!$. In particular,
for $\CF_1,\CF_2\in \Whit((\BunNb)^{G\on{-level}_{n\cdot x}}_{\infty\cdot x})\to \Whit(\Ran(X)_x\times \CY)$, the map
$$\CHom_{\Whit((\BunNb)^{G\on{-level}_{n\cdot x}}_{\infty\cdot x})}(\CF_1,\CF_2) \to
\CHom_{\Whit(\CY)}(\pi^!(\CF_1),\pi^!(\CF_2))$$
is a \emph{retract} of the map
$$\CHom_{\Whit((\BunNb)^{G\on{-level}_{n\cdot x}}_{\infty\cdot x})}(\CF_1,\CF_2) \to
\CHom_{\Whit(\Ran(X)_x\times \CY)}((\pi_\Ran\circ \on{unit}_\Ran)^!(\CF_1),(\pi_\Ran\circ \on{unit}_\Ran)^!(\CF_2)).$$

\medskip

Hence, since the latter is an is isomorphism, so is the former. 

\qed[\thmref{t:main}]

\begin{rem}
Instead of using the retract argument, we could have finished the proof differently, using the fact that the marked Ran space
$\Ran(X)_x$ is contractible. Indeed, the latter implies that the functor
$$\Shv(\CY)\to \Shv(\Ran(X)_x\times \CY)$$
is fully faithful. 

\medskip

Hence, the fact that the composite 
$$\Whit((\BunNb)^{G\on{-level}_{n\cdot x}}_{\infty\cdot x}) \overset{\pi^!}\longrightarrow \Whit(\CY)\to 
\Whit(\Ran(X)_x\times \CY)$$
is fully faithful implies that the first arrow is fully faithful.
\end{rem} 

\ssec{Unital structure}  \label{ss:proof of unit}

The goal of this subsection is to supply a crucial ingredient that will be used in the proof of 
\thmref{t:unit}. It will amount to a \emph{unital structure} on $\Whit(\CY_{\Ran_x} )$ in the world of factorization
categories and modules over them. 

\sssec{}

For $I$ an object in category of finite sets with a marked point (see \secref{sss:marked}), denote
$$\CY_I:=X^I\underset{\Ran(X)_x}\times \CY_{\Ran(X)_x}.$$

We have
$$\CY_{\Ran_x} \simeq \underset{I}{\on{colim}}\, \CY_I,$$
and hence
\begin{equation} \label{e:Y Ran}
\Shv(\CY_{\Ran_x} )\simeq \underset{I}{\on{lim}}\, \Shv(\CY_I).
\end{equation}

\medskip

We have the corresponding full subcategories
$$\Whit(\CY_I)\subset \Shv(\CY_I).$$
and under the equivalence \eqref{e:Y Ran}, the full subcategory 
$$\underset{I}{\on{lim}}\, \Whit(\CY_I)\subset \underset{I}{\on{lim}}\, \Shv(\CY_I)$$
corresponds to $\Whit(\CY_{\Ran_x} )\subset \Shv(\CY_{\Ran_x} )$.

\medskip

Let $\on{unit}_I$ denote the corresponding map
$$X^I\times  \CY\to \CY_I.$$

To prove \thmref{t:unit}, it suffices to prove its version for every $I$ individually:

\begin{thm} \label{t:unit I}
For every $I$, the functor $\on{unit}_I^!$ induces an equivalence
$$\Whit(\CY_I)\to \Whit(X^I\times \CY).$$
\end{thm} 

The assertion of \thmref{t:unit I} is Zariski-local in $X$, so from now on we will assume that $X$ is affine.
(This is done for notational convenience, when manipulating the two versions of the parameterized formal disc, 
$\wh\cD_{x_I}$ and $\cD_{x_I}$ below.)

\sssec{} 

For an affine test-scheme $S$ and an $S$-point $x_I$ of $X^I$ (i.e., an $I$-tuple of maps $S\to X$),
let $\wh\cD_{x_I}$ be the formal completion of $S\times X$ along the union of the graphs of maps that comprise $x_I$.
We consider $\wh\cD_{x_I}$ as an ind-scheme. 

\medskip

The assumption that $X$ be affine implies that $\wh\cD_{x_I}$ is 
\emph{ind-affine}; in particular, it gives rise to a well-defined ind-object in the category of affine schemes. Let
$\cD_{x_I}$ denote the colimit of $\wh\cD_{x_I}$, taken in the category of affine schemes. I.e., if
$$\wh\cD_{x_I}=\underset{\alpha}{\on{colim}}\, \Spec(A_\alpha),$$
where the colimit is taken in $\on{PreStk}$, then $\cD_{x_I}=\Spec(A)$, where
$$A:=\underset{\alpha}{\on{lim}}\, A_\alpha,$$
where the limit is taken in the category of commutative algebras.

\medskip

The points comprising $x_I$ give rise to an $I$-tuple of maps $S\to \cD_{x_I}$. 
Let $\ocD_{x_I}$ be the (affine) open subscheme of $\cD_{x_I}$ obtained by removing the union of the graphs of the above maps
$S\to \cD_{x_I}$. We will also consider the larger (affine) open subscheme 
$$\cD'_{x_I}:=\cD_{x_I}-(S\times \{x\}).$$
That is, instead of all maps that comprise $x_I$, we only take the distinguished constant map $S\to \on{pt}\overset{x}\to X$. 

\medskip

When instead of $x_I$ we just use the distinguished map to $x$, we obtain the corresponding versions of
the formal (resp., formal punctured) disc around $x$:
$$S\times \wh\cD_x,\,\, S\wh\times \cD_x:=\on{colim}(S\times \wh\cD_x) \text { and }
S\wh\times \ocD_x:=S\wh\times \cD_x-(S\times \{x\}).$$

We have the naturally defined maps
$$S\times \wh\cD_x\to \wh\cD_{x_I},\,\, 
S\wh\times \cD_x\to \cD_{x_I} \text{ and } S\wh\times \ocD_x\to \cD'_{x_I}.$$

\sssec{} \label{sss:rel loop groups}

For every $I$, let $\fL^+_I(N)$ (resp., $\fL^+_I(N)'\subset \fL_I(N)$) denote the following group-schemes 
(resp., group ind-schemes) over $X^I$:

\medskip

\begin{itemize} 

\item 
A lift of $x_I$ to a map $S\to \fL^+_I(N)$ is a map $\cD_{x_I}\to N^{\omega^\rho}$ (or, equivalently, a 
map $\wh\cD_{x_I}\to N^{\omega^\rho}$), compatible with a projection to $X$. 

\item 
A lift of $x_I$ to a map $S\to \fL_I(N)$ is a map $\ocD_{x_I}\to N^{\omega^\rho}$, compatible with a projection to $X$. 

\item A lift of $x_I$ to a map $S\to \fL^+_I(N)'$ is a map 
$\cD'_{x_I}\to N^{\omega^\rho}$, compatible with a projection to $X$.

\end{itemize} 

\medskip

We have the closed embeddings
$$\fL^+_I(N)\subset \fL^+_I(N)'\subset \fL_I(N),$$
and the projections
$$\fL^+_I(N)\to X^I\times \fL^+_x(N) \text{ and } \fL^+_I(N)'\to X^I\times \fL_x(N).$$

\begin{rem}
Let a $k$-point $x_I$ be given by a collection of $I$ distinct points $y_i$ of $X$, and the distinguished point $x$.
Then the fibers of $\fL^+_I(N)$, $\fL^+_I(N)'$ and $\fL_I(N)$ over such $x_I$ are given, respectively, by
$$\underset{i}\Pi\, \fL^+_{y_i}(N)\times \fL^+_{x}(N),\,\, 
\underset{i}\Pi\, \fL^+_{y_i}(N)\times \fL_x(N) \text{ and } \underset{i}\Pi\, \fL_{y_i}(N)\times \fL_x(N).$$
\end{rem}

\sssec{}

By definition
$$\Whit(\CY_I)=\Shv(\CY_I)^{\fL_I(N),\chi_I}.$$

We note that the map $\on{unit}_I$ is compatible with the action of $\fL^+_I(N)'$, where $\fL^+_I(N)'$ acts on $X^I\times \CY$
via the projection 
\begin{equation} \label{e:project on disting}
\fL^+_I(N)'\to X^I\times \fL_x(N)
\end{equation} 
and the $\fL_x(N)$-action on $\CY$.

\medskip

Hence, the functor $\on{unit}_I^!$ gives rise to a functor
$$\Whit(\CY_I)=\Shv(\CY_I)^{\fL_I(N),\chi_I}\to \Shv(X^I\times \CY)^{\fL^+_I(N)',\chi_I},$$
while the functor
$$\Whit(X^I\times \CY)=\Shv(X^I\times \CY)^{\fL_x(N),\chi_x}\to \Shv(X^I\times \CY)^{\fL^+_I(N)',\chi_I}$$
is an equivalence, since the kernel of \eqref{e:project on disting} is pro-unipotent. 

\medskip

This shows that $\on{unit}_I^!$ gives rise to a well-defined functor
$$\Whit(\CY_I)\to \Whit(X^I\times \CY).$$

\sssec{}

We now claim:

\begin{thm} \label{t:averaging}
The functor 
$$\on{unit}_I^!:\Whit(\CY_I)\to \Whit(X^I\times \CY)$$
admits a left adjoint.  
Moreover, this left adjoint respects the actions of $\Shv(X^I)$ on the two sides, given by
the operation of !-pullback and $\sotimes$. 
\end{thm}

The proof of \thmref{t:averaging} will be given in \secref{ss:proof of average}. 

\ssec{Proof of \thmref{t:unit I}}  

In this subsection we will show how \thmref{t:averaging} implies \thmref{t:unit I}. 

\sssec{}

Consider the stratification of $X^I$ according to the pattern of collision of points (including the distinguished point $x$). 
The strata are enumerated by equivalence relations on $I$ (partitions of $I$ as a disjoint union of subsets). 
For each partition $\fP$, let $X^\fP$ denote the corresponding locally closed subset of $X^I$. Denote
$$\CY_\fP:=X^\fP\underset{X^I}\times \CY_{\Ran_x} .$$

Consider the corresponding categories
$$\Whit(\CY_\fP)\subset \Shv(\CY_\fP) \text{ and } \Whit(X^\fP\times \CY)\subset \Shv(X^\fP\times \CY).$$

The map $\on{unit}_I$ induces a map
$$\on{unit}_\fP:X^\fP\times \CY\to \CY_\fP.$$

We will prove:

\begin{prop} \label{p:unit beta}
The functor
$$\on{unit}_\fP^!:\Whit(\CY_\fP)\to \Whit(X^\fP\times \CY)$$
is an equivalence for every $\fP$. 
\end{prop}

\sssec{}

Let us deduce \thmref{t:unit I} from \propref{p:unit beta}, combined with \thmref{t:averaging}. First off, \propref{p:unit beta}
implies that the functor $\on{unit}_I^!$ is conservative. Hence, it remains to show that the unit of the adjunction
$$\CF\to \on{unit}_I^! \circ (\on{unit}_I^!)^L(\CF), \quad \CF\in \Whit(X^I\times \CY)$$
is an isomorphism. 
 
 \medskip

Let $\iota_\fP$ denote the locally closed embedding $X^\fP\to X^I$. It suffices to show that
each of the maps
$$(\iota_\fP)_*\circ \iota_\fP^!(\CF) \to (\iota_\fP)_*\circ \iota_\fP^!\circ 
\on{unit}_I^! \circ (\on{unit}_I^!)^L(\CF)$$
is an isomorphism.

\medskip

We have a commutative diagram
$$
\CD
(\iota_\fP)_*\circ \iota_\fP^!(\CF)  @>>>  (\iota_\fP)_*\circ (\iota_\fP)^!\circ  \on{unit}_I^! \circ (\on{unit}_I^!)^L(\CF) \\ 
@A{\on{Id}}AA    @AAA   \\
(\iota_\fP)_*\circ \iota_\fP^!(\CF)  @>>>   (\iota_\fP)_*\circ \on{unit}_\fP^! \circ (\on{unit}_\fP^!)^L\circ (\iota_\fP)^!(\CF). 
\endCD
$$

The bottom horizontal arrow in this diagram is an isomorphism by \propref{p:unit beta}. Hence, it remains to show that the
right vertical arrow is an isomorphism. 

\medskip

We have
$$\iota_\fP^!\circ  \on{unit}_I^! \simeq \on{unit}_\fP^! \circ \iota_\fP^!
\text{ and }
(\iota_\fP)_*\circ  \on{unit}_\fP^! \simeq \on{unit}_I^! \circ (\iota_\fP)_*.$$

So it suffices to show that the map 
$$(\on{unit}_\fP^!)^L\circ (\iota_\fP)^!(\CF)\to (\iota_\fP)^! \circ (\on{unit}_I^!)^L(\CF)$$
is an isomorphism in $\Whit(\CY_\fP)$.

\medskip

Since $\iota_\fP$ is a locally closed embedding, the functor $(\iota_\fP)_*$ is fully faithful. Hence, 
for any $\CF'\in \Whit(\CY_\fP)$ we have
\begin{multline*}
\CHom((\on{unit}_\fP^!)^L\circ (\iota_\fP)^!(\CF),\CF')\simeq \\
\simeq \CHom((\iota_\fP)^!(\CF),\on{unit}_\fP^!(\CF'))
\simeq  \CHom((\iota_\fP)_*\circ (\iota_\fP)^!(\CF),(\iota_\fP)_*\circ  \on{unit}_\fP^!(\CF'))\simeq \\
\simeq\CHom((\iota_\fP)_*\circ (\iota_\fP)^!(\CF),\on{unit}_I^!\circ (\iota_\fP)_*(\CF'))\simeq
\CHom((\iota_\fP)_*(\omega_{X^\fP})\sotimes \CF,\on{unit}_I^!\circ (\iota_\fP)_*(\CF'))\simeq \\
\simeq \CHom((\on{unit}_I^!)^L((\iota_\fP)_*(\omega_{X^\fP})\sotimes \CF),(\iota_\fP)_*(\CF'))\simeq
\CHom((\iota_\fP)_*(\omega_{X^\fP})\sotimes (\on{unit}_I^!)^L(\CF),(\iota_\fP)_*(\CF'))\simeq \\
\simeq \CHom((\iota_\fP)_*\circ (\iota_\fP)^!\circ (\on{unit}_I^!)^L(\CF),(\iota_\fP)_*(\CF'))
\simeq \CHom((\iota_\fP)^!\circ (\on{unit}_I^!)^L(\CF),\CF'),
\end{multline*}
 as desired, where the only non-trivial isomorphism is that on the fourth line, and it takes place due to
 the fact the functor $(\on{unit}_I^!)^L$ commutes with !-tensor products with objects of $\Shv(X^I)$
 (by \thmref{t:averaging}).
 
\qed[\thmref{t:unit I}]

\sssec{}

The rest of this subsection is devoted to the proof of \propref{p:unit beta}. Let $k$ be the number of elements
in the partition $\fP$, not counting the element containing the distinguished point. Then $X^\fP\simeq (X-x)^k-\on{Diag}$,
where $\on{Diag}\subset (X-x)^k$ is the diagonal advisor. 

\medskip

We have
$$\CY_\fP\simeq \left(((X-x)^k-\on{Diag})\underset{\Ran(X)}\times \ol{S}{}^0_\Ran\right)\times \CY$$

\medskip

Note also that we have a canonical isomorphism 
$$\fL_\fP(N):=X^\fP\underset{X^I}\times \fL_I(N)\simeq 
\left(((X-x)^k-\on{Diag})\underset{\Ran(X)}\times \fL_\Ran(N)\right)\times \fL_x(N).$$

Consider the the open subset 
$$\left(((X-x)^k-\on{Diag})\underset{\Ran(X)}\times S^0_\Ran\right)\times \CY\subset 
\left(((X-x)^k-\on{Diag})\underset{\Ran(X)}\times \ol{S}{}^0_\Ran\right)\times \CY.$$

The assertion of \propref{p:unit beta} follows from the combination of the next two lemmas:

\begin{lem} \label{l:restr to open}
Restriction defines an equialence
\begin{multline*}
\Shv\left(\left(((X-x)^k-\on{Diag})\underset{\Ran(X)}\times \ol{S}{}^0_\Ran\right)\times \CY\right)^{\fL_\fP(N),\chi_\Ran}\to \\
\to \Shv\left(\left(((X-x)^k-\on{Diag})\underset{\Ran(X)}\times S^0_\Ran\right)\times \CY\right)^{\fL_\fP(N),\chi_\Ran}.
\end{multline*}
\end{lem}

\begin{lem}  \label{l:open to unit}
Restriction along $\on{unit}_\fP$ defines an equivalence
\begin{multline*}
\Whit(\CY_\fP)=
\Shv\left(\left(((X-x)^k-\on{Diag})\underset{\Ran(X)}\times S^0_\Ran\right)\times \CY\right)^{\fL_\fP(N),\chi_\Ran}\simeq \\
\simeq \Shv\left(((X-x)^k-\on{Diag})\times \CY\right)^{\fL_x(N),\chi_x}= \Whit(X^\fP\times \CY).
\end{multline*}
\end{lem}

\begin{proof}[Proof of \lemref{l:restr to open}]

We claim that the category
$$\Shv\left(\left(((X-x)^k-\on{Diag})\underset{\Ran(X)}\times (\ol{S}{}^0_\Ran-S^0_\Ran)\right)\times \CY\right)^{\fL_\fP(N),\chi_\Ran}$$
is zero. This follows in the same way as \propref{p:generated}(a).

\end{proof}

\begin{proof}[Proof of \lemref{l:open to unit}]

Follows from the fact that $\fL_\fP(N)'\subset \fL_\fP(N)$ identifies with
$$\left(((X-x)^k-\on{Diag})\underset{\Ran(X)}\times \fL^+_\Ran(N))\right)\times \fL_x(N)\subset
\left(((X-x)^k-\on{Diag})\underset{\Ran(X)}\times \fL_\Ran(N))\right)\times \fL_x(N).$$

\end{proof}

\ssec{Proof of \thmref{t:averaging}}  \label{ss:proof of average}

We are going to show that \thmref{t:averaging} follows from a Ran version of \thmref{t:Raskin}. 

\sssec{}

The desired left adjoint is given as the composition of 
\begin{equation} \label{e:stage one}
\Whit(X^I\times \CY)\hookrightarrow \Shv(X^I\times \CY)\overset{(\on{unit}_I)_!}\hookrightarrow \Shv(\CY_I),
\end{equation}
and the partially defined functor $\Av^{\fL_I(N),\chi_I}_!$.

\medskip

We need to show that $\Av^{\fL_I(N),\chi_I}_!$ is defined on the essential image of \eqref{e:stage one},
and commutes with !-tensoring by pullback of objects of $\Shv(X^I)$.

\begin{rem}
Note that \thmref{t:averaging} formally follows from \thmref{t:unit I}. 

\medskip

Let us also note that we can use an appropriately
defined functor $\Av^{\fL_I(N),\chi_I}_{*,\on{ren}}$ to construct a left inverse of the functor $\on{unit}_I^!$; the functor
$\Av^{\fL_I(N),\chi_I}_{*,\on{ren}}$ commutes with !-tensoring by objects of $\Shv(X^I)$ by construction. 
What is not a priori
clear is that $\Av^{\fL_I(N),\chi_I}_{*,\on{ren}}\circ (\on{unit}_I)_!$ is the left adjoint adjoint of $\on{unit}_I^!$. However, once we know 
\thmref{t:unit I}, we will obtain an isomorphism
$$\Av^{\fL_I(N),\chi_I}_{*,\on{ren}}\circ (\on{unit}_I)_!\simeq \Av^{\fL_I(N),\chi_I}_!\circ (\on{unit}_I)_!.$$

\end{rem}

\sssec{}

For $j\in \BZ^{\geq 0}$, let $I^j\subset \fL_x(G)$ be the subgroup defined in \secref{sss:adolesc}. As in \secref{sss:proof of enough},
the !-averaging functor 
$$\Av^{\fL_x(N),\chi_x}_!:\Shv(X^I\times \CY)\to \Whit(X^I\times \CY)$$
is defined on the essential image of 
$$\oblv_{I^j,\chi_x}:\Shv(X^I\times \CY)^{I^j,\chi_x}\to \Shv(X^I\times \CY),$$
and the essential images of the functors $\Av^{\fL_x(N),\chi_x}_!\circ \oblv_{I^j,\chi_x}$ generate
$\Whit(X^I\times \CY)$.  

\medskip

Moreover, the proof of \thmref{t:Raskin} shows that the functor $\Av^{\fL_x(N),\chi_x}_!\circ \oblv_{I^j,\chi_x}$ 
commutes with !-tensoring by pullback of objects of $\Shv(X^I)$. Hence, it suffices to check that for every $j$, the functor
$\Av^{\fL_I(N),\chi_I}_!$ is defined on the essential image of
\begin{equation} \label{e:forget j I}
(\on{unit}_I)_!\circ \oblv_{I^j,\chi_x}: \Shv(X^I\times \CY)^{I^j,\chi_x}\to \Shv(\CY_I),
\end{equation}
and commutes with !-tensoring by pullback of objects of $\Shv(X^I)$. 

\sssec{}

Let $G^{\omega^\rho}$ be the group-scheme over $X$ corresponding to automorphisms of the $G$-bundle $\CP^{\omega^\rho}_G$. 
I.e., $G^{\omega^\rho}$ is the twist of the constant group-scheme with fiber $G$ by the $G$-torsor $\CP^{\omega^\rho}_G$ using 
the adjoint action.

\medskip

Let $\fL^+_I(G)$ (resp., $\fL^+_I(G)'\subset \fL_I(G)$) be the group-scheme (resp., group ind-scheme) over $X^I$,
defined in the same way as in \secref{sss:rel loop groups} with $N^{\omega^\rho}$ replaced by $G^{\omega^\rho}$.

\medskip

We have the projection $\fL^+_I(G)'\to \fL_x(G)$, and let
$\fL^+_I(G)^j$ be the group subscheme of $\fL^+_I(G)'$ equal to the preimage of $I^j\subset \fL_x(G)$.

\begin{rem}
For a $k$-point $x_I$ given by a collection of $I$ distinct points $y_i$ of $X$, and the distinguished point $x$, 
the fiber of $\fL^+_I(G)^j$ over such $x_I$ is given by
$$\underset{i}\Pi\, \fL^+_{y_i}(G)\times I^j.$$
\end{rem}

\sssec{}

Let $\CY_I^{\on{big}}$ be the following ind-scheme over $X^I$: it classifies quadruples $(x_I,\CP_G,\gamma,\epsilon)$, where: 

\begin{itemize}

\item $x_I$ is a point of $X^I$: 

\item $\CP_G$ is a $G$-bundle on $X$; 

\item $\gamma$ is an identification of $\CP_G$ with $\CP_G^{\omega^\rho}$ over the complement of $x_I$; 

\item $\epsilon$ is a structure of level $n$ on $\CP_G$ at $x$.

\end{itemize} 

In other words, the difference between $\CY_I^{\on{big}}$ and $\CY_I$ is that we no longer require that the maps
\eqref{e:new Plucker} be regular away from $x_I$. The ind-scheme $\CY_I^{\on{big}}$ is acted on by $\fL_I(G)$.

\medskip

We have a closed embedding
$$\CY_I\hookrightarrow \CY_I^{\on{big}}.$$ 

The action of $\fL^+_I(G)'$ preserves the image of the composition
$$X^I\times \CY \overset{\on{unit}_I}\hookrightarrow \CY_I \to \CY_I^{\on{big}},$$ 
and the resulting action of $\fL^+_I(G)'$ on $X^I\times \CY$ factors through the projection $\fL^+_I(G)'\to \fL_x(G)$,
and the $\fL_x(G)$-action on $X^I\times \CY$ via the $\CY$-factor. 

\medskip

Hence, the essential image of the functor \eqref{e:forget j I}, composed with $\Shv(\CY_I)\hookrightarrow \Shv(\CY_I^{\on{big}})$, 
factors as 
$$\Shv(X^I\times \CY)^{I^j,\chi_x}\to \Shv(\CY^{\on{big}}_I)^{\fL^+_I(G)^j,\chi_x} \to \Shv(\CY_I^{\on{big}}),$$
where the second arrow is the forgetful functor.

\sssec{}

We define the full subcategory
$$\Whit(\CY_I^{\on{big}}) \subset \Shv(\CY_I^{\on{big}})$$
by the same procedure as for $\CY_I$.

\medskip

We now have the following extension of \thmref{t:Raskin}:

\begin{thm}  \label{t:Raskin I}
The partially defined functor 
$$\Av^{\fL_I(N),\chi_I}_!:\Shv(\CY_I^{\on{big}})\to \Whit(\CY_I^{\on{big}}),$$
left adjoint to the forgetful functor $\Whit(\CY_I^{\on{big}})\to \Shv(\CY_I^{\on{big}})$, 
is defined on the essential image of the forgetful functor 
$$\Shv(\CY^{\on{big}}_I)^{\fL^+_I(G)^j,\chi_x} \to \Shv(\CY_I^{\on{big}}),$$
and commutes with !-tensoring by pullback of objects of $\Shv(X^I)$. 
\end{thm}

In particular, \thmref{t:Raskin I} implies that $\Av^{\fL_I(N),\chi_I}_!$ is defined on the essential image of
$(\on{unit}_I)_!\circ \oblv_{I^j,\chi_x}$ and commutes with !-tensoring by pullback of objects of $\Shv(X^I)$. 

\medskip

A proof of \thmref{t:Raskin I} is a straightforward adaptation of the proof of \thmref{t:Raskin}, given in \secref{s:proof of Rask}.

\section{Generalizations}  \label{s:gen}

\ssec{Full level structure}  \label{ss:full level}

In this subsection we will show that, when considering the Whittaker category, 
one can replace $\CY:=\fL(G)/K$ by the loop group $\fL(G)$ itself. 

\sssec{}

Consider the ind-scheme $\fL(G)$, as acted on by itself on the left. Consider
the category $\Shv(\fL(G))$, see \secref{sss:sheaves on loop group}.

\medskip

Note that for any group subscheme $N'\subset \fL(N)$, the functor
$$\Av_*^{N',\chi}:\Shv(\fL(G))\to \Shv(\fL(G))$$
makes sense. 

\medskip

Hence, we can define 
$$\Whit(\fL(G))\subset \Shv(\fL(G))$$
as full subcategory consisting of objects $\CF$, for which the map
$$\Av_*^{N',\chi}(\CF)\to \CF$$
is an isomorphism for any $N'$. 

\sssec{}

Here is a more explicit description of the subcategory. By \secref{sss:sheaves on ind-pro}, 
\begin{equation} \label{e:shv on loop}
\Shv(\fL(G))\simeq \underset{n}{\on{lim}}\, \Shv(\fL(G)/K_n),
\end{equation}
where the transition functors 
\begin{equation} \label{e:* pushforward}
\Shv(\fL(G)/K_{n''})\to \Shv(\fL(G)/K_{n'}), \quad n''\geq n'
\end{equation}
are given by the operation of *-direct image. 

\medskip

In terms of this identification, we have
\begin{equation} \label{e:* Whit as limit}
\Whit(\fL(G))=\underset{n}{\on{lim}}\, \Whit(\fL(G)/K_n)\subset \underset{n}{\on{lim}}\, \Shv(\fL(G)/K_n).
\end{equation} 

\sssec{}

Note that for a DG category $\bC$, the full subcategory
$$\on{Funct}_{\on{cont}}(\Shv(\fL(G)),\bC)^{\fL(N),\chi}\subset \on{Funct}_{\on{cont}}(\Shv(\fL(G)),\bC)$$
makes sense. 

\medskip

Hence, we can also define 
$$\Whit(\fL(G))_{\on{co}}:=\Shv(\fL(G))_{\fL(N),\chi}.$$

\sssec{}

Here is a more explicit description of $\Whit(\fL(G))_{\on{co}}$. Recall that according to \secref{sss:limits and colimits},
in addition to the realization of $\Shv(\fL(G))$ given by \eqref{e:shv on loop}, we also have an identification
\begin{equation} \label{e:shv on loop colim}
\Shv(\fL(G))\simeq \underset{n}{\on{colim}}\, \Shv(\fL(G)/K_n),
\end{equation}
where the transition functors 
\begin{equation} \label{e:* pullback}
\Shv(\fL(G)/K_{n'})\to \Shv(\fL(G)/K_{n''}), \quad n''\geq n'
\end{equation} 
are given by the operation of *-pullback.

\medskip

It follows that we have a canonical equivalence:
\begin{equation} \label{e:Whit as colim}
\Whit(\fL(G))_{\on{co}}\simeq \underset{n}{\on{colim}}\, \Whit(\fL(G)/K_n)_{\on{co}},
\end{equation}
where the transition functors are induced by \eqref{e:* pullback}.

\sssec{}

Note now that the functor $\on{Ps-Id}$ (see \secref{ss:pseudo-id}) makes sense also as a functor
$$\Whit(\fL(G))_{\on{co}}\to \Whit(\fL(G)).$$

We claim:

\begin{thm}  \label{t:inv vs coinv inf}
The functor $\on{Ps-Id}:\Whit(\fL(G))_{\on{co}}\to \Whit(\fL(G))$ is an equivalence.
\end{thm} 

\begin{proof}

Note that the *-pushforward functors \eqref{e:* pushforward} induce functors
$$\Whit(\fL(G)/K_{n''})_{\on{co}}\to \Whit(\fL(G)/K_{n'})_{\on{co}},$$
which provide right adjoints to the *-pullback functors 
$$\Whit(\fL(G)/K_{n'})_{\on{co}}\to \Whit(\fL(G)/K_{n''})_{\on{co}}.$$

Hence, by \secref{sss:limits and colimits}, we can realize $\Whit(\fL(G))_{\on{co}}$
as 
$$\underset{n}{\on{lim}}\, \Whit(\fL(G)/K_n)_{\on{co}},$$
with the transition functors induced by \eqref{e:* pushforward}.

\medskip

In terms of this identification and \eqref{e:* Whit as limit}, the fuctor $\on{Ps-Id}$
is given by the family of functors
$$\on{Ps-Id}: \Whit(\fL(G)/K_n)_{\on{co}}\to \Whit(\fL(G)/K_n),$$
while the latter are isomorphisms by \thmref{t:inv vs coinv}. 

\end{proof}

\sssec{}

By the same token, we can consider the prestack $(\BunNb)^{G\on{-level}_{\infty\cdot x}}_{\infty\cdot x}$
(where we equip our $G$-bundle with a full level structure at $x$),
the category $\Shv((\BunNb)^{G\on{-level}_{\infty\cdot x}}_{\infty\cdot x})$ and its full subcategory
$$\Whit((\BunNb)^{G\on{-level}_{\infty\cdot x}}_{\infty\cdot x})
\subset \Shv((\BunNb)^{G\on{-level}_{\infty\cdot x}}_{\infty\cdot x}).$$

\medskip

Passing to the limit in \thmref{t:main}, we obtain:

\begin{thm}
The !-pullback functor along $\fL_x(G)\to (\BunNb)^{G\on{-level}_{\infty\cdot x}}_{\infty\cdot x}$ defines an equivalence
$$\Whit((\BunNb)^{G\on{-level}_{\infty\cdot x}}_{\infty\cdot x})\to \Whit(\fL_x(G)).$$
\end{thm}

\ssec{Multi-point version}

The local Whittaker category we have defined is attached to the formal disc $\wh\cD_x$ for some point $x$ on a curve $X$. 
We will now show how to generalize this by considering a \emph{parameterized multi-disc} $\wh\cD_{x_I}$, which lives
over $X^I$. 

\sssec{}  \label{sss:multi}

Fix a finite set $I$ and a map 
$$n_I:I\to \BZ^{\geq 0}, \quad i\mapsto n_i.$$

Let $\CY$ denote the following ind-scheme over $X^I$. For an affine test-scheme $S$, an $S$-point of $\CY$ is a datum of
$(x_I,\CP_G,\gamma,\epsilon)$, where:

\begin{itemize}  

\item $x_I$ is an $S$-point of $X^I$ (i.e., an $I$-tuple of $S$-points $x_i$ of $X$);

\item $\CP_G$ is a $G$-bundle on $\cD_{x_I}$ (equivalently, on $\wh\cD_{x_I}$);

\item $\gamma$ is an identification between $\CP_G$ and $\CP_G^{\omega^\rho}$ over $\ocD_{x_I}$;

\item $\epsilon$ is a trivialization of the restriction of $\CP_G$ to the subscheme $\Sigma\, n_i\cdot \on{Graph}_{x_i}\subset \wh\cD_{x_I}$
(we view each $\on{Graph}_{x_i}$ as a Cartier divisor on $S\times X$). 

\end{itemize} 

\sssec{}

The above ind-scheme $\CY$ is acted on by $\fL_I(G)$ and, in particular, by $\fL_I(N)$. Proceeding as in the single-point case, we 
can introduce the corresponding categories
$$\Whit(\CY):=\Shv(\CY)^{\fL_I(N),\chi_I} \text{ and } \Whit(\CY)_{\on{co}}:=\Shv(\CY)_{\fL_I(N),\chi_I},$$
and the functor
\begin{equation} \label{e:Ps multi}
\on{Ps-Id}:\Whit(\CY)_{\on{co}}\to \Whit(\CY).
\end{equation} 

\medskip

We also have a relative (over $X^I$) version of the ind-algebraic stack $(\BunNb)^{G\on{-level}_{n \cdot x}}_{\infty\cdot x}$,
denoted
$$(\BunNb)^{G\on{-level}_{n_I\cdot x_I}}_{\infty\cdot x},$$
and the corresponding full subcategory
$$\Whit((\BunNb)^{G\on{-level}_{n_I\cdot x_I}}_{\infty\cdot x})\subset \Shv((\BunNb)^{G\on{-level}_{n_I\cdot x_I}}_{\infty\cdot x}).$$

\medskip

A straightforward generalization of \thmref{t:main} gives:

\begin{thm}
Pullback along $\CY\to (\BunNb)^{G\on{-level}_{n_I\cdot x_I}}_{\infty\cdot x}$ defines an equivalence
$$\Whit((\BunNb)^{G\on{-level}_{n_I\cdot x_I}}_{\infty\cdot x}))\to \Whit(\CY).$$
\end{thm}

From here, as in the case of a single point, we obtain:

\begin{thm}
The functor \eqref{e:Ps multi} is an equivalence.
\end{thm} 

\sssec{}  \label{sss:inf level mult}

As in \secref{ss:full level}, the above constructions and assertions can be generalized to the case when the
function $n^I$ is allowed take value $\infty$ on some elements of $I$. 

\medskip

The resulting geometric objects
are inverse limits of the corresponding objects for finite values of $n_I$. 

\medskip

The resulting Whittaker
categories can be realized both as limits and as colimits of one corresponding to finite values of $n_I$. 

\ssec{``Abstract" Whittaker categories}

In this subsection we will study various versions of the Whittaker model of an \emph{abstract} category $\bC$, 
equipped with an action of $\fL(G)$. 

\sssec{}

Let $\bC$ be a category acted on by $\fL(G)$; see \secref{sss:action on categ}. For any group-subscheme
$N'\subset \fL(N)$, we can consider the functor
$$\Av^{N',\chi}_*:\bC\to \bC.$$

Hence, as in \secref{ss:full level}, we can consider the full category
$$\Whit(\bC):=\bC^{\fL(N),\chi}\simeq \underset{\alpha}{\on{lim}}\, \bC^{N^\alpha,\chi}=
\underset{\alpha}\cap\, \bC^{N^\alpha,\chi}\subset \bC,$$
where $N^\alpha$ are as on \eqref{e:loop N as colim}. 

\medskip

In addition, we can consider
$$\Whit(\bC)_{\on{co}}:=\bC_{\fL(N),\chi}\simeq \underset{\alpha}{\on{colim}}\, \bC_{N^\alpha,\chi}\simeq 
\underset{\alpha}{\on{colim}}\, \bC^{N^\alpha,\chi},$$
where the last colimit is formed using the transition functors
$$\Av^{N^{\alpha''}/N^{\alpha'},\chi}_*:\bC^{N^{\alpha'},\chi}\to \bC^{N^{\alpha''},\chi}, \quad N^{\alpha'}\subset N^{\alpha''}.$$

\medskip

In addition, we have a well-defined functor
$$\on{Ps-Id}:\Whit(\bC)_{\on{co}}\to \Whit(\bC).$$

\sssec{}

We will prove:

\begin{thm}   \label{t:abstract} \hfill

\smallskip

\noindent{\em(a)} For any $\bC$, equipped with an action of $\fL(G)$, 
 the functor $\on{Ps-Id}:\Whit(\bC)_{\on{co}}\to \Whit(\bC)$ is an equivalence. 

\smallskip

\noindent{\em(b)} The functor 
$$\bC\mapsto \Whit(\bC), \quad \fL(G)\mmod \to \on{DGCat}_{\on{cont}}$$
commutes with limits and colimits, and for an abstract DG category $\bC_0$, the naturally defined functor
$$\Whit(\bC)\otimes \bC_0\to \Whit(\bC\otimes \bC_0)$$
is an equivalence. 

\smallskip

\noindent{\em(c)} If $\bC$ is dualizable, then so is $\Whit(\bC)$ and we have a canonical equivalence 
$$\Whit(\bC)^\vee\simeq  \Whit(\bC^\vee).$$

\end{thm}

\sssec{}

The proof of \thmref{t:abstract} is based on the following observation. For a pair of categories $\bC_1$ and $\bC_2$ acted on by $\fL(G)$
we can consider the category. 
$$(\bC_1\otimes \bC_2)^{\fL(G)}.$$

It is a basic fact (see \thmref{t:coinvariants and limits, loop}) that for $G$ reductive, the functor
$$\bC_1,\bC_2\mapsto (\bC_1\otimes \bC_2)^{\fL(G)}$$
commutes with colimits in each variable.

\begin{rem}
In fact the above functor is canonically isomorphic to the functor 
$$\bC_1,\bC_2\mapsto  \bC_1\underset{\Shv(\fL(G))}\otimes \bC_2,$$
see \thmref{t:coinvariants and limits, loop}(b). 
\end{rem}

\sssec{}

Note that for a category $\bC$ acted on by $\fL(G)$, we have a canonical identification 
$$\bC\simeq (\Shv(\fL(G))\otimes \bC)^{\fL(G)},$$
where we consider $\Shv(\fL(G))$ as acted on by $\fL(G)$ via right multiplication. The above equivalence is
an equivalence of categories acted on by $\fL(G)$, where we endow $\Shv(\fL(G))\otimes \bC)^{\fL(G)}$
with a $\fL(G)$-action via left multiplication.

\medskip

We will prove:

\begin{prop} \label{p:Whit co ten}
The natural map
$$\Whit(\bC)_{\on{co}}\simeq \Whit((\Shv(\fL(G))\otimes \bC)^{\fL(G)})_{\on{co}}\to (\Whit(\Shv(\fL(G)))_{\on{co}}\otimes \bC)^{\fL(G)}$$
is an equivalence.
\end{prop}

\begin{proof}

We have:
$$\Whit(\bC)_{\on{co}}:=\underset{\alpha}{\on{colim}}\, \bC^{N^\alpha,\chi},$$
while
\begin{multline*}
(\Whit(\Shv(\fL(G)))_{\on{co}}\otimes \bC)^{\fL(G)}:=\left((\underset{\alpha}{\on{colim}}\, \Shv(\fL(G))_{N^\alpha,\chi}) \otimes \bC\right)^{\fL(G)}\simeq \\
\simeq \left(\underset{\alpha}{\on{colim}}\, \left(\Shv(\fL(G))_{N^\alpha,\chi}\otimes \bC\right)\right)^{\fL(G)}\overset{\text{commutation with colimits}}\simeq 
\underset{\alpha}{\on{colim}}\, \left(\left(\Shv(\fL(G))_{N^\alpha,\chi} \otimes \bC\right)^{\fL(G)}\right)\simeq \\
\simeq \underset{\alpha}{\on{colim}}\, (\Shv(\fL(G))\otimes \bC)^{(N^\alpha,\chi),\fL(G)}\simeq 
\underset{\alpha}{\on{colim}}\, \left((\Shv(\fL(G))\otimes \bC)^{\fL(G)}\right){}_{N^\alpha,\chi}\simeq 
\underset{\alpha}{\on{colim}}\, \bC_{N^\alpha,\chi},
\end{multline*}
as desired. 

\end{proof}

\begin{proof}[Proof of \thmref{t:abstract}]

Note that in addition to the equivalence of \propref{p:Whit co ten}, we have the tautological equivalence
$$\Whit(\bC)\simeq \Whit\left((\Shv(\fL(G))\otimes \bC)^{\fL(G)}\right)\simeq \left(\Whit(\Shv(\fL(G))\otimes \bC)\right)^{\fL(G)}.$$

We have a commutative diagram
$$
\CD
\Whit(\bC)_{\on{co}}   @>{\text{\propref{p:Whit co ten}}}>>  \left(\Whit(\Shv(\fL(G)))_{\on{co}}\otimes \bC\right)^{\fL(G)} @>{\sim}>> \left(\Whit(\Shv(\fL(G))\otimes \bC)_{\on{co}}\right)^{\fL(G)} \\
@VVV  & & @VVV  \\
\Whit(\bC)  &  @>{\sim}>>   &   \left(\Whit(\Shv(\fL(G))\otimes \bC)\right)^{\fL(G)}.
\endCD
$$

Now, the equivalence of \thmref{t:inv vs coinv inf} extends to an equivalence
$$\Whit(\Shv(\fL(G))\otimes \bC') _{\on{co}}\to \Whit(\Shv(\fL(G))\otimes \bC')$$
for any test DG category $\bC'$. Hence, the right vertical arrow in the above diagram is an equivalence.
Therefore, so is the left vertical arrow. 

\medskip

This proves point (a) of the theorem. 

\medskip

Point (b) follows from point (a): the functor 
$$\bC\mapsto \Whit(\bC)$$
manifestly commutes with limits (given the fact that the forgetful functor $\fL(G)\mmod \to \on{DGCat}_{\on{cont}}$
commutes with limits), while the functor 
$$\bC\mapsto \Whit(\bC)_{\on{co}}$$
manifestly commutes with limits (given the fact that the forgetful functor $\fL(G)\mmod \to \on{DGCat}_{\on{cont}}$
commutes with colimits). 

\medskip

To prove point (b), given (a), it suffices to establish a canonical isomorphism
$$\on{Funct}_{\on{cont}}(\Whit(\bC)_{\on{co}},\bC_0)\simeq \Whit(\bC^\vee)\otimes \bC_0$$
that functorially depends on the test DG category $\bC_0$.

\medskip

By definition, we have
$$\on{Funct}_{\on{cont}}(\Whit(\bC)_{\on{co}},\bC_0)\simeq 
\on{Funct}_{\on{cont}}(\bC,\bC_0)^{\fL(N),\chi}\simeq \Whit(\bC^\vee \otimes \bC_0),$$
and the assertion follows from point (b). 

\end{proof}

Let is also note:

\begin{cor}
The natural map
$$\Whit(\Shv(\fL(G)))\underset{\Shv(\fL(G))}\otimes \bC  \to \Whit(\Shv(\fL(G))\underset{\Shv(\fL(G))}\otimes \bC)\simeq \Whit(\bC)$$
is an equivalence.
\end{cor} 

\begin{proof}
Follows from \thmref{t:abstract} as the assertion is manifestly true for $\Whit(-)$ replaced by $\Whit(-)_{\on{co}}$,
\end{proof}

\sssec{The ultimate generalization}

We now fix a finite set $I$, and consider the group ind-scheme $\fL_I(G)$ over $X^I$. Consider the category $\fL_I(G)\mmod$,
whose objects are DG categories $\bC$ acted on by the monoidal category $\Shv(\fL_I(G))$. 

\medskip

Proceeding as above for $\bC\in \fL_I(G)\mmod$, we define the categories 
$$\Whit(\bC) \text{ and } \Whit(\bC)_{\on{co}}$$
and the functor
$$\on{Ps-Id}: \Whit(\bC)_{\on{co}}\to \Whit(\bC).$$

Using \secref{sss:inf level mult}, we prove the corresponding version of \thmref{t:abstract}
in this situation.

\appendix 

\section{Proof of \thmref{t:Raskin}}  \label{s:proof of Rask} 

We will essentially copy the proof from \cite[Sect. 2.12]{Ras}, adding a few details. 

\ssec{Compactification of the action map}

\sssec{}

Denote $\CY:=\fL(G)/K_n$; we will consider it as an ind-scheme equipped with an action of $\fL(G)$. Let
$\CF$ be an object of $\Shv(\CY)^{I^j\cap \fL(N),\chi}$. Then we have a well-defined 
$$\chi^!(\on{A-Sch})\tboxtimes \CF\in \Shv(\fL(N) \overset{I^j\cap \fL(N)}\times \CY),$$
where $\overset{H}\times$ means ``divide by the diagonal action of $H$". 

\medskip

The action of $\fL(N)$ on $\CY$ defines a map
$$\on{act}:\fL(N) \overset{I^j\cap \fL(N)}\times \CY\to \CY.$$

The object $\CF$ is $(\fL(N),\chi)$-adapted if the partially defined left adjoint of
$$\on{act}^!\otimes \on{Id}_\bC:\Shv(\CY)\otimes \bC\to \Shv(\fL(N) \overset{I^j\cap \fL(N)}\times \CY)\otimes \bC$$
is defined on objects of the form $(\chi^!(\on{A-Sch})\tboxtimes \CF)\otimes \bc$ for all $\bc\in \bC$ and equals 
$\on{act}_!(\chi^!(\on{A-Sch})\tboxtimes \CF)\otimes \bc$. 

%\medskip

%In what follows, in order to unburden the notation, we will take $\bC=\Vect$ and $\bc=\sfe$; the proof
%in the general case is literally the same.

\sssec{}

Let us introduce some short-hand notation: for $H\subset \fL(G)$ denote
$$H^j:=\on{Ad}_{\on{Ad}_{-j\check\rho(t)}(H)}.$$

Note that
$$\fL^+(N)^j=I^j\cap \fL(N).$$

\medskip

Consider the ind-scheme $\fL(G)$ equipped with the projection
\begin{equation} \label{e:loop to Gr}
\fL(G)\to \fL(G)/\fL^+(G)^j,
\end{equation}
where we note that $\fL(G)/\fL^+(G)^j$ is isomorphic to the affine Grassmannian.

\medskip

Taking the preimage of the $\fL(N)$-orbit through the origin in $\fL(G)/\fL^+(G)^j$, we obtain
a locally closed ind-subscheme, denoted
$$\fL(N)\fL^+(G)^j\subset \fL(G),$$
equipped with a free action of $\fL^+(G)^j$. Note that we have an identification
$$\fL(N)\fL^+(G)^j\simeq \fL(N)\overset{\fL^+(N)^j}\times \fL^+(G)^j.$$

\medskip

We can form the fiber product
$$\fL(N)\fL^+(G)^j \overset{\fL^+(G)^j}\times \CY,$$
equipped with a locally closed embedding into 
$$\fL(G)\overset{\fL^+(G)^j}\times \CY.$$

In particular, $\fL(N)\fL^+(G)^j \overset{\fL^+(G)^j}\times \CY$ is 
an ind-scheme of ind-finite type, and we have an isomorphism 
$$\fL(N) \overset{\fL^+(N)^j}\times \CY\simeq \fL(N)\fL^+(G)^j \overset{\fL^+(G)^j}\times \CY.$$

\medskip

Under the above identification, the map
$$\on{act}:\fL(N) \overset{\fL^+(N)^j}\times \CY\to \CY$$
equals the composition
$$\fL(N)\fL^+(G)^j \overset{\fL^+(G)^j}\times \CY\to \fL(G)\overset{\fL^+(G)^j}\times \CY\to \CY,$$
where the second arrow is given by the action of $\fL(G)$ on $\CY$. 

\sssec{}

Let 
$$\ol{\fL(N)/\fL^+(N)^j}\subset \fL(G)/\fL^+(G)^j$$
denote the closure of the $\fL(N)$-orbit $\fL(N)/\fL^+(N)^j$ through the origin in $\fL(G)/\fL^+(G)^j$.

\medskip

Let 
$$\ol{\fL(N)\fL^+(G)^j}\subset \fL(G)$$
denote the preimage of $\ol{\fL(N)/\fL^+(N)^j}$ under the projection \eqref{e:loop to Gr}. In other words,
$\ol{\fL(N)\fL^+(G)^j}$ is the closure of $\fL(N)\fL^+(G)^j$ in $\fL(G)$. 

\medskip

The group-scheme $\fL^+(G)^j$ still
acts freely on $\ol{\fL(N)\fL^+(G)^j}$ by right multiplication. We can form
$$\ol{\fL(N)\fL^+(G)^j} \overset{\fL^+(G)^j}\times \CY,$$
which is a closed ind-subscheme of $\fL(G)\overset{\fL^+(G)^j}\times \CY$
(and thus is an ind-scheme of ind-finite type). 

\medskip

We have an open embedding
$$\fL(N)\fL^+(G)^j \overset{\fL^+(G)^j}\times \CY \overset{j}\hookrightarrow 
\ol{\fL(N)\fL^+(G)^j} \overset{\fL^+(G)^j}\times \CY$$
and a map
$$\ol{\on{act}}:\ol{\fL(N)\fL^+(G)^j} \overset{\fL^+(G)^j}\times \CY\to \CY.$$

\sssec{}

We will deduce \thmref{t:Raskin} from the following observation:

\begin{prop}  \label{p:clean}
If $\CF\in \Shv(\CY)^{\fL(N)^j,\chi}$ lies in essential image of the forgetful functor
$$\oblv_{I^j/\fL(N)^j,\chi}:\Shv(\CY)^{I^j,\chi}\to \Shv(\CY)^{\fL(N)^j,\chi},$$
then for any $\bC$ and any $\bc\in \bC$, the functor left adjoint to 
$$(j^!\otimes \on{Id}_\bC):\Shv(\ol{\fL(N)\fL^+(G)^j} \overset{\fL^+(G)^j}\times \CY)\otimes \bC\to
\Shv(\fL(N)\fL^+(G)^j \overset{\fL^+(G)^j}\times \CY)\otimes \bC$$
is defined on 
$$(\chi^!(\on{A-Sch})\tboxtimes \CF)\otimes \bc\in \Shv(\fL(N) \overset{\fL^+(N)^j}\times \CY)\otimes \bC\simeq 
\Shv(\fL(N)\fL^+(G)^j \overset{\fL^+(G)^j}\times \CY)\otimes \bC$$
and equals 
$$j_*(\chi^!(\on{A-Sch})\tboxtimes \CF)\otimes \bc.$$
\end{prop}

Note that \propref{p:clean} says, in particular, that the object 
$$\chi^!(\on{A-Sch})\tboxtimes \CF\in \Shv(\fL(N)\fL^+(G)^j \overset{\fL^+(G)^j}\times \CY)$$
is \emph{clean} with respect to the map $j$. 

\sssec{}

Let us assume \propref{p:clean} and deduce that for $\CF$ in the essential image of $\oblv_{I^j/\fL(N)^j,\chi}$, the object 
$\on{act}_!(\chi^!(\on{A-Sch})\tboxtimes \CF)$ exists and 
$$\on{act}_!(\chi^!(\on{A-Sch})\tboxtimes \CF)\otimes \bc$$
provides the value of the left adjoint to $\on{act}^!\otimes \on{Id}_\bC$ on $(\chi^!(\on{A-Sch})\tboxtimes \CF)\otimes \bc$. 

\medskip

Using \propref{p:clean}, it suffices to show that that for \emph{any}
$$\wt\CF\in \Shv(\ol{\fL(N)\fL^+(G)^j} \overset{\fL^+(G)^j}\times \CY),$$
the object 
$$\ol{\on{act}}_!(\wt\CF)\in \Shv(\CY)$$ exists, and 
$$\ol{\on{act}}_!(\wt\CF)\otimes \bc\in \Shv(\CY)\otimes \bC$$ 
provides the value of the left adjoint to 
$$\ol{\on{act}}^!\otimes \on{Id}_\bC:\Shv(\CY)\otimes \bC\to \Shv(\ol{\fL(N)\fL^+(G)^j} \overset{\fL^+(G)^j}\times \CY)\otimes \bC$$
on $\wt\CF\otimes \bC$.

\sssec{}

In fact we claim that the left adjoint to $\ol{\on{act}}^!$  is given by $\ol{\on{act}}_*$ (which implies that the left adjoint to 
$\ol{\on{act}}^!\otimes \on{Id}_\bC$ is given by $\ol{\on{act}}_*\otimes \on{Id}_\bC$). 

\medskip  

Indeed, we claim that the morphism $\ol{\on{act}}$ is ind-proper. To show this, it is enough to show that the action morphism 
\begin{equation} \label{e:loop group on Y}
\fL(G)\overset{\fL^+(G)^j}\times \CY\to \CY
\end{equation} 
is ind-proper.  

\medskip

For this, we note that the automorphism
$$(g,y)\mapsto (g,g\cdot y)$$
isomorphes $\fL(G)\overset{\fL^+(G)^j}\times \CY$ to the product $\fL(G)/\fL^+(G)^j\times \CY$, and the action morphism 
\eqref{e:loop group on Y} gets transformed to the projection on the second factor. 

\medskip

Now the assertion follows from the fact that the ind-scheme $\fL(G)/\fL^+(G)^j$
is ind-proper (being isomorphic to the affine Grassmannian). 

\ssec{Proof of the cleanness statement}

In this subsection we will prove \propref{p:clean}. In order to unburden the notation we will take $\bC=\Vect$ and $\bc=\sfe$;
the proof in the general case is literally the same. 

\sssec{}

We need to show that objects of the form 
$$\chi^!(\on{A-Sch})\tboxtimes \CF$$
for $\CF$ in the essential image of 
$$\oblv_{I^j/\fL(N)^j,\chi}:\Shv(\CY)^{I^j,\chi}\to \Shv(\CY)^{\fL(N)^j,\chi}$$
are clean with respect to $j$.

\medskip

With no restriction of generality we can assume that $\CF$ is supported on a $\fL^+(G)^j$-stable (finite-dimensional)
subscheme $\CY'\subset \CY$. The action of $\fL^+(G)^j$ on such $\CY'$ factors through a quotient by a normal
subgroup $H\subset \fL^+(G)^j$.  

\medskip

In what follows, when we write
$$\ol{\fL(N)\fL^+(G)^j} \overset{\fL^+(G)^j}\times \CY$$
we will actually mean
$$\ol{\fL(N)\fL^+(G)^j}/H \overset{\fL^+(G)^j/H}\times \CY'.$$

When we will write
$$\ol{\fL(N)\fL^+(G)^j}\times \CY$$
we will actually mean 
$$\ol{\fL(N)\fL^+(G)^j}/H \times \CY'.$$

We perform this manipulation in order to emphasize that we are dealing with ind-schemes of
ind-finite type. However, we will omit $H$ and $\CY'$ from the notation in order to unburden the formulas.

\sssec{}

Consider the pullback of $\chi^!(\on{A-Sch})\tboxtimes \CF$ to
$$\fL(N)\fL^+(G)^j\times \CY$$
(see the above conventions)
along the projection
$$\fL(N)\fL^+(G)^j\times \CY\to \fL(N)\fL^+(G)^j \overset{\fL^+(G)^j}\times \CY\simeq
\fL(N) \overset{\fL^+(N)^j}\times \CY.$$

\medskip

Recall that $K_j$ denotes the $j$-th congruence subgroup in $\fL^+(G)$; and recall our notation
$$K_j^j\subset \fL^+(G)^j.$$

The assertion of \propref{p:clean} is obtained as a combination of the following two statements:

\begin{prop} \label{p:K-equiv}
For $\CF$ in the essential image of $\oblv_{I^j/\fL(N)^j,\chi}$, the pullback of $\chi^!(\on{A-Sch})\tboxtimes \CF$ 
to $\fL(N)\fL^+(G)^j\times \CY$ is $(\fL(N),\chi)$-equivariant with respect to the action 
$$n\cdot (n_1g,y)=(nn_1g,g)$$
and is $K^j_j$-equivariant with respect to the action 
$$k\cdot (n_1g,y)=(n_1gk^{-1},y).$$
\end{prop}

\begin{prop} \label{p:pure clean}
For any (ind-scheme) $\CY$, any $\wt\CF\in \Shv((\fL(N)\fL^+(G)^j)\times \CY)$ with the equivariance 
properties specified in \propref{p:K-equiv}, its *-extension to $\ol{\fL(N)\fL^+(G)^j}\times \CY$ equals
the !-extension. 
\end{prop} 

\begin{proof}[Proof of \propref{p:K-equiv}]

Let $\CF'$ denote the pullback of $\CF$ to $\fL^+(G)^j\times \CY$ along the action map
$$\fL^+(G)^j\times \CY\to \CY$$
(see our conventions). 

\medskip

We can write
$$\fL(N)\fL^+(G)^j\times \CY\simeq  \fL(N) \overset{\fL^+(N)^j}\times (\fL^+(G)^j\times \CY),$$
and with respect to this identification, the pullback of $\chi^!(\on{A-Sch})\tboxtimes \CF$ to $\fL(N)\fL^+(G)^j\times \CY$
goes over to
$$\chi^!(\on{A-Sch})\tboxtimes \CF'.$$

\medskip

This makes the assertion about $(\fL(N),\chi)$-equivariance is immediate. For the assertion regarding $K^j_j$-equivariance, 
it suffices to show that $\CF'$ is $K^j_j$-equivariant with respect to the action
$$k\cdot (g,y)=(g\cdot k^{-1},y).$$

\medskip

Note that $K^j_j\subset I^j$ and $\chi|_{K^j_j}$ is trivial. Hence, $\CF$ is obtained as pullback of an object $\CF''$ on the quotient stack
$K^j_j\backslash \CY$. Our $\CF'$ is thus the pullback of $\CF''$ along the composite map 
$$\fL^+(G)^j\times \CY\to \CY\to K^j_j\backslash \CY.$$

Hence, it suffices to show that the above composite map is $K^j_j$-\emph{invariant} for the above action of $K^j_j$ on $\fL^+(G)^j\times \CY$.
However, this follows from the normality of $K^j_j$ in $\fL^+(G)^j$.

\end{proof}

\ssec{Proof of \propref{p:pure clean}}

\sssec{}

The idea of the proof is that the category 
$$\Shv((\ol{\fL(N)\fL^+(G)^j}-\fL(N)\fL^+(G)^j)/K^j_j\times \CY)^{\fL(N),\chi)}$$
is zero (which follows from \lemref{l:character non-trivial} below). 

\medskip

However, that fact on its own does not seem to suffice for the proof
of the cleanness statement, because the functor 
$$\Av_*^{\fL(N),\chi}:\Shv(\ol{\fL(N)\fL^+(G)^j}/K^j_j\times \CY)\to
\Shv(\ol{\fL(N)\fL^+(G)^j}/K^j_j\times \CY)^{\fL(N),\chi)}$$
does not a priori preserve the subcategory of objects supported on
$$(\ol{\fL(N)\fL^+(G)^j}-\fL(N)\fL^+(G)^j)/K^j_j\subset \ol{\fL(N)\fL^+(G)^j}/K^j_j.$$

So we need a employ more delicate analysis. 

\sssec{}
Denote 
$$\CZ:=\ol{\fL(N)\fL^+(G)^j}/K^j_j\times \CY \text{ and } \CZ^0:=\fL(N)\fL^+(G)^j/K^j_j\times \CY.$$

%More generally, for $\lambda\in \Lambda^{\on{pos}}$, denote
%$$\CZ^\lambda:=\fL(N)\lambda(t)\fL^+(G)^j/K^j_j,$$
%so that the ind-schemes $\CZ^\lambda$ form a stratification of $\CZ$. 

\medskip

We will argue by contradiction, so let us be given a non-zero map
\begin{equation} \label{e:map to vanish}
j_*(\CF)\to \CF',
\end{equation}
where $\CF_1$ is supported on $\CZ-\CZ^0$.

\medskip

Let $Z\subset \CZ$ be a (finite-dimensional) subscheme of $\CZ$ such that the resulting map
\begin{equation} \label{e:map to vanish Z}
j_*(\CF)|_Z\to \CF'|_Z
\end{equation}
is non-zero. 

\medskip

Consider the intersection $Z\cap (\CZ-\CZ^0)$. We will prove:

\begin{lem} \label{l:character non-trivial}
There exists a large enough subgroup $N^\alpha\subset \fL(N)$ so that for every point $z$ in the intersection $Z\cap (\CZ-\CZ^0)$,
the restriction of the character $\chi$ to
$$\on{Stab}_{N^\alpha}(z)\subset N^\alpha\subset \fL(N)$$
is non-trivial. 
\end{lem}

The lemma will be proved below. Let is proceed with the proof of \propref{p:pure clean}.

\sssec{}
 
Since $\CF$ is $(\fL(N),\chi)$-equivariant, and in particular $(N^\alpha,\chi)$-equivariant, the map \eqref{e:map to vanish} factors
as $$j_*(\CF)\to \Av_*^{N_\alpha}(\CF')\to \CF'.$$

In particular, the map \eqref{e:map to vanish Z} factors as 
$$j_*(\CF)|_Z\to \Av_*^{N_\alpha}(\CF')|_Z\to \CF'|_Z.$$

\medskip

We will arrive to a contradiction by showing that 
$$\Av_*^{N_\alpha}(\CF')|_Z=0.$$

Indeed, \lemref{l:character non-trivial} implies that for any $\CF''\in \Shv(\CZ-\CZ^0)^{N^\alpha,\chi}$, the restriction $\CF''|_Z$ vanishes. 

\ssec{Proof of \lemref{l:character non-trivial}}

\sssec{Step 1}

Note that for any $z\in \CZ$, its stabilizer $\on{Stab}_{\fL(N)}(z)$ is a bounded subgroup in $\fL(N)$. Hence, given a finite-dimensional
$Z\subset \CZ$, there exists a large enough subgroup $N^\alpha\subset \fL(N)$ so that
$$\on{Stab}_{\fL(N)}(z)\subset N^\alpha,\,\,\forall z\in Z.$$

\medskip

Hence, to prove the lemma, it suffices to show that for any $z\in \CZ-\CZ^0$, 
the restriction of $\chi$ to $\on{Stab}_{\fL(N)}(z)$ is non-trivial. 

\sssec{Step 2}

Note that for the analysis of the stabilizers, the $\CY$ factor is irrelevant. Thus, let $z$ belong to 
$$\fL(N)\lambda(t)\fL^+(G)^j/K^j_j$$
with $0\neq \lambda\in -\Lambda^{\on{pos}}$. In particular, $\lambda$ is non-dominant. 

\medskip

Conjugating by an element of $\fL(N)$, we can further assume that $z\in  \lambda(t)\fL^+(G)^j/K^j_j\times \CY$. Furthermore, since 
$K^j_j$ is normal in $\fL^+(G)^j$. Hence, we can assume that $z=\lambda(t)$. 

\sssec{Step 3}

Note that 
$$\fL^+(N)\subset K^j_j,$$
hence
$$\on{Ad}_{-\lambda(t)}(\fL^+(N))\subset \on{Stab}_{\fL(N)}(z).$$

However, it is clear that for $\lambda$ non-dominant, the restriction of $\chi$ to $\on{Ad}_{-\lambda(t)}(\fL^+(N))$
is non-trivial. 

\section{Invariants vs coinvariants for group actions}  \label{s:coinv}

\ssec{The statement}

\sssec{}

Let $H$ be an algebraic group (of finite type). Let $\bC$ be a DG category equipped with an action of $H$,
which by definition means an action of the monoidal category $\Shv(H)$ (the monoidal structure is given
by *-convolution). 

\medskip

Consider the functor
\begin{equation} \label{e:Av functor}
\on{Av}^H_*:\bC\to \bC^H.
\end{equation}

The goal of this appendix is to prove the following result:

\begin{thm} \label{t:coinvariants}
The functor \eqref{e:Av functor} is universal among $H$-invariant functors from $\bC$ to categories 
equipped with the \emph{trivial} $H$-action. 
\end{thm}

Another way to state \thmref{t:coinvariants} is that the $H$-invariant functor \eqref{e:Av functor}
defines an equivalence 
\begin{equation} \label{e:inv vs coinv}
\bC_H\simeq \bC^H. 
\end{equation}

\sssec{An example}  \label{sss:H reg}

Take $\bC=\Shv(H)$. We have
$$\Vect\simeq \Shv(H)^H, \quad \sfe\mapsto \sfe_H.$$
The functor $\on{Av}^H_*: \Shv(H)\to \Shv(H)^H\simeq \Vect$ identifies with
$$\CF\mapsto \on{C}^\cdot(H,\CF).$$

This makes the assertion of  \thmref{t:coinvariants} manifest in this case.

\sssec{}

As a formal corollary of \thmref{t:coinvariants}, we obtain:

\begin{cor} \label{c:inv and colimits}  \hfill

\smallskip

\noindent{\em(a)} 
The functor 
$$\bC\mapsto \bC^H, \quad H\mmod\to \on{DGCat}_{\on{cont}}$$ commutes
with colimits.

\smallskip

\noindent{\em(b)}  The functor
$$\bC\mapsto \bC_H, \quad H\mmod\to \on{DGCat}_{\on{cont}}$$ commutes
with limits.
\end{cor}

\begin{proof}
The assertion about colimits is obvious for the functor $\bC\mapsto \bC_H$ and about limits for the functor $\bC\mapsto \bC^H$.
Now apply \eqref{e:inv vs coinv}.
\end{proof}

\sssec{}  \label{sss:limits enough}

Let also note that the conclusion of either point (a) or (b) of \corref{c:inv and colimits} implies \thmref{t:coinvariants}.
Let us prove this for point (b):

\begin{proof}

For any $\bC$ acted on by $H$, we have
$$\bC\simeq (\Shv(H)\otimes \bC)^H,$$
where $H$-invariants are taken with respect to the diagonal action on $\bC$ and the action on $\Shv(H)$
by right translations. The above equivalence respects the $H$-actions, where the action on the RHS comes from
the action on $\Shv(H)$ be left translations.

\medskip

In other words, we obtain that $\bC$ is isomorphic to the totalization of the cosimplicial DG category acted on by $H$ 
with terms 
$$\bC^n:=\Shv(H) \otimes \bC \otimes \Shv(H)^{\otimes n}, \quad n\geq 0.$$

As in Example \ref{sss:H reg}, the functor
$$(\bC^n)_H\to (\bC^n)^H$$
is an equivalence for every $n$. 

\medskip

We have a commutative diagram
$$
\CD
\bC_H  @>>>  \bC^H  \\
@V{\sim}VV  @VV{\sim}V  \\
\on{Tot}(\bC^\bullet)_H   @>>>  \on{Tot}(\bC^\bullet)^H \\
@VVV   @VV{\sim}V  \\
\on{Tot}((\bC^\bullet)_H)  @>{\sim}>> \on{Tot}((\bC^\bullet)^H). 
\endCD
$$

Assuming \corref{c:inv and colimits}(b), we obtain that the lower left vertical arrow is an equivalence.
Hence, $\bC_H\to \bC^H$ is also an equivalence, as desired. 

\end{proof}

\ssec{Locally constant actions}

\sssec{}  \label{sss:loc const}

For a scheme $Y$, let 
$$\Shv(Y)^0\subset \Shv(Y)$$
be the full subcategory generated by the constant sheaf $\sfe_Y\in \Shv(Y)$. Since  $\sfe_Y$ is compact, the tautological
embedding 
$$\Shv(Y)^0\hookrightarrow \Shv(Y)$$
admits a continuous right adjoint. 

\medskip

Let $\on{C}^\cdot(Y)$ denote the (commutative) algebra of cochains on $H$ (in our sheaf theory). I.e.,
$$\on{C}^\cdot(Y):=\CEnd_{\Shv(Y)}(\sfe_Y,\sfe_Y).$$

We have a canonical equivalence
$$\Shv(Y)^0\simeq \on{C}^\cdot(Y)\mod, \quad \CF\mapsto \CHom_{\Shv(Y)}(\sfe_Y,\CF)\simeq \on{C}^\cdot(Y,\CF).$$

\sssec{}

We take $Y=H$. Note that the subcategory 
$$\Shv(H)^0\hookrightarrow \Shv(H)$$
is preserved by the monoidal operation on $\Shv(H)$. Hence, $\Shv(H)^0$ acquires a monoidal structure.

\medskip

In terms of the identification
$$\Shv(H)^0\simeq \on{C}^\cdot(H)\mod,$$
this monoidal sructure corresponds to the structure of (commutative) Hopf algebra on $\on{C}^\cdot(H)$, given by the
group law on $H$.

\medskip

We note that the functor
\begin{equation} \label{e:coloc H}
\Shv(H)^0\twoheadleftarrow \Shv(H),
\end{equation} 
which is a priori lax-monoidal, is actually monoidal. This follows, e.g., from the fact that
$$\on{C}^\cdot(H,\CF_1\star \CF_2)\simeq \on{C}^\cdot(H,\CF_1)\otimes \on{C}^\cdot(H,\CF_2).$$
In particular, we obtain that $\Shv(H)^0$ is unital.

\sssec{}  \label{sss:loc const limits}

We obtain that $\Shv(H)^0$ is a retract of $\Shv(H)$ as a category acted on by $H$. In particular, we 
obtain that $\Shv(H)^0$ is dualizable \emph{as a $\Shv(H)$}-module category (see \cite[Chapter 1, Sect. 8.6]{GR1} for what this means).

\medskip

In particular, we obtain that the functor 
$$\bC\mapsto \Shv(H)^0\underset{\Shv(H)}\otimes \bC, \quad  \Shv(H)\mmod\to  \Shv(H)^0\mmod,$$
left adjoint to the restriction functor 
\begin{equation} \label{e:loc const actions}
\Shv(H)^0\mmod\to \Shv(H)\mmod=:H\mmod,
\end{equation}
commutes with \emph{limits}. 

\begin{rem}
Since the functor \eqref{e:coloc H} is a colocalization, we obtain that the functor \eqref{e:loc const actions},
and its left and right adjoints are isomorphic. Indeed, the right adjoint in question is given by
$$\bC\mapsto \on{Funct}_{\Shv(H)\mmod}(\Shv(H)^0\mmod,\bC).$$
However, the self-duality of $\Shv(H)$ as a left/right module category over itself implies that the dual of 
$\Shv(H)^0$ as a left $\Shv(H)$-module identifies with $\Shv(H)^0$ as a right $\Shv(H)$-module, so
\begin{equation} \label{e:ten with H0}
\on{Funct}_{\Shv(H)\mmod}(\Shv(H)^0\mmod,\bC)\simeq \Shv(H)^0\underset{\Shv(H)}\otimes \bC.
\end{equation}
\end{rem}

\begin{rem} \label{r:loc const act}
Note also that for $\bC$ as above, $\Shv(H)^0\underset{\Shv(H)}\otimes \bC$ is the colocalization of $\bC$,
and is the maximal full subcategory on which the action of $\Shv(H)$ factors through $\Shv(H)^0$.
\end{rem} 

\sssec{}

Since the functor
$$\Vect\to \Shv(H), \quad \sfe\mapsto \sfe_H$$
factors through $\Shv(H)^0$, the augmentation functor
$$\Shv(H) \to \Vect, \quad \CF\mapsto \on{C}^\cdot(H,\CF)$$
factors as
$$\Shv(H) \twoheadrightarrow \Shv(H)^0\to \Vect.$$

According to  \secref{sss:limits enough}, in order to prove \thmref{t:coinvariants}, it suffices
to show that the functor
$$\bC\mapsto \Vect\underset{\Shv(H)}\otimes \bC, \quad H\mmod\to \Vect$$
commutes with limits. We rewrite
$$\Vect\underset{\Shv(H)}\otimes \bC\simeq \Vect\underset{\Shv(H)^0}\otimes (\Shv(H)^0 \underset{\Shv(H)}\otimes \bC).$$
 
Hence, by \ref{sss:loc const limits}, in order to prove
\thmref{t:coinvariants}, it suffices to show that the functor
$$\bC'\mapsto \Vect\underset{\Shv(H)^0}\otimes \bC', \quad \Shv(H)^0\mmod\to \on{DGCat}_{\on{cont}}$$
commutes with limits. We will prove:

\begin{prop} \label{p:H0 limits}
The functor
$$\bC'\mapsto \bC^r\underset{\Shv(H)^0}\otimes \bC', \quad \Shv(H)^0\mmod\to \on{DGCat}_{\on{cont}}$$
commutes with limits for any right $\Shv(H)^0$-module category $\bC^r$ which is dualizable as a plain DG category.
\end{prop} 

\ssec{Rigidity}

\sssec{}  \label{sss:ren}

For a finite type scheme $Y$, let 
$$\on{C}^\cdot(Y)\mod^{\on{fin.dim}}\subset \on{C}^\cdot(Y)\mod$$
be the full (but not cocomplete) subcategory consisting of modules are that
are finite-dimensional\footnote{In particular, have finitely many non-zero cohomology groups.}
over the field of coefficients $\sfe$.

\medskip

Let $\on{C}^\cdot(Y)\mod^{\on{ren}}$ denote the ind-completion of $\on{C}^\cdot(Y)\mod^{\on{fin.dim}}$.
The tautological embedding
$$\on{C}^\cdot(Y)\mod^{\on{fin.dim}}\hookrightarrow \on{C}^\cdot(Y)\mod$$
gives rise to a continuous functor
$$\Psi:\on{C}^\cdot(Y)\mod^{\on{ren}}\to \on{C}^\cdot(Y)\mod.$$

Since $\on{C}^\cdot(Y)$ is finite-dimensional, the functor $\Psi$ admits a left adjoint, denoted $\Xi$, given by sending 
the compact generator
\begin{equation} \label{e:generator}
\on{C}^\cdot(Y)\in \on{C}^\cdot(Y)\mod
\end{equation} 
to $\on{C}^\cdot(Y)$ viewed as an object of $\on{C}^\cdot(Y)\mod^{\on{fin.dim}}\subset \on{C}^\cdot(Y)\mod^{\on{ren}}$.

\medskip

It is clear that the co-unit of the adjunction
$$\on{Id}\to \Psi\circ \Xi$$
is an isomorphism when evaluated on the generator \eqref{e:generator}. Hence $\Xi$ is fully faithful, and so 
$\Psi$ is a colocalization. 

\sssec{}

Take $Y=H$. The subcategory 
$$\on{C}^\cdot(H)\mod^{\on{fin.dim}}\subset \on{C}^\cdot(H)\mod$$
is preserved by the monoidal operation. Hence, 
$\on{C}^\cdot(H)\mod^{\on{ren}}$ acquires a monoidal structure so that the functor $\Psi$ is monoidal.

\medskip

Hence, the restriction functor
$$(\on{C}^\cdot(H)\mod)\mmod \to (\on{C}^\cdot(H)\mod^{\on{ren}})\mmod$$
is fully faithful, and for a pair of a left/right $\on{C}^\cdot(H)\mod$-module categories
$\bC^l$ and $\bC^r$, we have 
$$\bC^r\underset{\on{C}^\cdot(H)\mod}\otimes \bC^l\simeq \bC^r\underset{\on{C}^\cdot(H)\mod^{\on{ren}}}\otimes \bC^l.$$

\sssec{}

Hence, in order to prove \propref{p:H0 limits}, it suffices to show that the functor
$$\bC^l\mapsto \bC^r\underset{\Shv(H)^0}\otimes \bC^l, \quad (\on{C}^\cdot(H)\mod^{\on{ren}})\mmod
\to \on{DGCat}_{\on{cont}}$$
commutes with limits for any $\bC^r$ which is dualizable as a plain DG category.

\medskip

However, this follows from the fact that the monoidal category $\on{C}^\cdot(H)\mod^{\on{ren}}$ is 
\emph{rigid}, see  \cite[Chapter 1, Prop. 9.5.3]{GR1}.  

\qed[\thmref{t:coinvariants}]

\ssec{Recovering the category category from invariants}

\sssec{}

The functor
$$\bC\mapsto \bC^H, \quad H\mmod\to \on{DGCat}_{\on{cont}}$$
naturally upgrades to a functor
\begin{equation} \label{e:inv enh}
\bC\mapsto \bC^{H,\on{enh}},\quad H\mmod\to \Shv(\on{pt}/H)\mmod,
\end{equation}
where we note that
$$\Shv(\on{pt}/H)\simeq \on{Funct}_H(\Vect,\Vect).$$

\medskip

The functor \eqref{e:inv enh} is clearly not conservative. However, by construction,
it factors as the composition
$$H\mmod\to \Shv(H)^0\mmod, \quad \bC\mapsto 
\on{Funct}_{H\mmod}(\Shv(H)^0,\bC) \overset{\text{\eqref{e:ten with H0}}}\simeq  \Shv(H)^0 \underset{\Shv(H)}\otimes \bC$$
and the functor 
$$\Shv(H)^0\mmod\to \Shv(\on{pt}/H)$$
equal to the restriction of \eqref{e:inv enh} along the fully faithful embedding $\Shv(H)^0\mmod\to H\mmod$.

\sssec{}

We claim:

\begin{thm}  \label{t:reconstr}
The functor 
\begin{equation}  \label{e:enh 0}
\bC\mapsto \bC^{H,\on{enh}},\quad \Shv(H)^0\mmod \to \Shv(\on{pt}/H)\mmod
\end{equation}
is an equivalence of categories.
\end{thm}

The rest of this subsection is devoted to the proof of \thmref{t:reconstr}

\sssec{}

The functor left adjoint to \eqref{e:inv enh} is given by
\begin{equation} \label{e:Lloc}
\bD\mapsto \Vect \underset{\Shv(\on{pt}/H)}\otimes \bD.
\end{equation}

We claim that its essential image belongs to $\Shv(H)^0\mmod\subset H\mmod$. 
Indeed, since \eqref{e:Lloc}
commutes with colimits, it is enough to show this for the generator, i.e., $\bD=\Shv(\on{pt}/H)$, and
in this case the assertion is clear.

\medskip

We will show that both the unit and the counit of the adjunction are isomorphisms. 

\sssec{}

For $\bD$
as above, the unit of the adjunction is the canonical map
$$\bD\simeq \Shv(\on{pt}/H) \underset{\Shv(\on{pt}/H)}\otimes \bD \simeq 
\Vect^H  \underset{\Shv(\on{pt}/H)}\otimes \bD \to (\Vect \underset{\Shv(\on{pt}/H)}\otimes \bD)^H.$$

Now, by \corref{c:inv and colimits}(a), the last arrow in the above composition is an equivalence, as desired.

\sssec{}

Since the functor \eqref{e:inv enh} commutes with colimits, in order to show that that the counit of the adjunction
is an isomorphism, it is enough to do so when evaluated on $\bC=\Shv(H)^0$. In this case, the assertion
amounts to the fact that the functor
\begin{equation} \label{e:rec H0}
\Vect \underset{\Shv(\on{pt}/H)}\otimes \Vect\to \Shv(H)^0
\end{equation}
is an equivalence. 

\sssec{}

We have:
$$\Shv(\on{pt}/H)\simeq \on{C}_\cdot(H)\mod,$$
where the algebra structure on $\on{C}_\cdot(H)$ is given by the group law on $H$. The (symmetric) monoidal structure 
on $\on{C}_\cdot(H)\mod$ is given by the structure of (cocommutative) Hopf algebra on $\on{C}_\cdot(H)$.

\medskip

As in \secref{sss:ren}, we can write $\on{C}_\cdot(H)\mod$ as a (symmetric) monoidal colocalization of
the category $\on{C}_\cdot(H)\mod^{\on{ren}}$, which is equivalent to
$\on{C}^\cdot(\on{pt}/H)\mod$. 

\medskip

Write $\on{C}^\cdot(\on{pt}/H)\simeq \Sym(\fa)\mod$. Then 
$$\Vect \underset{\Shv(\on{pt}/H)}\otimes \Vect\simeq \Vect\underset{ \Sym(\fa)\mod}\otimes \Vect\simeq \Sym(\fa[1])\mod.$$

The desired equivalence \eqref{e:rec H0} follows from the identification
$$\on{C}^\cdot(H)\simeq \Sym(\fa[1]),$$
given by transgression. 

\qed[\thmref{t:reconstr}]

\ssec{The maximal subcategory with a locally constant action}  \label{ss:max}

\sssec{}

Let $\bC$ be equipped with an action of $G$. Set
$$\bC_{\on{l.c.}}:=\Shv(H)^0\underset{\Shv(H)}\otimes \bC.$$

The adjunction
$$\Shv(H)^0\rightleftarrows \Shv(H)$$
as $H$-module categories defines an adjunction
\begin{equation} \label{e:max lc}
\bC_{\on{l.c.}}\rightleftarrows \bC
\end{equation} 
as $H$-module categories.

\medskip

In particular, we obtain that $\bC_{\on{l.c.}}$ is a colocalization of $\bC$
as a plain DG category. 

\medskip

Thus, we can think of  $\bC_{\on{l.c.}}$ as the maximal sub/quotient category of $\bC$
on which the action of $H$ is locally constant.

\sssec{}

Note that the functors in \eqref{e:max lc} induce equivalences on the corresponding categories
of $H$-coinvariants, and hence invariants
$$\bC_{\on{l.c.}}^H\simeq \bC^H.$$

It is is easy to see from the constructions that the resulting colocalization functor on $\bC$ can be explicitly described 
as follows 
\begin{equation} \label{e:coloc expl}
\bc\mapsto \sfe\underset{\on{C}^\cdot(H)}\otimes\oblv_H\circ \on{Av}^H_*(\bc).
\end{equation} 

Indeed, by construction, the functor \eqref{e:coloc expl} takes values in $\bC_{\on{l.c.}}\subset \bC$; hence
by \thmref{t:reconstr}, it suffices to show that it induces the identity endo-functor on $\bC^H$, which is immediate. 

\section{Sheaf theory in infinite type}  \label{s:placid}

In this section we collect miscellanea related to the definition of the category of sheaves on ``infinite-dimensional"
algebro-geometric objects. We will the use this to define the notion of action of the loop group $\fL(G)$ on a DG
category. 

\ssec{Placid (ind-)schemes}

Although one can, in principle, define the category $\Shv(\CZ)$ for \emph{any} $k$-scheme (or even prestack) $\CZ$,
the result would be rather unwieldy. In this subsection we single out a certain class of schemes (we call them \emph{placid}),
and for which the category $\Shv(\CZ)$ is manageable. 

\medskip

The main point of the notion of placidity is that it is a \emph{property} and \emph{not} extra structure on a scheme. 

\sssec{}

Let $\CZ$ be a scheme over $k$, but not necessarily of finite type. We shall say that $\CZ$ is \emph{placid} if $\CZ$
can be written as filtered limit
\begin{equation} \label{e:present placid}
\CZ\simeq \underset{\alpha}{\on{lim}}\, Z_\alpha,
\end{equation} 
where $Z_\alpha$ are schemes of finite type, and the transition maps $Z_\alpha\to Z_\beta$ are affine, smooth and surjective.

\sssec{}  \label{sss:indep placid}

It is not difficult to show show that if $\CZ$ is placid, then the category of presentations of $\CZ$ as \eqref{e:present placid}
has an initial object, and in particular is contractible. So any two presentations \eqref{e:present placid} are essentially
equivalent.

\sssec{}

Let $\CZ$ be a placid scheme and $\CZ'\subset \CZ$ a closed subscheme. We shall see that this closed embedding
is placid if for some/any presentation of $\CZ$ as \eqref{e:present placid}, there exists an index $i$ and a closed
subscheme $Z'_\alpha\subset Z_\alpha$ so that
$$\CZ'=Z'_\alpha\underset{Z_\alpha}\times \CZ.$$

\sssec{}

Let $\CY$ be an ind-scheme (not necessarily of ind-finite type). We shall say that $\CY$ is placid if it can be written
as a filtered colimit
\begin{equation} \label{e:placid indsch}
\underset{i}{\on{colim}}\, \CZ_i,
\end{equation} 
where $\CZ_j$ are placid schemes, and the transition maps $\CZ_i\to \CZ_j$ are placid closed embeddings.

\sssec{}  \label{sss:indep placid indsch}

It is not difficult to show that  the category of presentations of $\CY$ as \eqref{e:placid indsch}
has a final object, and in particular is contractible. So any two presentations \eqref{e:placid indsch} are essentially
equivalent.

\sssec{}  \label{sss:ind-pro}

Let $\alpha\mapsto \CY_\alpha$ be a filtered family of ind-schemes of ind-finite type with transition maps 
$f_{\alpha,\beta}:\CY_\alpha\to \CY_\beta$ affine, smooth and surjective. With no restriction of generality, we can assume that
the index category $A$ has an initial object $\alpha_0$. 

\medskip

Set $$\CY_{\alpha_0}\simeq \underset{i\in I}{\on{colim}}\, Y_i,$$
for a filtered category $I$, where $Y_i$ are schemes of finite type, and the transition maps
$Y_i\to Y_j$ are closed embeddings. 

\medskip

Set $\CZ_i:=\underset{\alpha}{\on{lim}}\, Y_i\underset{\CY_{\alpha_0}}\times \CY_\alpha$. Then 
$\CZ_i$ is a placid scheme, and for $i\leq j$, the corresponding map $\CZ_i\to \CZ_j$
is a placid closed embedding. 

\medskip

Set 
$$\CY:=\underset{i\in I}{\on{colim}}\, \CZ_i.$$

Then $\CY$ is a placid ind-scheme. 

\ssec{The category of sheaves on a placid (ind-)scheme}

\sssec{}

For a placid scheme $\CZ$ presented as in \eqref{e:present placid}
we let
$$\Shv(\CZ):=\underset{\alpha}{\on{colim}}\, \Shv(Z_\alpha),$$
where for a $Z_\alpha\to Z_\beta$, the corresponding functor 
$$\Shv(Z_\beta)\to \Shv(Z_\alpha)$$ 
is the *-pullback.

\medskip

By \secref{sss:indep placid}, this definition is canonically independent of the presentation. 

\begin{rem}
As the functor of *-pullback is not t-exact (in the perverse t-structure), the category $\Shv(\CZ)$ does \emph{not} 
come equipped with a t-structure. However, since we are taking *-pullbacks with respect to smooth maps, which
are t-exact up to a cohomological shift, one can define a certain $\BZ$-gerbe on $\CZ$, called \emph{the dimension gerbe},
such that a choice of its trivializations gives rise to a t-structure on $\CZ$. 
We will not pursue this in the present paper. 

\medskip

Similarly, if our sheaf theory is that of D-modules, one may wish to construct the category $\Dmod(\CY)$ equipped with a 
forgetful functor to an appropriately defined version of the category $\IndCoh(\CY)$. A choice of such a functor involves 
a trivialization of certain Picard gerbe. We will not pursue this in the present paper either. 
\end{rem} 

\sssec{}

Note that by \secref{sss:limits and colimits}, we have a canonical isomorphism 
$$\Shv(\CZ):=\underset{\alpha}{\on{lim}}\, \Shv(Z_\alpha),$$
where for $Z_\alpha\to Z_\beta$, the corresponding functor $\Shv(Z_\alpha)\to \Shv(Z_\beta)$ 
is the *-pushforward. 

\sssec{}

The latter presentation that the assignment
$$\CZ\mapsto \Shv(\CZ)$$
is functorial with respect to *-pushforwards. Explicitly, for a map $f:\CZ\to \CZ'$
written as 
$$\CZ\simeq \underset{\alpha}{\on{lim}}\, Z_\alpha \text{ and } \CZ'\simeq \underset{\alpha'}{\on{lim}}\, Z'_{\alpha'},$$
respectively, the corresponding functor
$$f_*:\Shv(\CZ)\to \Shv(\CZ')$$
is characterized by the property that for every $i'$ the composition
$$\Shv(\CZ)\to \Shv(\CZ')\to \Shv(Z'_{\alpha'})$$
equals 
$$\Shv(\CZ)\to \Shv(Z_\alpha)\overset{(f_{\alpha,\alpha'})_*}\longrightarrow \Shv(Z'_{\alpha'}),$$
where $i$ is some/any index such that the composite $\CZ\overset{f}\to \CZ'\to Z'_{\alpha'}$
factors as 
$$\CZ\to Z_\alpha\overset{f_{\alpha,\alpha'}}\longrightarrow Z'_{\alpha'}.$$

\sssec{}

Let $f:\CZ'\to \CZ$ be a placid closed embedding. It follows from base change that the functor
$$f_*:\Shv(\CZ')\to \Shv(\CZ)$$
admits a continuos right adjoint, to be denoted $f^!$. 

\medskip

Explicitly, if $\CZ'=Z'_{\alpha_0}\underset{Z_{\alpha_0}}\times \CZ$ for some index $i$, then $f^!$ is given by the
compatible family of functors
$$f_\alpha^!:\Shv(Z_\alpha)\to \Shv(Z'_\alpha), \quad Z'_\alpha:=Z'_{\alpha_0}\underset{Z_{\alpha_0}}\times Z_\alpha, \quad \alpha\geq \alpha_0.$$

\sssec{}

Let $\CY$ be a placid ind-scheme, presented as in \eqref{e:placid indsch}. We define
$$\Shv(\CY):=\underset{i}{\on{lim}}\, \Shv(\CZ_i),$$
with respect to the !-pullback functors. By \secref{sss:indep placid indsch}, the category $\Shv(\CY)$
defined in this way does not depend on the choice of presentation \eqref{e:placid indsch}.

\medskip

By \secref{sss:limits and colimits} we can also write
\begin{equation} \label{e:shv placid colim}
\Shv(\CY):=\underset{i}{\on{colim}}\, \Shv(\CZ_i),
\end{equation} 
with respect to the *-pushforward functors. 

\medskip

In particular, we obtain that $\Shv(\CY)$ is compactly generated. 

\sssec{}  \label{sss:functoriality of shv on placid indsch}

The presentation \eqref{e:shv placid colim} implies that if $f:\CY_1\to \CY_2$ is a map between placid ind-schemes,
we have a well-defined functor
$$f_*:\Shv(\CY_1)\to \Shv(\CY_2).$$

\sssec{}  \label{sss:prod placid}

Let $\CY_1$ and $\CY_2$ be two placid ind-schemes. In this case, $\CY_1\times \CY_2$ is also 
placid, and we have a canonically defined fully faithful functor
$$\Shv(\CY_1)\otimes \Shv(\CY_2)\to \Shv(\CY_1\times \CY_2),$$
which preserves compactness. 

\sssec{}  \label{sss:sheaves on ind-pro}

Let us be in the situation of \secref{sss:ind-pro}. The direct image functors
$$(f_{\alpha,\beta})_*:\Shv(\CY_\alpha)\to \Shv(\CY_\beta)$$
admit left adjoints, $f_{\alpha,\beta}^*$.

\medskip

By swapping the order of limits, we obtain that the projections
$$f_\alpha:\CY\to \CY_\alpha,$$
and the corresponding functors
$$(f_\alpha)_*:\Shv(\CY)\to \Shv(\CY_\alpha)$$
give rise to an equivalence
$$\Shv(\CY) \simeq \underset{\alpha}{\on{lim}}\, \Shv(\CY_\alpha),$$
where the limit is taken with respect to $(f_{\alpha,\beta})_*$ as transition functors. 

\medskip

By \secref{sss:limits and colimits}, we obtain that we also have an equivalence
$$\Shv(\CY) \simeq \underset{\alpha}{\on{colim}}\, \Shv(\CY_\alpha),$$
where the colimit is taken with respect to $(f_{\alpha,\beta})^*$ as transition functors. 

\medskip

In particular, we have well-defined functors
$$f_\alpha^*:\Shv(\CY_\alpha)\to \Shv(\CY).$$

\ssec{Sheaves on the loop group} 

\sssec{}

Consider the group ind-scheme $\fL(G)$. We claim that it is placid as an ind-scheme. Namely, we claim
that it falls in the paradigm of \secref{sss:ind-pro}. 

\medskip

Indeed, we take the category $A$ to be natural numbers, and we set
$$\CY_n:=\fL(G)/K_n.$$

\sssec{} \label{sss:sheaves on loop group}

In particular, we obtain that we have a well-defined category $\Shv(\fL(G))$.

\begin{rem}
We emphasize again that being a placid ind-scheme is a property and not extra structure.
So, accessing $\fL(G)$ via the schemes $K_n\backslash \fL(G)$ will lead to an equivalent
definition of the category of sheaves.
\end{rem} 

\sssec{}   \label{sss:loop group acting}

By virtue of Sects. \ref{sss:prod placid} and \ref{sss:functoriality of shv on placid indsch}, the group structure on $\fL(G)$
defines on $\Shv(\CY)$ a structure of monoidal category,

\medskip

Furthermore, if $\CY$ is another placid ind-scheme equipped with an action of $\fL(G)$, the 
category $\Shv(\CY)$ acquires an action of $\Shv(\fL(G))$. 

\sssec{}

The monoidal category $\Shv(\fL(G))$ is unital, where the unit object is $\delta_1$, i.e., the direct image 
of $\sfe$ under the unit map $\on{pt}\to \fL(G)$.  

\medskip

Note, however, that we have a canonical identification
\begin{equation} \label{e:delta as colim}
\delta_1\simeq \underset{n}{\on{colim}}\, \sfe_{K_n},
\end{equation} 
where by a slight abuse of notation we denote by $\sfe_{K_n}$ the direct image
of the constant sheaf under the tautological map $K_n\to \fL(G)$. 

\section{Invariants and coinvariants for loop group actions}  \label{s:loop coinv}

\ssec{Categories acted on by the loop group}  \label{ss:action on categ}

\sssec{}  \label{sss:action on categ}

By a category acted on by $\fL(G)$ we will mean a module category over $\Shv(\fL(G))$. We will denote the
totality of such categories by $\fL(G)\mmod$. 

\medskip

We will use a similar notation for $\fL^+(G)$ or $K_n$. We have the natural restriction functors
$$\fL(G)\mmod\to \fL^+(G)\mmod\to K_n\mmod.$$

\sssec{}

Note that for every $n\geq 1$, the object $\sfe_{K_n}\in \Shv(\fL(G))$ is an idempotent. For $\bC\in \fL(G)\mmod$,
let $\bC^{K_n}$ the image of that idempotent. 

\medskip

From \eqref{e:delta as colim}, we obtain that
$$\bC\simeq \underset{n}{\on{colim}}\, \bC^{K_n},$$
where the transition maps are the natural inclusions. 

\medskip

By \secref{sss:limits and colimits} we can also write
$$\bC\simeq \underset{n}{\on{lim}}\, \bC^{K_n},$$
where the transition maps 
$$\bC^{K_n}\to \bC^{K_m}$$
are the averaging functors $\on{Av}_*^{K_m/K_n}$. 

\sssec{}

The basic example of an object of $\fL(G)\mmod$ is $\Vect$, which is acted upon by $\Shv(\fL(G))$ via the 
augmentation functor
$$\Shv(\fL(G))\to \Shv(\on{pt})=\Vect,$$
given by direct image. 

\medskip

For $\bC\in \fL(G)\mmod$, set
$$\bC^{\fL(G)}:=\on{Funct}_{\fL(G)\mmod}(\Vect, \bC).$$

Similarly, let $\bC_{\fL(G)}$ be the universal recipient of a $\fL(G)$-invariant functor. 

\medskip

We have the following result: 

\begin{thm}  \label{t:coinvariants and limits, loop}
Let $G$ be reductive. Then:

\smallskip

\noindent{\em(a)} The functor
$$\bC\mapsto \bC^{\fL(G)}, \quad \fL(G)\mmod\to \on{DGCat}_{\on{cont}}$$
commutes with limits. 

\smallskip

\noindent{\em(b)} There exists a canonical isomorphism $\bC^{\fL(G)}\simeq \bC_{\fL(G)}$.

\smallskip

\noindent{\em(c)} If $\bC$ is dualizable as a plain category, then so is $\bC^{\fL(G)}$, and we have a 
canonical equivalence $(\bC^{\fL(G)})^\vee\simeq (\bC^\vee)^{\fL(G)}$. 

\end{thm} 

The rest of this section is devoted to the proof of \thmref{t:coinvariants and limits, loop}. First, let us note
that points (a) and (c) both follow from point (b). 

\ssec{Proof of \thmref{t:coinvariants and limits, loop}}   \label{ss:proof in loop}

\sssec{}

Consider the functor
\begin{equation} \label{e:plus inv}
\bC\mapsto \bC^{\fL^+(G)}, \quad \fL(G)\mmod\to \on{DGCat}_{\on{cont}},
\end{equation} 

The first observation is that this functor commutes with colimits. Indeed, the functor
$$\bC\mapsto \bC^{K_1},\quad \fL(G)\mmod\to \on{DGCat}_{\on{cont}}$$
commutes with colimits, because it is given by the image of an idempotent. Now, 
the above functor naturally lifts to a functor
$$\fL(G)\mmod\to G\mmod,$$
and we have 
$$\bC^{\fL^+(G)}\simeq (\bC^{K_1})^G.$$

Hence, the commutation with colimits follows from \corref{c:inv and colimits}(a). 

\sssec{}

Note that the functor \eqref{e:plus inv} can also be interpreted as 
$$\bC\mapsto \on{Funct}_{\fL(G)\mmod}(\Shv(\fL(G)/\fL^+(G)),\bC).$$

Set
$$\BH:=\on{Funct}_{\fL(G)\mmod}(\Shv(\fL(G)/\fL^+(G)),\Shv(\fL(G)/\fL^+(G))).$$

This is the Hecke category of $\fL(G)$ with respect to $\fL^+(G)$. As a plain DG category we can identify it with
$$\Shv(\fL(G))^{\fL^+(G)\times \fL^+(G)},$$
with the monoidal structure given by convolution.

\medskip

Hence, the functor \eqref{e:plus inv} upgrades to a functor
\begin{equation} \label{e:plus inv enh}
\bC\mapsto \bC^{\fL^+(G),\on{enh}}, \quad  \fL(G)\mmod\to\BH\mmod.
\end{equation} 

%\begin{rem}
%We do not distinguish between left and right modules for $\BH$: the anti-involution on $\BH$
%given by the inversion on $\fL(G)$ allows to identify the two.
%\end{rem}

\sssec{}

The functor \eqref{e:plus inv enh} admits a left adjoint, given by 
\begin{equation} \label{e:tensor up from Hecke}
\bC'\mapsto \Shv(\fL(G)/\fL^+(G))\underset{\BH}\otimes \bC'.
\end{equation}

We claim:

\begin{prop}  \label{p:right adj to Hecke}
The functor 
\begin{equation} \label{e:Hom up from Hecke}
\bC'\mapsto \on{Funct}_{\BH\mmod}(\fL^+(G)\backslash\fL(G)),\bC'),
\end{equation}
provides a right adjoint to \eqref{e:plus inv enh}.
\end{prop}

\begin{proof}

We need to show that for $\bC\in \fL(G)\mmod$, there exists a canonical equivalence
$$\on{Funct}_{\fL(G)\mmod}(\bC,\on{Funct}_{\BH\mmod}(\Shv(\fL^+(G)\backslash\fL(G)),\bC')) \simeq
\on{Funct}_{\BH\mmod}(\bC^{\fL^+(G),\on{enh}},\bC').$$

We rewrite the LHS as
$$\on{Funct}_{\BH\mmod}(\Shv(\fL^+(G)\backslash\fL(G)) \underset{\Shv(\fL(G)}\otimes \bC,\bC'),$$
hence it remains to establish an equivalence
$$\Shv(\fL^+(G)\backslash\fL(G)) \underset{\Shv(\fL(G)}\otimes \bC\simeq \bC^{\fL^+(G),\on{enh}}$$
as $\BH$-modules.

\medskip

We rewrite $\Shv(\fL^+(G)\backslash \fL(G))\simeq \Shv(\fL(G))^{\fL^+(G),\on{enh}}$ as categories acted on by $\fL(G)$ on the right and by $\BH$
on the left. We have a map 
$$\Shv(\fL(G))^{\fL^+(G),\on{enh}} \underset{\Shv(\fL(G)}\otimes \bC \to 
\left(\Shv(\fL(G))\underset{\Shv(\fL(G)}\otimes \bC\right)^{\fL^+(G),\on{enh}} \simeq  \bC^{\fL^+(G),\on{enh}}.$$

To show that this map is an equivalence, we need to show that the first arrow is an equivalence at the level
of the underlying DG categories. However, this follows from the commutation of the functor \eqref{e:plus inv}
with colimits.

\end{proof} 

\begin{rem}
The commutation of the functor \eqref{e:plus inv} with colimits implies that the functor \eqref{e:tensor up from Hecke}
is fully faithful. It follows formally that the functor \eqref{e:Hom up from Hecke} is fully faithful as well.
\end{rem}

\sssec{}

Consider $\Vect^{\fL^+(G)}$ as an object of $\BH\mmod$. We claim:

\begin{prop} \hfill  \label{p:recover Vect}

\smallskip

\noindent{\em(a)}
The functor \eqref{e:tensor up from Hecke} sends $\Vect^{\fL^+(G)}\in \BH\mmod$ to $\Vect\in \fL(G)\mmod$.

\smallskip

\noindent{\em(b)}
The functor \eqref{e:Hom up from Hecke} sends $\Vect^{\fL^+(G)}\in \BH\mmod$ to $\Vect\in \fL(G)\mmod$.

\end{prop} 

We will prove \propref{p:recover Vect} in \secref{ss:recover Vect}. We proceed with the proof of \thmref{t:coinvariants and limits, loop}. 

\begin{cor} \label{c:inv via Hecke} \hfill

\smallskip

\noindent{\em(a)} 
For $\bC\in \fL(G)\mmod$, we have a canonical isomorphism
$$\bC^{\fL(G)}\simeq \on{Funct}_{\BH\mmod}(\Vect^{\fL^+(G)},\bC^{\fL^+(G),\on{enh}}).$$

\smallskip

\noindent{\em(b)} 
For $\bC\in \fL(G)\mmod$, we have a canonical isomorphism
$$\bC_{\fL(G)}\simeq \bC^{\fL^+(G),\on{enh}}\underset{\BH}\otimes \Vect^{\fL^+(G)}.$$

\end{cor}

Thus, from the above corollary we obtain that in order to prove \thmref{t:coinvariants and limits, loop}(b),
it remains to show the following:

\begin{prop}  \label{p:inv dualizable}
The object $\Vect^{\fL^+(G)}\in \BH\mmod$ is dualizable and self-dual.
\end{prop}

\ssec{Proof of \propref{p:inv dualizable}}

\sssec{}

Before we begin the proof, let us note that the contents of \secref{ss:proof in loop} up until
\propref{p:inv dualizable} were \emph{not} specific to the case of the loop group $\fL(G)$
for $G$ reductive. In fact, Propositions \ref{p:right adj to Hecke}, \ref{p:recover Vect} and 
\corref{c:inv via Hecke} remain valid for any pair $\CG^+\subset \CG$, where:

\begin{itemize}

\item $\CG$ is a placid group ind-scheme;

\item $\CG^+\subset \CG$ is a closed placid group-subscheme;

\item $\CG^+$ admits a homomorphism to a group-scheme of finite type with a pro-unipotent kernel.

\end{itemize}

By contrast, \propref{p:inv dualizable} (and with it, \thmref{t:coinvariants and limits, loop}) are specific to the
situation when $\CG=\fL(G)$ with $G$ reductive and $\CG^+=\fL^+(G)$. The key feature of this situation is that
the ind-scheme $\CG/\CG^+$ (which is of ind-finite type by assumption) is \emph{ind-proper}. 

\medskip

We will prove \propref{p:inv dualizable} in this slightly more general context. 

\sssec{}

Let $\BH^{\on{loc.fin}}\subset \BH$ be the full (but not cocomplete) category consisting of objects which get
sent to compact objects in $\Shv(\CG)$ under the forgetful functor
$$\BH\simeq \Shv(\CG)^{\CG^+\times \CG^+}\to \Shv(\CG).$$

Let
$$\BH^{\on{ren}}$$ be the ind-completion of $\BH^{\on{loc.fin}}$. Ind-extending the tautological embedding,
we obtain a functor
$$\Psi:\BH^{\on{ren}}\to \BH.$$

This functor admits a left adjoint, denoted $\Xi$, given by ind-extending the inclusion
$$\BH^c\subset \BH^{\on{loc.fin}}.$$

It is clear that the unit of the adjunction
$$\on{Id}\to \Psi\circ \Xi$$
is an isomorphism. Hence, $\Xi$ is fully faithful, and $\Psi$ is a localization.

\sssec{}

The assumption that $\CG/\CG^+$ is proper implies that the monoidal operation on $\BH$ preserves both
$\BH^{\on{loc.fin}}$. Ind-extending, we obtain that $\BH^{\on{ren}}$ acquires a monoidal structure, for which
the functor $\Psi$ is monoidal.

\medskip

Thus, we obtain that $\BH$ is a monoidal localization of $\BH^{\on{ren}}$. In particular, for a right (resp., left)
$\BH$-module category $\bC^r$ (resp., $\bC^l$), the functor
$$\bC^r\underset{\BH^{\on{ren}}}\otimes \bC^l\to \bC^r\underset{\BH}\otimes \bC^l$$
is an equivalence. 

\medskip

Hence, in order to prove \propref{p:inv dualizable}, it suffices to show that the dual of $\Vect^{\CG^+}$ 
considered as a left $\BH^{\on{ren}}$-module is $\Vect^{\CG^+}$ identifies canonically with $\Vect^{\CG^+}$ 
considered as a right $\BH^{\on{ren}}$-module.

\sssec{}

The key observation now is that the ind-properness assumption on $\CG/\CG^+$ implies that $\BH^{\on{ren}}$
is rigid. Indeed, this follows from \cite[Chapter 1, Lemma 9.1.5]{GR1}. Explicitly, the monoidal dual of an object
$$\CF\in \BH^{\on{loc.fin}}\simeq \Shv(\CG)^{\CG^+\times \CG^+}\simeq \Shv(\CG/\CG^+)^{\CG^+}$$
is 
$$\tau(\BD(\CF)),$$
where $\BD$ is Verdier duality on $\CG/\CG^+$ and $\tau$ is the involution on $\Shv(\CG)^{\CG^+\times \CG^+}$,
given by the inversion on $\CG$ (note that it swaps the two factors in $\CG^+\times \CG^+$). 

\sssec{}

Hence, the required self-duality of $\Vect^{\CG^+}$ follows from \cite[Chapter 1, Proposition 9.5.3]{GR1}.

\qed

\ssec{Proof of \propref{p:recover Vect}} \label{ss:recover Vect}

We will prove point (a). Point (b) is obtained by considering maps from both sides to $\Vect$. 

\medskip

In order to unburden the notation, we will write $\CG$ for $\fL(G)$ and $\CG^+$ for $\fL^+(G)$. 

\sssec{}

We need to show that the tautological functor
\begin{equation} \label{e:recover vect}
\Phi:\Shv(\CG/\CG^+)\underset{\BH}\otimes \Vect^{\CG^+}\to \Vect
\end{equation}
is an equivalence.

\medskip

First, the commutation of \eqref{e:plus inv} with colimits implies that the functor \eqref{e:recover vect}
induces an equivalence after taking $\CG^+$-invariants. Hence, by \thmref{t:reconstr}, we obtain 
that \eqref{e:recover vect} induces an equivalence on the full subcategories of both sides, on
which the action of $\Shv(\CG^+)$ factors through an action of $\Shv(\CG^+)^0$,
see Remark \ref{r:loc const act}. 

\medskip

In particular, we obtain that the functor \eqref{e:recover vect} admits a fully faithful left adjoint, to be denoted $\Psi$,  
(which is also a right inverse), compatible with the actions of $\CG^+$. A priori, the functor $\Psi$ is \emph{co-lax}
compatible with the action of $\CG$. I.e., for $\CF\in \Shv(\CG)$ we have a map
\begin{equation} \label{e:left lax}
\Psi(\CF\star \sfe)\to \CF\star \Psi(\CF).
\end{equation}

We claim, however, that this co-lax compatibility is strict, i.e., the maps \eqref{e:left lax} are isomorphisms. Let us assume 
that for a moment and finish the proof of \propref{p:recover Vect}. 

\sssec{}

To prove that the functors $\Psi$ and $\Phi$ are mutually inverse equivalences, it suffices to show that $\on{ker}(\Phi)=0$.
However, the embedding
$$\on{ker}(\Phi)\hookrightarrow \Shv(\CG/\CG^+)\underset{\BH}\otimes \Vect^{\CG^+}$$
admits a $\CG$-invariant left inverse given by
$$\on{coFib}(\Psi\circ \Phi\to \on{Id}).$$

Hence, it suffices to show that for any $\bC\in \CG\mmod$, a $\CG$-invariant functor
$$\Shv(\CG/\CG^+)\underset{\BH}\otimes \Vect^{\CG^+}\to \bC,$$
whose composition with $\Psi$ vanshes, is actually zero. 

\medskip

However, for any functor as above, the resulting functor
$$\Vect^{\CG^+}\simeq (\Shv(\CG/\CG^+)\underset{\BH}\otimes \Vect^{\CG^+})^{\CG^+}\to \bC^{\CG^+}$$
is zero. Hence, the original functor vanishes by adjunction. 

\sssec{}

We will now prove that the maps \eqref{e:left lax} are isomorphisms. This will be done 
in the following general framework, whose slogan is ``a functor lax-compatible with an action of a group is actually
strictly compatible".  First, we show that this principle literally applies when we work with $\Shv(-)=\Dmod(-)$. 

\begin{lem} \label{l:adj inv}
Let $F:\bC_1\to \bC_2$ be a functor between categories acted on by $\CG$. Suppose that $F$ is equipped with a structure
of lax/co-lax compatibility with the action of $\CG$. Then this compatibility is actually strict.
\end{lem}

\begin{proof}

We will consider the co-lax case; the lax case is similar. By assumption, we are given a natural transformation
$$
\xy
(0,0)*+{\Shv(\CG)\otimes \bC_1}="A";
(40,0)*+{\bC_1}="B";
(0,-20)*+{\Shv(\CG)\otimes  \bC_2}="C";
(40,-20)*+{\bC_2,}="D";
{\ar@{->}_{\on{Id}\otimes \Psi} "A";"C"};
{\ar@{->}^{\on{act}}"A";"B"};
{\ar@{->}^{\on{act}} "C";"D"};
{\ar@{->}^{\Psi} "B";"D"};
{\ar@{=>}^\alpha "B";"C"};
\endxy
$$
i.e., $$\alpha: \Psi\circ \on{act}\to \on{act}\circ (\on{Id}_{\Shv(\CG)}\otimes \Psi).$$

We will explicitly construct an inverse of this natural transformation.   

\medskip

Let $i:\CG\to \CG\times \CG\times \CG$ be the map
$$g\mapsto (g,g^{-1},g).$$
Using the identification
\begin{equation} \label{e:triple}
\Shv(\CG)\otimes \Shv(\CG)\otimes \Shv(\CG)\simeq \Shv(\CG\times \CG\times \CG),
\end{equation} 
we obtain a functor
\begin{multline}  \label{e:triple comp1}
\Shv(\CG)\otimes \bC_1 \overset{i_*\otimes \on{Id}_{\bC_1}}\longrightarrow \Shv(\CG\times \CG\times \CG)\otimes \bC_1
\overset{\on{id}_\CG\otimes \on{id}_\CG \otimes \on{act}} \longrightarrow \\
\to \Shv(\CG\times \CG)\otimes \bC_1
\overset{\on{id}_\CG\otimes\on{act}} \longrightarrow \Shv(\CG)\otimes \bC_1 \overset{\on{Id}_{\Shv(\CG)}\otimes \Psi}
 \longrightarrow \Shv(\CG)\otimes \bC_2 \overset{\on{act}}\longrightarrow \bC_2.
 \end{multline} 

On the one hand, we claim the composition \eqref{e:triple comp1} identifies with
$$\on{act}\circ (\on{Id}_{\Shv(\CG)}\otimes \Psi).$$

Indeed, we claim that the composition of the first three arrows in  \eqref{e:triple comp1} , i.e., 
\begin{equation} \label{e:simple comp}
\Shv(\CG)\otimes \bC_1 \overset{i_*\otimes \on{Id}_{\bC_1}}\longrightarrow \Shv(\CG\times \CG\times \CG)\otimes \bC_1
\overset{\on{id}_\CG\otimes \on{id}_\CG \otimes \on{act}} \longrightarrow \Shv(\CG\times \CG)\otimes \bC_1
\overset{\on{id}_\CG\otimes \on{act}} \longrightarrow \Shv(\CG)\otimes \bC_1
\end{equation}
is the identity functor. Indeed, we rewrite \eqref{e:simple comp} as
$$
\Shv(\CG)\otimes \bC_1 \overset{i_*\otimes \on{Id}_{\bC_1}}\longrightarrow \Shv(\CG\times \CG\times \CG)\otimes \bC_1
\overset{\on{id}_\CG\otimes \on{mult}\otimes \on{Id}_{\bC_1}}\longrightarrow  \Shv(\CG\times \CG)\otimes \bC_1 
\overset{\on{id}_\CG\otimes \on{act}} \longrightarrow \Shv(\CG)\otimes \bC_1,$$
which is the same as
$$\Shv(\CG)\otimes \bC_1 \overset{\on{id}_\CG\otimes \on{unit}\otimes \on{Id}_{\bC_1}}
\longrightarrow  \Shv(\CG\times \CG)\otimes \bC_1 
\overset{\on{id}_\CG\otimes \on{act}} \longrightarrow \Shv(\CG)\otimes \bC_1,$$
and the latter is indeed the identity functor. 

\medskip

On the other hand, the natural transformation $\alpha$ defines a map from \eqref{e:triple comp1} to
\begin{multline}  \label{e:triple comp2}
\Shv(\CG)\otimes \bC_1 \overset{i_*\otimes \on{Id}_{\bC_1}}\longrightarrow \Shv(\CG\times\CG\times \CG)\otimes \bC_1
\overset{\on{id}_\CG\otimes \on{id}_\CG \otimes \on{act}} \longrightarrow \\
\to \Shv(\CG\times \CG)\otimes \bC_1 \overset{\on{id}_\CG\otimes \on{id}_\CG\otimes \Psi}\longrightarrow 
\Shv(\CG\times \CG)\otimes \bC_2 \overset{\on{id}_\CG\otimes \on{act}} \longrightarrow 
 \Shv(\CG)\otimes \bC_2 \overset{\on{act}}\longrightarrow \bC_2.
\end{multline}  

Now, we claim that the composition \eqref{e:triple comp2} identifies with $\Psi\circ \on{act}$. To prove this, we first rewrite 
\eqref{e:triple comp2} as 
\begin{multline}   \label{e:triple comp3}
\Shv(\CG)\otimes \bC_1 \overset{i_*\otimes \on{Id}_{\bC_1}}\longrightarrow \Shv(\CG\times\CG\times \CG)\otimes \bC_1
\overset{\on{id}_\CG\otimes \on{id}_\CG \otimes \on{act}} \longrightarrow \\
\to \Shv(\CG\times \CG)\otimes \bC_1 \overset{\on{id}_\CG\otimes \on{id}_\CG\otimes \Psi}\longrightarrow 
\Shv(\CG\times \CG)\otimes \bC_2 \overset{\on{mult}\otimes \on{id}_{\bC_2}} \longrightarrow 
 \Shv(\CG)\otimes \bC_2 \overset{\on{act}}\longrightarrow \bC_2.
\end{multline}  

Next we note that we have a commutative square
\begin{equation} \label{e:crucial square}
\CD
\Shv(\CG\times \CG)\otimes \bC_1  @>{\on{id}_\CG\otimes \on{id}_\CG\otimes \Psi}>>  \Shv(\CG\times \CG)\otimes \bC_2  \\
@V{\on{mult}\otimes \on{id}_{\bC_1}}VV  @VV{\on{mult}\otimes \on{id}_{\bC_2}}V  \\
\Shv(\CG)\otimes \bC_1 @>{\on{id}_\CG\otimes \Psi}>>   \Shv(\CG)\otimes \bC_2.
\endCD
\end{equation}

Hence, we can rewrite \eqref{e:triple comp3} as
\begin{multline}   \label{e:triple comp4}
\Shv(\CG)\otimes \bC_1 \overset{i_*\otimes \on{Id}_{\bC_1}}\longrightarrow \Shv(\CG\times\CG\times \CG)\otimes \bC_1
\overset{\on{id}_\CG\otimes \on{id}_\CG \otimes \on{act}} \longrightarrow \\
\to \Shv(\CG\times \CG)\otimes \bC_1 \overset{{\on{mult}\otimes \on{id}_{\bC_1}}}\longrightarrow \Shv(\CG)\otimes \bC_1 
\overset{\on{id}_\CG\otimes \Psi}\longrightarrow \Shv(\CG)\otimes \bC_2 \overset{\on{act}}\longrightarrow \bC_2.
\end{multline} 

Now, the composition of the first three arrows in \eqref{e:triple comp4}, i.e.,
$$\Shv(\CG)\otimes \bC_1 \overset{i_*\otimes \on{Id}_{\bC_1}}\longrightarrow \Shv(\CG\times\CG\times \CG)\otimes \bC_1
\overset{\on{id}_\CG\otimes \on{id}_\CG \otimes \on{act}} \longrightarrow 
\Shv(\CG\times \CG)\otimes \bC_1 \overset{{\on{mult}\otimes \on{id}_{\bC_1}}}\longrightarrow \Shv(\CG)\otimes \bC_1$$
is the functor
$$\Shv(\CG)\otimes \bC_1 \overset{\on{act}}\longrightarrow \bC_1 \overset{\on{unit}\otimes \on{Id}_{\bC_1}} \longrightarrow \Shv(\CG)\otimes \bC_1.$$

Hence, the composition in \eqref{e:triple comp4} is indeed isomorphic to 
$$\Shv(\CG)\otimes \bC_1 \overset{\on{act}}\longrightarrow \bC_1 \overset{\Psi}\otimes \bC_2,$$
as claimed. 

\medskip

By unwinding the construction, one checks that the natural transformation
$$\on{act}\circ (\on{Id}_{\Shv(\CG)}\otimes \Psi)\to \Psi\circ \on{act}$$
constructed above is indeed the inverse of $\alpha$. 

\end{proof}

\sssec{} \label{sss:Psi Y}

We will now adapt this argument in order to deduce the fact that the maps \eqref{e:left lax} are isomorphisms. For a scheme $\CY$ 
consider the action functor 
\begin{equation} \label{e:action Y}
\Phi_\CY:\Shv(\CY \times \CG/\CG^+)\underset{\BH}\otimes \Vect^{\CG^+}\to \Shv(\CY). 
\end{equation}

Let us denote by $\Psi_\CY$ the functor 
\begin{equation} \label{e:Psi Y}
\Shv(\CY) \overset{\on{Id}_{\Shv(\CY)}\otimes \Psi}\longrightarrow 
\Shv(\CY)\otimes \Shv(\CG/\CG^+)\underset{\BH}\otimes \Vect^{\CG^+} \to 
\Shv(\CY\times \CG/\CG^+)\underset{\BH}\otimes \Vect^{\CG^+}.
\end{equation}

We claim that $(\Psi_\CY,\Phi_\CY)$ form an adjoint pair.  Indeed, let us denote by $i$ the external tensor product functor
$$\Shv(\CY)\boxtimes \Shv(\CG/\CG^+)\to \Shv(\CY\times \CG/\CG^+).$$

The functor $i$ preserves compactness; hence it admits a continuous right adjoint, to be denoted $i^R$. We have
a tautological isomorphism
$$\on{Id}_{\Shv(\CY)}\otimes \Phi \simeq \Phi_\CY\otimes (i\underset{\BH}\otimes \on{Id}_{\Vect^{\CG^+}}).$$

From here we obtain a natural transformation
\begin{equation} \label{e:Phi Y}
(\on{Id}_{\Shv(\CY)}\otimes \Phi) \circ (i^R\underset{\BH}\otimes \on{Id}_{\Vect^{\CG^+}})\to \Phi_\CY.
\end{equation}

We have to show that \eqref{e:Phi Y} is an isomorphism.

\sssec{}

To prove this, it suffices to show that the corresponding natural transformation becomes an isomorphism after precomposition with 
\begin{equation} \label{e:Phi Y prec}
\Shv(\CY\times \CG/\CG^+) \overset{\on{Id}_{\Shv(\CY\times \CG/\CG^+) }\otimes \sfe}\longrightarrow 
\Shv(\CY\times \CG/\CG^+) \otimes \Vect^{\CG^+}\to \Shv(\CY\times \CG/\CG^+) \underset{\BH}\otimes \Vect^{\CG^+}.
\end{equation}

The precomposition of the LHS of \eqref{e:Phi Y} with \eqref{e:Phi Y prec} is the functor 
\begin{equation} \label{e:Phi Y prec 1}
\Shv(\CY\times \CG/\CG^+) \overset{(p_\CY)_*}\longrightarrow \Shv(\CY),
\end{equation}
where $p_\CY$ denotes the projection $\CY\times \CG/\CG^+\to \CY$.

\medskip

The precomposition of the LHS of \eqref{e:Phi Y} with \eqref{e:Phi Y prec} is the functor
\begin{equation} \label{e:Phi Y prec 2}
\Shv(\CY\times \CG/\CG^+)  \overset{i^R}\to \Shv(\CY)\boxtimes \Shv(\CG/\CG^+) \overset{\on{Id}\otimes \on{C}^\cdot(\CG/\CG^+,-)}
\longrightarrow \Shv(\CY).
\end{equation}

\sssec{}

We claim that the latter isomorphism takes place for $\CG/\CG^+$ replaced by any 
ind-scheme of ind-finite type
$$Z=\underset{\alpha}{\on{colim}}\, Z_\alpha.$$

\medskip

Indeed, for $\CF'\in \Shv(\CY\times Z)$ and $\CF''\in \Shv(\CY)^c$, we have:
$$\CHom_{\Shv(\CY)}(\CF'',(p_\CY)_*(\CF'))\simeq 
\underset{\alpha}{\on{colim}}\, \CHom_{\Shv(\CY\times Z)}(\CF''\boxtimes \sfe_{Z_\alpha},\CF'),$$
and
\begin{multline*}
\CHom_{\Shv(\CY)}\left(\CF'',(\on{Id}\otimes \on{C}^\cdot(\CG/\CG^+,-))\circ i^R(\CF')\right)\simeq
\underset{\alpha}{\on{colim}}\,  \CHom_{\Shv(\CY)\boxtimes \Shv(Z)}(\CF''\otimes \sfe_{Z_\alpha},i^R(\CF'))\simeq \\
\simeq \underset{\alpha}{\on{colim}}\, \CHom_{\Shv(\CY\times Z)}(\CF''\boxtimes \sfe_{Z_\alpha},\CF'),
\end{multline*}
establishing the desired isomorphism. 

\sssec{}

We finally return to the proof of \propref{p:recover Vect}.  Consider the simplicial categories
$$\bC^\bullet_1:=\Shv(\CG^\bullet) \text{ and }  \bC^\bullet_2:=\Shv(\CG^\bullet\times \CG/\CG^+)\underset{\BH}\otimes \Vect^{\CG^+}.$$
We have a naturally defined simplicial functor $\Phi^\bullet:\bC_2^\bullet\to \bC_1^\bullet$. By \secref{sss:Psi Y}, the functor $\Phi^\bullet$
admits a term-wise left adjoint, to be denoted $\Psi^\bullet$.

\medskip

The argument proving \lemref{l:adj inv} shows that the natural transformation
$$
\xy
(0,0)*+{\Shv(\CG)}="A";
(60,0)*+{\Vect}="B";
(0,-20)*+{\Shv(\CG\times \CG/\CG^+)\underset{\BH}\otimes \Vect^{\CG^+}}="C";
(60,-20)*+{\Shv(\CG/\CG^+)\underset{\BH}\otimes \Vect^{\CG^+}}="D";
{\ar@{->}_{\Psi^1} "A";"C"};
{\ar@{->}^{\on{act}}"A";"B"};
{\ar@{->}^{\on{act}} "C";"D"};
{\ar@{->}^{\Psi^0} "B";"D"};
{\ar@{=>}^\alpha "B";"C"};
\endxy
$$
is an isomorphism. Indeed, the only non-formal part of the argument was the commutation of the square \eqref{e:crucial square},
which commutes in our case due to the shape of $\Psi^\bullet$ established in \secref{sss:Psi Y}. 

\medskip

In addition, due to the shape of $\Psi^\bullet$, we obtain a commutative diagram of functors
\begin{equation} \label{e:BC} 
\CD
\Psi\circ \on{act}  @>>> \on{act}\circ (\on{Id}_{\Shv(\CG)}\otimes \Psi) \\
@V{\on{id}}VV  @VV{\sim}V  \\
\Psi\circ \on{act}  @>>>  \on{act}\circ \Psi^1
\endCD
\end{equation}

Knowing that the bottom horizontal arrow in \eqref{e:BC} is an isomorphism, we conclude that the natural
transformation
$$\Psi\circ \on{act}  \to \on{act}\circ (\on{Id}_{\Shv(\CG)}\otimes \Psi)$$
is an isomorphism, as desired.

\end{document}